\documentclass[final]{siamart171218}
\usepackage{hyperref}
\usepackage{epsfig}
\usepackage{amsfonts}
\usepackage{algorithm}
\usepackage{algorithmic}
\usepackage{graphicx}
\usepackage{subfigure}
\usepackage{array}
\usepackage{enumerate}
\usepackage{enumitem}
\usepackage{rotating}
\usepackage[T1]{fontenc}

\usepackage{pgfplots,tikz}

\newcommand\Item[1][]{%
  \ifx\relax#1\relax  \item \else \item[#1] \fi
  \abovedisplayskip=0pt\abovedisplayshortskip=0pt~\vspace*{-1.25\baselineskip}}

\newcommand{\sr}{\mathsf{r}}

\newcommand{\red}{\mathsf{red}}

\newtheorem{assumption}[theorem]{Assumption}

\xdefinecolor{lavendar}{rgb}{0.4,0.2,0.6}
\xdefinecolor{lavender}{rgb}{0.92,0.9,0.97}
\xdefinecolor{paleblue}{rgb}{0.3,0.3,0.7}
\xdefinecolor{darkblue}{rgb}{0.15,0.15,0.9}
\xdefinecolor{mygreen}{rgb}{0.1,0.6,0.1}
\xdefinecolor{mydarkgreen}{rgb}{0.1,0.4,0.1}

\usepackage{todonotes}

\headers{E. Mengi}{Finding Locally Optimal Solutions in ${\mathcal L}_\infty$ Model Reduction}

\title{A Subspace Framework for ${\mathcal L}_\infty$ Model Reduction}

\author{
Emre Mengi\thanks{Ko\c{c} University, Department of Mathematics, Rumeli Feneri Yolu 34450, Sar{\i}yer, Istanbul, Turkey, E-Mail: \texttt{emengi@ku.edu.tr}.} 
}

\begin{document}
\maketitle

\begin{abstract}
We consider the problem of locating a nearest descriptor system of prescribed reduced order
to a descriptor system with large order with respect to the ${\mathcal L}_\infty$ norm.
Widely employed approaches such as the balanced truncation and best Hankel norm approximation
for this ${\mathcal L}_\infty$ model reduction problem are usually expensive and  yield solutions
that are not optimal, not even locally. We propose approaches based on the minimization
of the ${\mathcal L}_\infty$ objective by means of smooth optimization techniques. As we illustrate,
direct applications of smooth optimization techniques are not feasible, 
since the optimization techniques converge at best at a linear rate requiring too many evaluations
of the costly ${\mathcal L}_\infty$-norm objective to be practical. We replace the original large-scale 
system with a system of smaller order that interpolates the original system at points on the imaginary axis,
and minimize the ${\mathcal L}_\infty$ objective after this replacement. The smaller system 
is refined by interpolating at additional imaginary points determined based on the local minimizer of the 
${\mathcal L}_\infty$ objective, and the optimization is repeated. We argue the framework converges
at a quadratic rate under smoothness and nondegeneracy assumptions, and describe how asymptotic
stability constraints on the reduced system sought can be incorporated into our approach. The numerical
experiments on benchmark examples illustrate that the approach leads to locally optimal solutions to 
the ${\mathcal L}_\infty$ model reduction problem, and the convergence occurs quickly for descriptors 
systems of order a few ten thousands.
\end{abstract}



\section{Introduction}
Various applications give rise to descriptor systems with large order. A model order reduction
technique typically aims at approximating the large order system with a system of much smaller
and prescribed order. There are several powerful numerical approaches for the model order 
reduction of descriptor systems at the moment. However, to our knowledge, there does not exist 
a work that addresses the determination of optimal reduced order systems with respect to the 
${\mathcal L}_\infty$ norm. Even finding a locally optimal solution for the ${\mathcal L}_\infty$-norm 
model reduction problem is not addressed thoroughly. The ${\mathcal H}_\infty$-norm model
reduction problem is closely related with the system at hand asymptotically stable,
and the reduced order system sought required to be asymptotically stable.

A descriptor system is often available in the state-space representation of the form
\begin{equation}\label{eq:state_space}
		E x'(t)	\;	 =	\;	A x(t)		+	B u(t),	\quad\quad
		y(t)	\;	 =	\;	C x(t)	+	D u(t),
\end{equation}
for given matrices $E, A \in {\mathbb R}^{n\times n}$, $B \in {\mathbb R}^{n\times m}$, 
$C \in {\mathbb R}^{p\times n}$, $D \in {\mathbb R}^{p\times m}$. 
The ${\mathcal L}_\infty$ norm of the transfer function
\begin{equation}\label{eq:full_transfer}
	H(s)	\;	=	\;	C(s E - A)^{-1} B		+	D
\end{equation}
of the system in (\ref{eq:state_space}) is defined as
\[
	\| H \|_{{\mathcal L}_\infty}	\;	:=		\;		\sup_{\omega \in {\mathbb R}} \: \sigma_{\max} ( H({\rm i} \omega)	)
				\;	=	\;	\sup_{\omega \in {\mathbb R}, \omega \geq 0} \: \sigma_{\max} ( H({\rm i} \omega)	),
\]
where $\sigma_{\max}(\cdot)$ denotes the largest singular value of its matrix argument,
and the last equality holds as $A, B, C, D, E$ are real matrices. 
Note that we customarily set 
$\, \| H \|_{{\mathcal L}_\infty}	=	\sup_{\omega \in {\mathbb R}} \: \sigma_{\max} ( H({\rm i} \omega)	)	=	\infty \,$
if $H$ has a pole on the imaginary axis, or the norm of its restriction to the imaginary axis is not bounded.  
If the descriptor system is asymptotically stable,  then its ${\mathcal L}_\infty$ norm
reduces to the ${\mathcal H}_\infty$-norm. Formally, let us denote with $L^k_2$ 
the space of functions $f : {\mathbb R} \rightarrow {\mathbb R}^k$ satisfying 
$\| f \|_{L^k_2} := \sqrt{\int_{-\infty}^{\infty} \| f(t) \|^2_2 \: {\rm d} t} < \infty$ with $k = m$ or $k = p$.
For simplicity, we omit the dependence of the space on the dimension $k$,
and write $L_2$ as well as $\| f \|_{L_2}$, as which space ($L_2^m$ or $L_2^p$) 
is referred to will be clear from the context. Moreover, suppose the system in
(\ref{eq:state_space}) is asymptotically stable with poles in the open left half of the
complex plane. 
Then the ${\mathcal L}_{\infty}$ norm of $H$ is the same as the ${\mathcal H}_\infty$
norm of $H$ defined as
\[
       \| H \|_{{\mathcal H}_\infty}	\,	:=		\,		
                    \sup_{s \in {\mathbb C}^+} \: \sigma_{\max} ( H(s)	) ,
\]
which in turn is equal to induced norm of the operator $\phi : L_2 \rightarrow L_2$
associated with (\ref{eq:state_space})  in the time domain that maps $u$ to $y$
defined as
\[
      \| \phi \|_{L_2}
           \;   :=   \;
	   \max  \left\{  \| \phi \, u \|_{L_2} 
	                                      \; | \;  u \in L_2 \text{ s.t. } \| u \|_{L_2} = 1
	                          \right\}   \, .
\]
Hence, under the asymptotic stability assumption on the descriptor system in (\ref{eq:state_space}),
we have
$
        \| H \|_{{\mathcal L}_\infty}	    = 
              \| H \|_{{\mathcal H}_\infty}	
					=	
		\| \phi \|_{L_2}$.				

The ${\mathcal L}_\infty$-norm model order reduction problem -- or the ${\mathcal L}_\infty$ 
model reduction problem in short -- for a given descriptor system
of order $n$ and for a prescribed positive integer 
$\sr < n$ concerns finding a reduced descriptor system of order $\sr$
that is closest to the given system of order $n$ with respect to the ${\mathcal L}_\infty$ norm.
Formally, let
$S^{\red} = (A^{\red}, E^{\red}, B^{\red}, C^{\red}, D^{\red})$
denote a system of order $\sr$ with the state-space representation
\begin{equation}\label{eq:state_space_red}
		E^{\red} \, x'(t)	\;	 =	\;	A^{\red} \, x(t)		+	B^{\red} \, u(t),	\quad\quad
		y(t)	\;	 =	\;	C^{\red} \, x(t)	+	D^{\red} \, u(t)
\end{equation}
described by the matrices $E^{\red}, A^{\red} \in {\mathbb R}^{\sr \times \sr}$, 
$B^{\red} \in {\mathbb R}^{\sr \times m}$,  $C^{\red} \in {\mathbb R}^{p\times \sr}$, 
$D^{\red} \in {\mathbb R}^{p\times m}$, and with the transfer function
\begin{equation}\label{eq:red_trans}
	H(s; S^{\red})	\;	=	\;	C^{\red}(s E^{\red} - A^{\red})^{-1} B^{\red}		+	D^{\red}.
\end{equation}
Furthermore, let $S = (A, E, B, C, D)$ be the given system of order $n$
and with the transfer function as in (\ref{eq:full_transfer}).
The ${\mathcal L}_\infty$ model reduction problem involves finding 
a descriptor system $S^{\red}_\star$ of order $\sr$ that minimizes the objective
\begin{equation}\label{eq:objective}
\begin{split}
	\| H -	H( \: \cdot \: ; S^{\red})	\|_{{\mathcal L}_\infty}	\;\;	&	=	\;\;\;\;\,
	\sup_{\omega \in {\mathbb R}} \:\: 
	\left[
		\sigma(\omega; S^{\red}) \:  :=  \: \sigma_{\max} (   H({\rm i} \omega)	-	
															H({\rm i} \omega ; S^{\red} )   )	
	\right]	\\
	&	=	\;\;	\sup_{\omega \in {\mathbb R}, \omega \geq 0} \:\: \sigma(\omega; S^{\red}) 
\end{split}
\end{equation}
over all descriptor systems $S^{\red}$ of order $\sr$.  The problem at hand, in particular
the objective in (\ref{eq:objective}), is non-convex, and here we aim to determine a local 
minimizer of the objective in (\ref{eq:objective}) numerically. The quality of the determined
local minimizer also matters, however this issue is largely dependent upon with which reduced system
of order $\sr$ our approach is initialized.

Two important remarks are in order regarding the minimization of the objective in (\ref{eq:objective}).
First, in addition to non-convexity, an additional difficulty is the nonsmooth nature of the problem. 
The objective in (\ref{eq:objective})  as a function of $S^{\red}$  is typically not differentiable 
when $\sigma(\omega; S^{\red})$ has multiple global maximizers over $\omega \geq 0$.
Secondly, under asymptotic stability assumptions on the original system and the reduced system,
the error $\| H -	H( \: \cdot \: ; S^{\red}) \|_{{\mathcal L}_\infty}$ gives a uniform upper bound
on how much the outputs of the original and reduced systems can differ. To be precise, suppose
that the system $S$ of order $n$, and the reduced system $S^\red$ of order $\sr$ are 
asymptotically stable. Furthermore, let us denote with $\phi$ and $\phi_{\sr}$ the operators 
in the time domain corresponding to the systems in (\ref{eq:state_space}) and (\ref{eq:state_space_red}),
respectively. For every $u \in L_2$, we have
\[
      \| y - y_{\sr} \|_{L_2} \;  \leq \;  \| H -	H( \: \cdot \: ; S^{\red})	\|_{{\mathcal L}_\infty}  \| u \|_{L_2}
\]
where $y, y_{\sr}$ are such that $y = \phi \, u$ and $y_{\sr} = \phi_{\sr} \, u$. This means that
if a small error $\| H -	H( \: \cdot \: ; S^{\red}) \|_{{\mathcal L}_\infty}$ can be ensured by the
minimization of (\ref{eq:objective}), then the output $y_{\sr} = \phi_{\sr} \, u$ of the minimizing
reduced system approximates the original output $y = \phi \, u$ well uniformly over every input 
$u$ of prescribed norm.

\subsection{Literature and Contributions}
For an asymptotically stable descriptor system with the transfer function $H$, the $H_\infty$
model reduction problem - that is, for a given small order $\sr$, finding an asymptotically stable 
system $S^\red$ of order $\sr$ such that $\| H - H( \: \cdot \: ; S^{\red}) \|_{{\mathcal H}_\infty}$ 
is small - has been under consideration for long time. One of the classical approaches for the
${\mathcal H}_\infty$ model reduction problem is the balanced truncation, which determines
a state transformation so that the observability and controllability grammians are the same diagonal 
matrix, and truncates the system matrices after applying this state transformation
\cite{DulP2010, Ant2005, Moo1981, MulR1976}.
The reduced system by the balanced truncation is typically not even a local minimizer of 
$\| H - H( \: \cdot \: ; S^{\red}) \|_{{\mathcal H}_\infty}$ over $S^{\red}$, though it usually is
a good quality approximation of $H$ with respect to the ${\mathcal H}_\infty$ norm \cite{GugA2004}.
The major difficulty with the balanced truncation that limits its applicability to larger systems
is that it requires the solution of two Lyapunov equations involving matrices of size equal
to the order of the system. With the iterative approaches for the solution of the Lyapunov equations,
such as the ADI method \cite{GugSA2003, Sty2004, Pen1999}, the balanced truncation is 
applicable to systems with higher order, but still Lyapunov equations stand as a hurdle.

A classical alternative is finding a best approximation with respect to the Hankel norm (HNA)
\cite{Glo1984} rather than the ${\mathcal H}_\infty$ norm.
Approaches to compute a globally optimal solution to HNA in polynomial time are proposed \cite{Glo1984}.
However, the globally optimal solution to HNA
is again usually not even a local minimizer of $\| H - H( \: \cdot \: ; S^{\red}) \|_{{\mathcal H}_\infty}$.
Furthermore, finding a globally optimal solution to HNA is even more costlier than the balanced
truncation. Even with the efficient use of computational linear algebra tools \cite{BenQQ2004},
solving HNA for systems with high order is out of reach.

Here, we propose an approach to compute a local minimizer of
$\| H - H( \: \cdot \: ; S^{\red}) \|_{{\mathcal L}_\infty}$ over all systems $S^{\red}$
of order $\sr$. To our knowledge, the approach is the first attempt 
to find such a locally optimal solution. The approach uses smooth optimization
techniques, which have been employed for solving
nonsmooth optimization problems \cite{LewO2013, AslO2020} in the last fifteen years.
Most often, they seem to be capable of locating locally optimal solutions,
but slowly at best at a linear rate. Consequently, as we shall see below, a direct application
of them to minimize $\| H - H( \: \cdot \: ; S^{\red}) \|_{{\mathcal L}_\infty}$
is prohibitively expensive even for systems of medium order, as it requires
the computation of the objective, that is the ${\mathcal L}_\infty$ norm, too
many times. Instead, we replace $H$ with an approximation $\widetilde{H}$
of small order greater than $\sr$. Rather than $\| H - H( \: \cdot \: ; S^{\red}) \|_{{\mathcal L}_\infty}$,
we minimize $\| \widetilde{H} - H( \: \cdot \: ; S^{\red}) \|_{{\mathcal L}_\infty}$,
then update $\widetilde{H}$ based on the minimizer, and repeat. The approximation
$\widetilde{H}$ is built using the Petrov-Galerkin framework, and an update
involves the expansion of the projection subspaces for the Petrov-Galerkin framework. 
We show that the proposed framework converges quadratically under simplicity
assumptions. We also describe how the asymptotic stability constraints
can be imposed on the variable $S^{\red}$ when minimizing $\| H - H( \: \cdot \: ; S^{\red}) \|_{{\mathcal L}_\infty}$
in case the original system in (\ref{eq:state_space}) is asymptotically stable.
As the system corresponding to $H - H( \: \cdot \: ; S^{\red})$
is asymptotically stable when $S$ and $S^\red$ is asymptotically stable,
the incorporation of this constraint into our approach leads to a 
local minimization of $\| H - H( \: \cdot \: ; S^{\red}) \|_{{\mathcal H}_\infty}$
over all asymptotically stable systems $S^\red$ of order $\sr$, i.e., locally
optimal solution of the ${\mathcal H}_\infty$ model reduction problem.

Iterative rational Krylov algorithm (IRKA) \cite{GugAB2008} is proposed to find a reduced 
order system of prescribed order that is locally optimal with respect to the ${\mathcal H}_2$
norm defined as 
$\| H \|_{{\mathcal H}_2} = \sqrt{\frac{1}{2\pi} \int_{-\infty}^{\infty} 
				\text{trace}( H({\rm i} \omega)^\ast H({\rm i} \omega) ) \: {\rm d} \omega}$
for a system with the transfer function $H$. Formally, IRKA is an iterative interpolatory
approach that finds a local minimizer of $\| H - H( \: \cdot \: ; S^{\red}) \|_{{\mathcal H}_2}$ 
over all systems $S^{\red}$ of order $\sr$. In \cite{FlaBG2013}, starting from the reduced
order system of order $\sr$ generated by IRKA, an optimization based approach is 
proposed to find a locally optimal solution of $\| H - H( \: \cdot \: ; S^{\red}) \|_{{\mathcal H}_\infty}$
for single-input-single-output (SISO) systems but with respect to particular rank-one modifications
$\Delta_A = \varepsilon {\mathbf e}{\mathbf e}^T$, $\Delta_B = -\varepsilon {\mathbf e}$, 
$\Delta_C = -\varepsilon {\mathbf e}^T$, $\Delta_D = \varepsilon$ of the system matrices 
$A^\red$, $B^\red$, $C^\red$, $D^\red$ generated by IRKA over the optimization parameter $\varepsilon$. 
In the reported results in \cite{FlaBG2013}, this optimization improves the accuracy of the reduced system 
returned by IRKA by a factor of 2-4 with respect to the ${{\mathcal H}_\infty}$ norm. But again the 
eventual system is usually not a local minimizer of the objective 
$\| H - H( \: \cdot \: ; S^{\red}) \|_{{\mathcal H}_\infty}$ over systems $S^{\red}$ of order $\sr$. 				 

On a related note, our recent work \cite{AliBMV2020} concerns the minimization of the
${\mathcal H}_\infty$ norm of a descriptor system with large order dependent on parameters.
At the center of that work is a subspace framework to cope with the large order of the system.
It may seem plausible to look at the current work from that perspective. However, we have too
many optimization parameters here. As a result applying the framework over there to attain quick
convergence in the setting here is not feasible, as doing so yields projection subspaces growing rapidly 
(i.e., see Algorithm 2 in \cite{AliBMV2020} to attain superlinear convergence). In the
framework here, only $4m$ new directions, independent of $\sr$, are added into the subspaces
at every iteration. Moreover, we observe quick convergence, so the subspaces remain small throughout.

\subsection{Outline}
We first consider the direct minimization of $\| H - H( \: \cdot \: ; S^{\red}) \|_{{\mathcal L}_\infty}$ 
over systems $S^{\red}$ of order $\sr$ by means of smooth optimization techniques 
in Section \ref{sec:direct_optimize}.
In this section, we indicate the optimization variables, and spell out expressions for the first derivatives
of the objective with respect to these variables. As we shall see, the direct optimization is 
too costly even for systems with moderate order, since smooth optimization techniques converge
very slowly and require the evaluation of the ${\mathcal L}_\infty$ objective
too many times. Consequently, in Section \ref{sec:sub_fr}, we replace the transfer
function $H$ with an approximating transfer function $\widetilde{H}$ of small order greater than $\sr$
that Hermite interpolates $H$ at several points on the imaginary axis. Then we
minimize $\| \widetilde{H} - H( \: \cdot \: ; S^{\red}) \|_{{\mathcal L}_\infty}$ (by smooth optimization
techniques), and 
refine $\widetilde{H}$ so that Hermite interpolation with $H$ at another point on the imaginary axis
is attained based on the computed minimizer of $\| \widetilde{H} - H( \: \cdot \: ; S^{\red}) \|_{{\mathcal L}_\infty}$.
We introduce a refinement step on $\widetilde{H}$ so that interpolation
properties can be attained between the full objective 
$\| H - H( \: \cdot \: ; S^{\red}) \|_{{\mathcal L}_\infty}$ and
the reduced objective $\| \widetilde{H} - H( \: \cdot \: ; S^{\red}) \|_{{\mathcal L}_\infty}$.
Then the procedure is repeated with the refined $\widetilde{H}$. In Section \ref{sec:int_prop},
we investigate the interpolation properties between full objective and the reduced objective.
Based on these interpolation properties, we argue in Section \ref{sec:quad_conv}
that the algorithm converges at a quadratic rate under smoothness and nondegeneracy
assumptions. If the original descriptor system is asymptotically stable, it may be natural
to minimize 
$\| \widetilde{H} - H( \: \cdot \: ; S^{\red}) \|_{{\mathcal L}_\infty}$
subject to the asymptotic stability constraints on the reduced system $S^\red$.
We discuss in Section \ref{sec:imp_stab} the incorporation of such asymptotic
stability constraints on the reduced system into our approach. Section \ref{sec:prac_issues}
is devoted to the details that need to be taken into account in a practical implementation
of the proposed algorithm such as the initialization of the smooth optimization routines,
and termination. A Matlab implementation of the algorithm is publicly available.
In Section \ref{sec:numerical_exp}, we report numerical results obtained with
this implementation. The numerical results indicate quick convergence to a locally
optimal solution, and the capability to deal with systems with large order on the 
order of ten thousands.


\section{Use of First and Second Order Derivative-Based Methods}\label{sec:direct_optimize}
First order methods such as the gradient descent algorithm, and second order methods such as quasi-Newton
algorithms equipped with proper line-searches have been successfully applied to nonsmooth optimization problems
in recent years. Here, if a curvature condition is employed in the line-search, this should take into consideration
the fact that the directional derivatives do not have to converge to zero unlike the situation for smooth optimization
problems, e.g., if Wolfe conditions are imposed in the line-search, weak Wolfe conditions should be used rather
than strong Wolfe conditions. Also, for termination small gradient norms should not be required. Instead, for instance,
a failure in sufficient decrease in the objective along the descent search direction may indicate convergence to a locally 
optimal solution.  

\smallskip

The objective to be minimized in (\ref{eq:objective}) for the ${\mathcal L}_\infty$-norm model reduction 
problem can be expressed as
\begin{equation}\label{eq:objective2}
	\begin{split}
	{\mathcal F}(S^{\red})
				\;	=	\;\:
\sup_{\omega \geq 0}	\:		\sigma_{\max}	\left(	H({\rm i} \omega)	-	H({\rm i} \omega; S^{\red}) 		\right)
			\hskip 32.7ex		\\				
				\;	=	\;\:
\sup_{\omega \geq 0}	\:		\sigma_{\max}	\left( {\mathcal H}({\rm i} \omega ; S^{\red})  \right)	\;	
				\;	=	\;\;
	\| {\mathcal H}(\: \cdot \: ; S^{\red})	\|_{{\mathcal L}_\infty}	\; ,		\hskip 19ex		\\
	{\mathcal H}( s ; S^{\red})
				:=	
			[
					C	\;\;	-C^{\red}
			]
			\left[
				\begin{array}{cc}
					s E  -  A		&		0		\\
					0			&		s E^{\red}  -  A^{\red}		
				\end{array}
			\right]^{-1}
			\left[
				\begin{array}{l}
					B	\\
					B^{\red}
				\end{array}
			\right]		
			 +  (D - D^{\red}) ,	
	\end{split}
\end{equation}
\normalsize
where $H(\: \cdot \: ; S^{\red})$ is as in (\ref{eq:red_trans}).
Assuming that the reduced system is at most index one and has semi-simple poles, 
by the Kronecker canonical form, there exist invertible $\sr \times \sr$ real matrices 
$W$, $V$ such that $WE^{\rm red}V$ is diagonal, and $WA^{\rm red}V$ is block diagonal 
with $2\times 2$ and $1\times 1$ blocks along the diagonal. Consequently,
the reduced system is equivalent to a system (with the same transfer function)
for which $A^{\red}$, $E^{\red}$ are converted into tridiagonal and diagonal forms, respectively.
Hence, under index one and semi-simple pole assumptions, we can perform the minimization
over tridiagonal $A^{\red}$ and diagonal $E^{\red}$.
Recalling the dimensions of $A^{\red}, B^{\red}, C^{\red}, D^{\red}, E^{\red}$, 
there are precisely $4\sr - 2 + \sr m + p \sr + pm$ optimization variables. 

The gradient descent algorithm, as well as quasi-Newton algorithms to minimize
${\mathcal F}$ require the gradients of ${\mathcal F}$. To this end, suppose there is a unique $\omega_\ast \geq 0$
satisfying 
\begin{equation*}
		\sigma_{\max}	\left(	H({\rm i} \omega_\ast)	-	H({\rm i} \omega_\ast ; S^{\red} )	\right)
						\;\;	=	\;\;
	{\mathcal F}(S^{\red})		
							\;\;	=	\;\;
		\sup_{\omega \geq 0}	\:		
		\sigma_{\max}	\left(	H({\rm i} \omega)	-	H({\rm i} \omega ; S^{\red})	\right)	\;\;	,
\end{equation*}
ensuring that ${\mathcal F}$ is differentiable at $S^{\red}$.
Additionally, let $u$, $v$ denote a consistent pair of unit left, right singular vectors corresponding to
$\sigma_{\max}	\left(	H({\rm i} \omega_\ast) \right.$	$-$	$\left. H({\rm i} \omega_\ast ; S^{\red} ) \right)$,
and let us introduce
$ \:
	\widetilde{u} := u^\ast C^{\red} ({\rm i} \omega_\ast E^{\red} -  A^{\red} )^{-1},	\;\;
	\widetilde{v} := ({\rm i} \omega_\ast E^{\red} -  A^{\red} )^{-1} B^{\red} v.
$
Then, by employing the analytical formulas for the derivatives of singular value functions 
\cite{Lancaster1964}, \cite[Section 3.3]{MenYK2014}, the gradients of ${\mathcal F}$ are given by
\begin{equation}\label{eq:grad_system}
	\begin{split}
	&	\nabla_{A^{\red}} {\mathcal F}(S^{\rm \red})
		\,	=	\,	-{\rm diag} ( \Re ( \widetilde{u}^T \odot \widetilde{v} ) )	-	{\rm diag}_{-1}	( \Re ( \widetilde{u}(2:\sr)^T \odot \widetilde{v}(1:\sr-1) ) )		\\[.3em]
		&	\hskip 35ex				-	{\rm diag}_{+1}	( \Re ( \widetilde{u}(1:\sr-1)^T \odot \widetilde{v}(2:\sr) ) ) \: ,		\\[.5em]
	&	\nabla_{E^{\rm \red}} {\mathcal F}(S^{\rm \red})
		\,	=	\, 	-\omega_\ast \cdot {\rm diag} ( \Im (  \widetilde{u}^T \odot \widetilde{v} ) ) \: ,	
		\;\;\;\;
		\nabla_{B^{\rm \red}} {\mathcal F}(S^{\rm \red})
		\,	=	\, 	- \Re ( \widetilde{u}^T \, v^T ) \: ,	\\[.5em]
	&	\nabla_{C^{\rm \red}} {\mathcal F}(S^{\rm \red})
		\,	=	\, 	- \Re ( \overline{u} \:\, \widetilde{v}^T ) \: ,		\;\;\;\;
		\nabla_{D^{\rm \red}} {\mathcal F}(S^{\rm \red})
		\,	=	\, 	- \Re( \overline{u} \:\, v^T ) \: ,
	\end{split}
\end{equation}
where $\odot$ denotes the Hadamard product, $\overline{u}$ denotes the complex conjugate of $u$, and
the notation ${\rm diag}(w)$ represents the square diagonal matrix whose diagonal entries are formed of the 
entries of the vector $w$. The notations ${\rm diag}_{-1}(w)$ and ${\rm diag}_{+1}(w)$ are similar to ${\rm diag}(w)$
but with the difference that the subdiagonal and superdiagonal entries of the matrix are filled with the entries of $w$
rather than the diagonal entries.

It is essential that a quasi-Newton method such as BFGS generates approximate Hessians that are positive definite.
This is traditionally imposed by the line-searches. For instance, if BFGS is to be used to minimize ${\mathcal F}$,
then a line-search ensuring the satisfaction of the weak Wolfe conditions may be adopted so that the approximate Hessians remain 
positive definite. On the other hand, for the gradient descent algorithm to minimize ${\mathcal F}$ it is sufficient to 
adopt a simpler line-search that guarantees only sufficient reduction in the objective, e.g., an Armijo backtracking line-search.

One difficulty with using methods such as gradient descent and BFGS to minimize ${\mathcal F}$
is that these algorithms converge rather slowly only at a linear rate at best. This may sound surprising
especially for BFGS, which typically converges superlinearly for smooth problems. Slower convergence
for BFGS is an artifact of nonsmoothness. As a result of linear convergence at best, the objective ${\mathcal F}$
typically needs to be evaluated many times until reaching a prescribed accuracy. This may be prohibitively expensive,
as it is apparent from (\ref{eq:objective2}) that evaluation of ${\mathcal F}$ involves the computation of the 
${\mathcal L}_\infty$ norm of ${\mathcal H}(\: \cdot \: ; S^{\red})$, the transfer function for a large-scale 
system assuming the original system in (\ref{eq:state_space}) is large-scale.

To illustrate the slow convergence issues in the previous paragraph, and the computational difficulties 
that come with it, we apply the gradient descent algorithm to the \texttt{iss} example from the
SLICOT collection. The system associated with this example has order $n = 270$,  and $m = p = 3$. 
We attempt to solve the ${\mathcal L}_\infty$ model reduction problem for $\sr = 12$ starting with the initial
reduced order model generated by the balanced truncation approach. 
 The errors (${\mathcal F}$) and the 2-norms of the gradients of the errors ($\| \nabla {\mathcal F} \|_2$) of the iterates 
 of the gradient descent algorithm are reported in Table \ref{table:direct_gd_converge}. It takes 37 iterations until
 the errors in two consecutive iterations differ by no more than $10^{-6}$ in a relative sense. 
  The initial ${\mathcal L}_\infty$-norm error 0.004470060020 (of the system obtained from the balanced truncation) 
 is reduced to 0.002415438945 after 37 iterations. The eventual reduced model obtained
 appears to be a local minimizer of ${\mathcal F}$ up to prescribed tolerances, as can be observed from the 
 plots in Figure \ref{fig:direct_gd_converge}. Note however that according to the last columns in Table \ref{table:direct_gd_converge} 
 the gradients of ${\mathcal F}$ do not seem to be converging to zero, which indicates that the objective is not
 differentiable at the local minimizer to be converged.
 Meanwhile, the objective ${\mathcal F}$ is evaluated 624 times, since the line-search at each iteration
 requires several objective function evaluations (i.e., to be precise 8-28 evaluations per iteration) until the satisfaction 
 of the sufficient decrease condition. This results in a total runtime of about 500 seconds, costly for a system of relatively 
 small order. To conclude, direct applications of the gradient descent and quasi-Newton algorithms do not seem viable 
 for systems of even modest order (e.g., a few thousands).

\begin{table} 
\centering
\caption{ This concerns the ${\mathcal L}_\infty$ model reduction of the \texttt{iss} example with $\sr = 12$. The objective
${\mathcal F}^{(k)} :=$  
${\mathcal F}( A^{(k)}, B^{(k)}, C^{(k)}, D^{(k)}, E^{(k)})$, and the 2-norm
of $\, \nabla {\mathcal F}^{(k)} := \nabla {\mathcal F}( A^{(k)}, B^{(k)}, C^{(k)}, D^{(k)}, E^{(k)})$ 
for the iterate $( A^{(k)}, B^{(k)}, C^{(k)}, D^{(k)}, E^{(k)})$ by the gradient descent method
at the $k$th iteration are listed.}
\label{table:direct_gd_converge}
\small
\begin{tabular}{cc}
	\hskip 3ex
	\begin{tabular}{c|cc}
		\hskip 1ex $k$ \hskip 2ex  		& 		 ${\mathcal F}^{(k)}$     		&	 		$\|  \nabla {\mathcal F}^{(k)}  \|_2$  \\    
		\hline
		0		&		0.004470060020  		&		1.000093488			\\
		1 		&		0.004346739384 		&	 	0.833556647			\\
		2 		&		0.003609940202 		&		1.000097230			\\
		3 		&		0.003175718111 		&		0.769359926			\\
		4 		&		0.002975716755  		&		1.000095596			\\
		5 		&		0.002946113130 		&		0.999918608			\\
		6 		&		0.002697635041 		&		0.844275929			\\
		7 		&		0.002656707905 		&		0.999952423
	\end{tabular}	
				\hskip 2ex
				&
				\hskip 2ex
	\begin{tabular}{c|cc}
		\hskip 1ex $k$ \hskip 2ex  		& 		 ${\mathcal F}^{(k)}$     		&	 		$\|   \nabla {\mathcal F}^{(k)}  \|_2$  \\    
		\hline
		30		&		0.002415516341 		&		0.803721909			\\
		31 		&		0.002415479783		&	 	1.000008471			\\
		32 		&		0.002415475189 		&		0.803718441			\\
		33 		&		0.002415456030 		&		1.000008467			\\
		34 		&		0.002415454613 		&		0.803716708			\\
		35 		&		0.002415444154 		&		1.000008465			\\
		36 		&		0.002415439844 		&		0.803714645			\\
		37 		&		0.002415438945 		&		1.000008462
	\end{tabular}
\end{tabular}
\end{table}

\begin{figure} 
 \centering
		\begin{tabular}{cc}
			 \subfigure[$(1,1)$, $A^{\red}$ \label{fig:1a}]{\includegraphics[width = .47\textwidth]{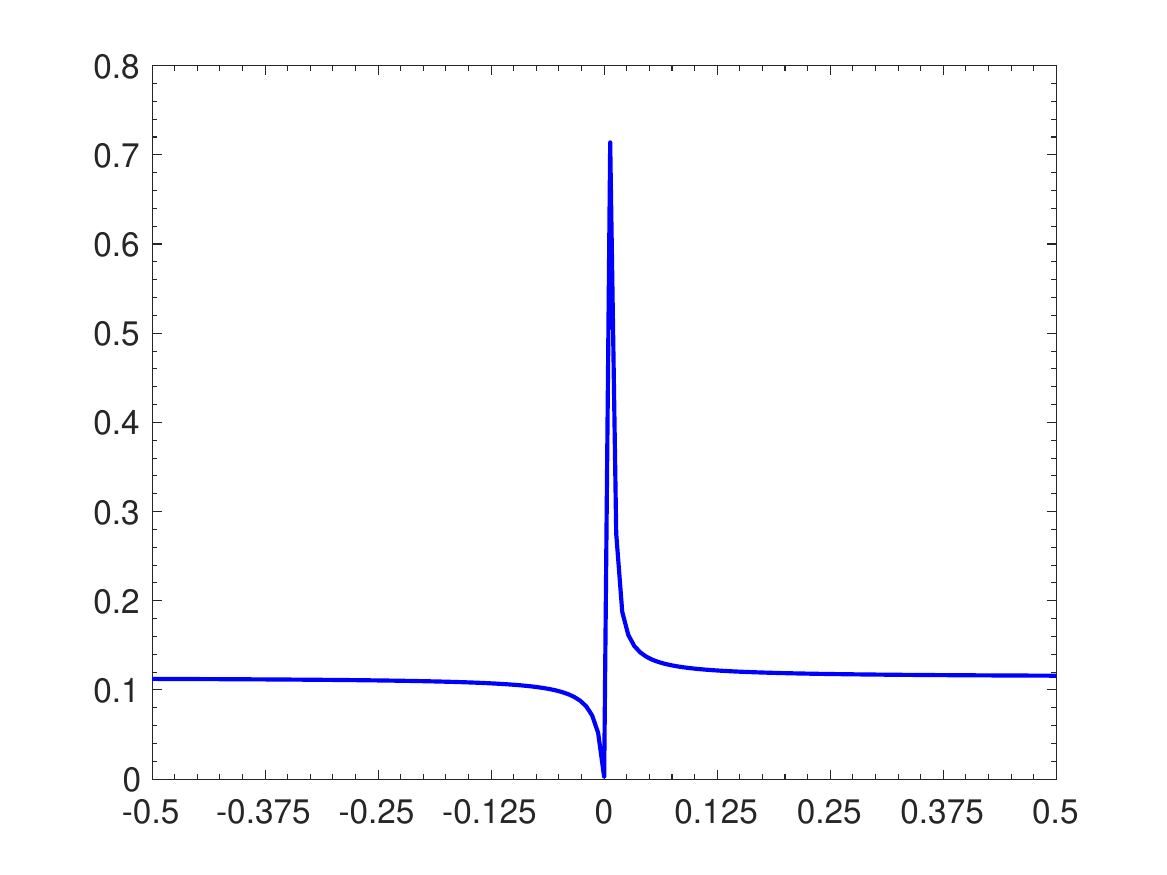}} & 		 
			 \subfigure[$(11,10)$, $A^{\red}$ \label{fig:1b}]{\includegraphics[width = .47\textwidth]{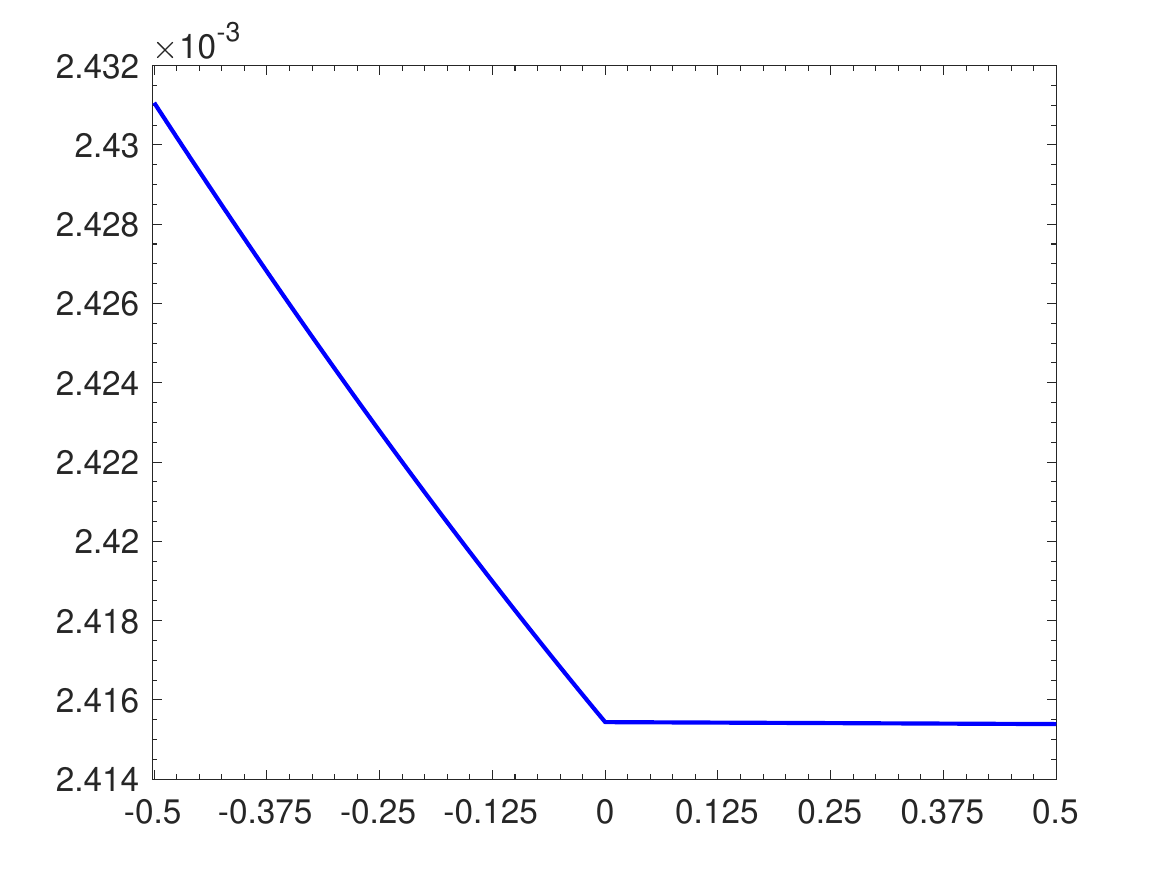}}
			 \\  
			 \subfigure[$(6,7)$, $A^{\red}$ \label{fig:1c}]{\includegraphics[width = .47\textwidth]{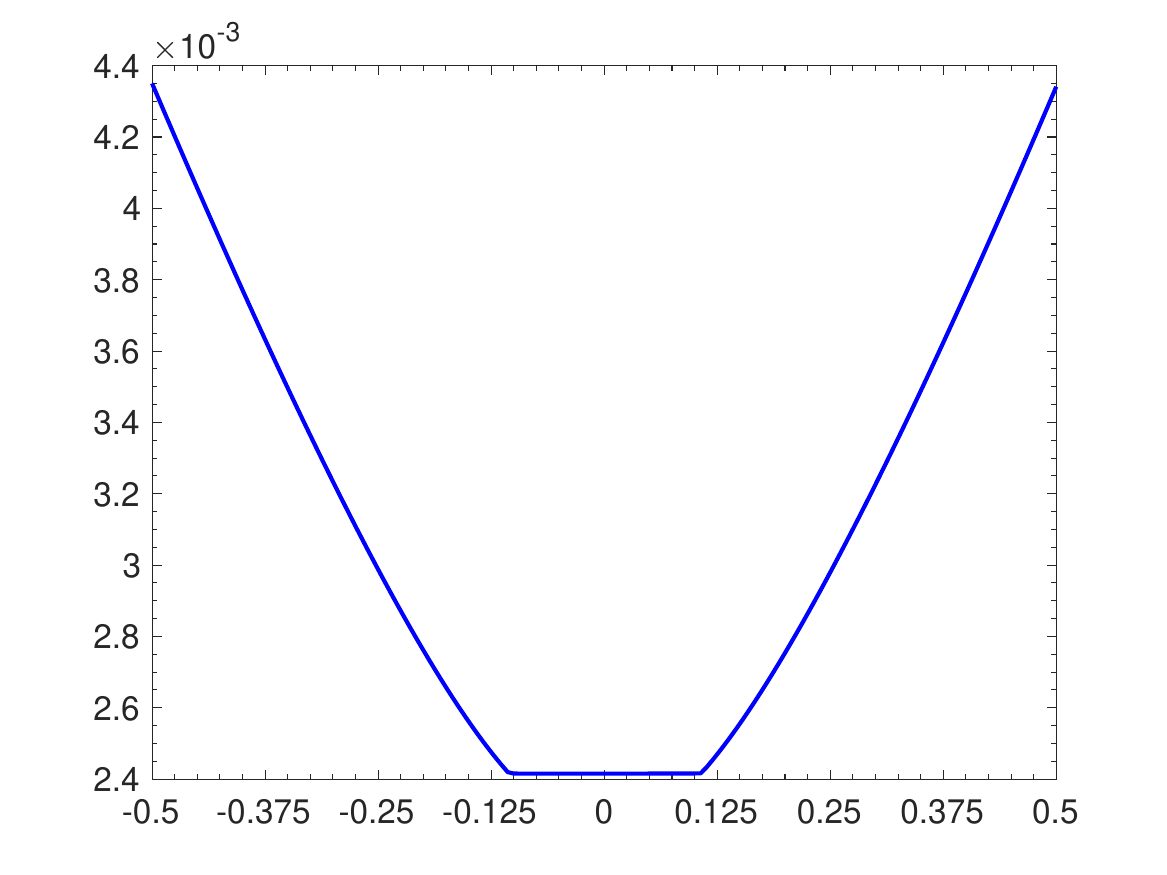}} &			 
			 \subfigure[$(1,1)$, $E^{\red}$ \label{fig:1d}]{\includegraphics[width = .47\textwidth]{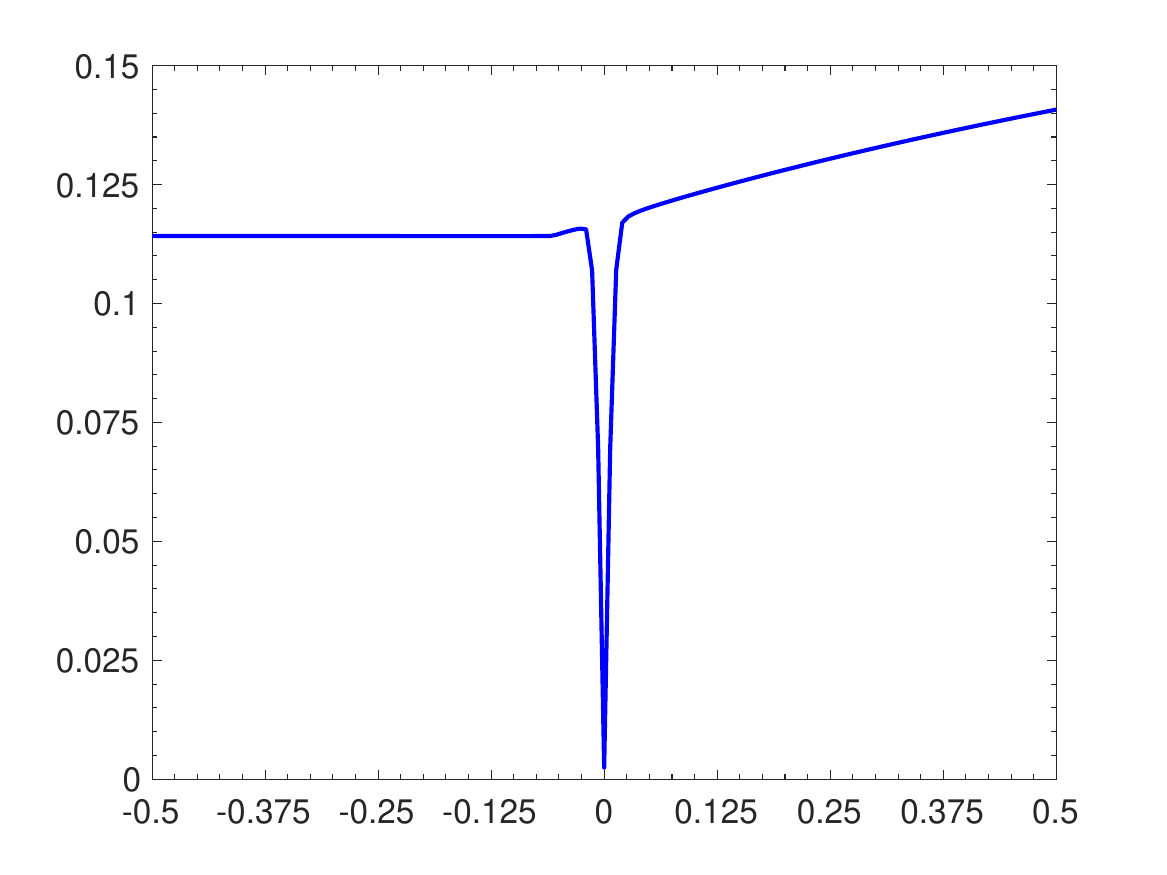}} 
			 \\
			 \subfigure[$(5,5)$, $E^{\red}$ \label{fig:1e}]{\includegraphics[width = .47\textwidth]{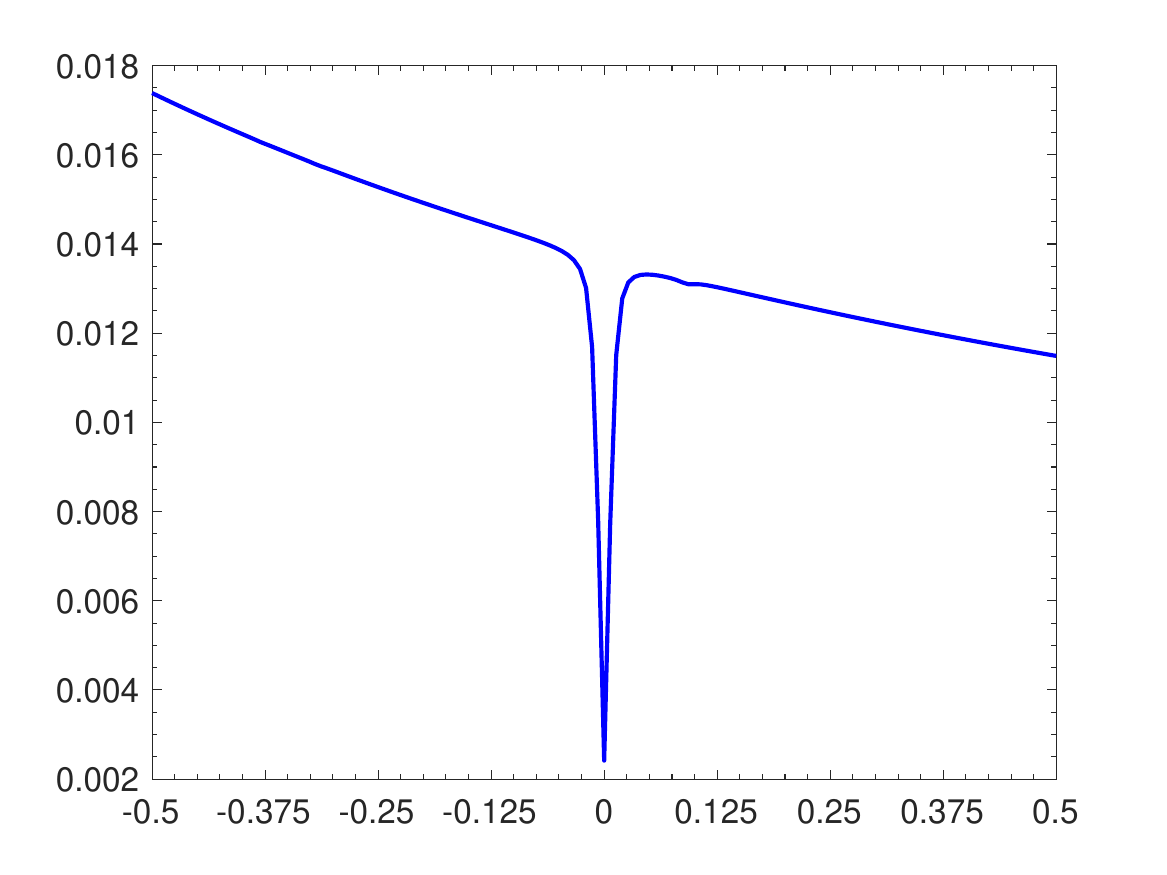}} &			 
			 \subfigure[$(8,3)$, $B^{\red}$ \label{fig:1f}]{\includegraphics[width = .47\textwidth]{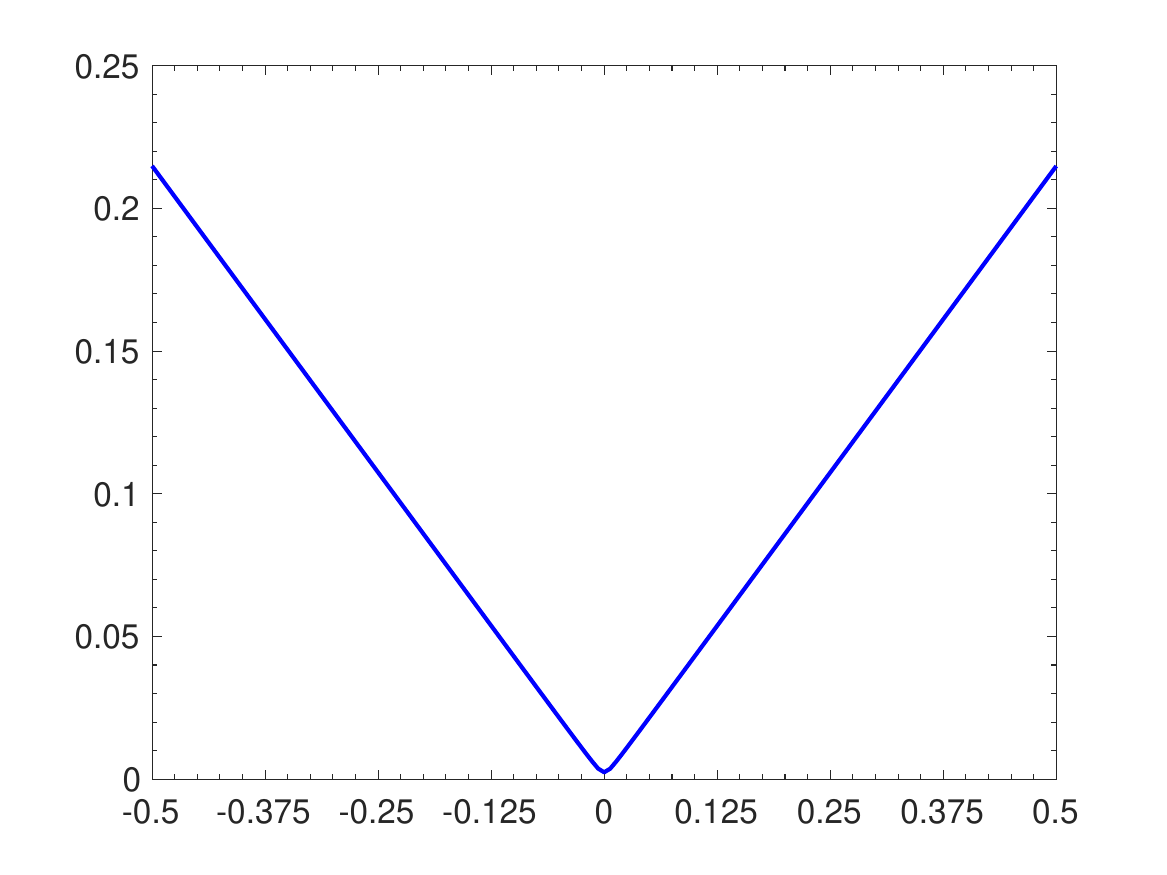}} 
			 \\
			 \subfigure[$(1,5)$, $C^{\red}$ \label{fig:1g}]{\includegraphics[width = .47\textwidth]{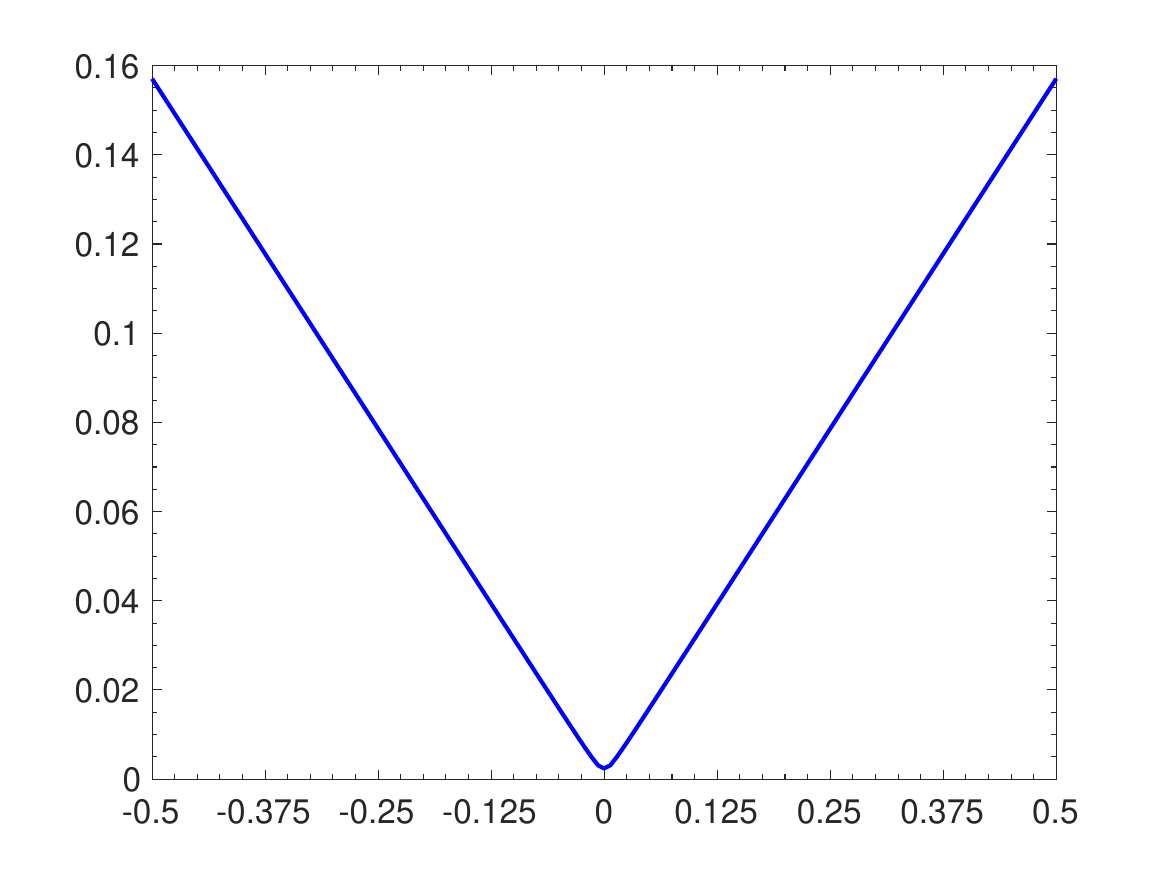}} &
			 \subfigure[$(3,3)$, $D^{\red}$ \label{fig:1h}]{\includegraphics[width = .47\textwidth]{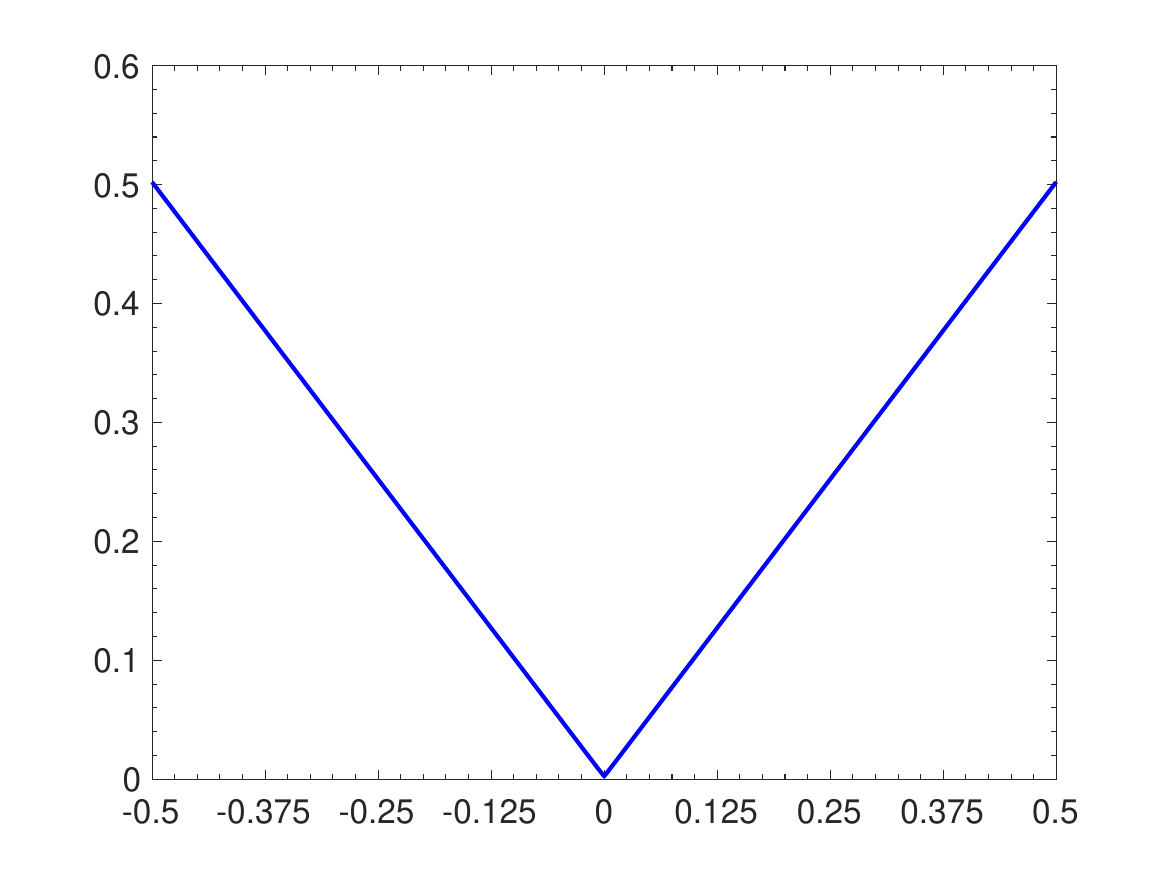}}		
		\end{tabular}   
		\caption{  
		The locally minimal reduced system generated by the gradient descent method for the \texttt{iss}
		example and $\sr = 12$ is varied, and the error ${\mathcal F}$ is plotted as a function of the variation. 
		In each one of the plots (a)$-$(h), only the indicated entry of one of the optimal coefficients 
		$A^{\red}, B^{\red}, C^{\red}, D^{\red}, E^{\red}$
		is varied by amounts in $[-0.5, 0.5]$. Zero variation corresponds to the optimal reduced system.}
		\label{fig:direct_gd_converge}
\end{figure}

\section{A Subspace Framework}\label{sec:sub_fr}
The computational difficulty in minimizing the objective ${\mathcal F}$ in (\ref{eq:objective2}) is due to the
large order of the original system $S = (A, E, B, C, D)$. In this section, we propose to replace this system with 
a system of smaller order $S_r = (A_r, E_r, B_r, C_r, D_r)$ with the state-space representation
\begin{equation}\label{eq:small_system}
		E_r x'_r(t)	\;	 =	\;	A_r x_r(t)		+	B_r u(t),	\quad\quad
		y(t)	\;	 =	\;	C_r x_r(t)	+	D u(t),
\end{equation}
and solve the resulting ${\mathcal L}_\infty$ model reduction problem, that is minimize
\begin{equation}\label{eq:reduced_objective}
	 {\mathcal F}_r(S^{\red})
				\;	=	\;
		\sup_{\omega \geq 0}	\:		
	\sigma_{\max}	\left(	H_r({\rm i} \omega)	-	H({\rm i} \omega ; S^{\red}) 		\right)	
				\;	=	\;
		\sup_{\omega \geq 0}	\:		
	\sigma_{\max}	\left(	{\mathcal H}_r({\rm i} \omega  	;  S^{\red} )	\right) 	\; ,	
\end{equation}
where
\begin{equation*}
	\begin{split}
	& H_r(s)	\;	:=	\;	C_r(s E_r  -  A_r)^{-1} B_r		+	D	\;	,	\;\;\;	\text{and}  \\
	& {\mathcal H}_r(s; S^{\red})			
			:=
			[
					C_r	\;\;	-C^{\red}
			]
			\left[
				\begin{array}{cc}
					s E_r  -  A_r		&		0		\\
					0		&		s E^{\red}  -  A^{\red}		
				\end{array}
			\right]^{-1}
			\left[
				\begin{array}{l}
					B_r	\\
					B^{\red}
				\end{array}
			\right]
					+
			(D - D^{\red})	\;  .		
	\end{split}
\end{equation*}
The question that we need to address is
how do we form a small system $S_r = (A_r, E_r, B_r, C_r, D)$ that is a good representative of the original
system near a local minimizer of the original ${\mathcal L}_\infty$ model reduction problem.

Recall how pure Newton's method operates to minimize a function $f : {\mathbb R}^q \rightarrow {\mathbb R}$. 
It approximates $f$ with a quadratic model, and finds a local minimizer $\widetilde{x}$ of the quadratic model. 
Then, assuming $f$ is twice differentiable at $\widetilde{x}$, 
it refines the  quadratic model so that the refined quadratic model $q$ satisfies  $f(\widetilde{x}) = q(\widetilde{x})$,  
$\nabla f(\widetilde{x}) =  \nabla q(\widetilde{x})$ and $\nabla^2 f(\widetilde{x}) = \nabla^2 q(\widetilde{x})$. 
In the context of ${\mathcal L}_\infty$ model reduction, we view ${\mathcal F}_r$ as the model function for ${\mathcal F}$,
even though ${\mathcal F}_r$ is not quadratic. We minimize ${\mathcal F}_r$ locally rather than ${\mathcal F}$,
and refine the small system in (\ref{eq:small_system}) with the hope that the objective error function ${\mathcal F}_{r+1}$
of the refined system interpolates ${\mathcal F}$ and its first two derivatives at the computed minimizer of ${\mathcal F}_r$.

The small system in (\ref{eq:small_system}) is obtained from the original system by applying the Petrov-Galerkin
framework; for given two subspaces ${\mathcal V}_r$, ${\mathcal W}_r$ of ${\mathbb R}^n$ of equal dimension, the state 
space of the original system is restricted to ${\mathcal V}_r$ and the differential part of the resulting system
is imposed to be orthogonal to ${\mathcal W}_r$. Formally, denoting with $V_r, W_r$ matrices whose
columns form orthonormal bases for 
 ${\mathcal V}_r$, ${\mathcal W}_r$ and with $x_r(t)$ the restricted state, the original system is approximated by
 \[
 	W_r^T ( E V_r x'_r(t)	\;	 -	\;	A  V_r x_r (t)		-	B u(t) ) \; = \; 0,	\quad\quad
		y(t)	\;	 =	\;	C V_r x_r (t)	+	D u(t),
 \]
 giving rise to a system of the form (\ref{eq:small_system}) with
 \begin{equation}\label{eq:small_coeffs}
 	E_r \: = \: W_r^T E V_r	,		\quad		A_r  \: = \:  W_r^T A V_r,	\quad		B_r  \: = \:  W_r^T B,		\;\;	\text{and} 	
				\;\;		C_r  \: = \: C V_r.
 \end{equation}
 For the realization of the ideas in the previous paragraph, we need to be equipped with a tool that gives us
 the capability to interpolate $H(s)$ and its derivatives at a prescribed point in the complex plane with 
 those of the transfer function for the small system. This tool is introduced in the next result, which
 follows from \cite[Theorem 1]{BeaG2009}.
 \begin{theorem}\label{thm:interpolate}
 Let $\mu \in {\mathbb C}$ be such that $A - \mu E$ is invertible.  
 Suppose
 \begin{equation*}
 \begin{split}
	& 	  \bigoplus_{j=0}^\kappa   \Re \left[ \{ (A - \mu E)^{-1} E \}^j (A - \mu E)^{-1} B	 \right] 	\; \subseteq \; {\mathcal V}_{r},		\\
	&		\bigoplus_{j=0}^\kappa \Im \left[ \{ (A - \mu E)^{-1} E \}^j (A - \mu E)^{-1} B \right]	\; \subseteq \; {\mathcal V}_{r},		\\
	& 	  \bigoplus_{j=0}^\kappa   \Re  \left[ C  (A - \mu E)^{-1}  \{ E (A - \mu E)^{-1}  \}^j  \right]^\ast 	\; \subseteq \; {\mathcal W}_{r},		\text{ and}		\\
	&		\bigoplus_{j=0}^\kappa   \Im  \left[ C  (A - \mu E)^{-1}  \{ E (A - \mu E)^{-1}  \}^j  \right]^\ast 	\; \subseteq \; {\mathcal W}_{r}.
\end{split}	
 \end{equation*}
 \normalsize
 Then, with $A_r, E_r, B_r, C_r$ defined as in (\ref{eq:small_coeffs}), if $A_r - \mu E_r$ is invertible, we have
 \medskip
 \begin{enumerate}
 	\item[\bf (i)]  $H(\mu) = H_r(\mu) \,$ and $\, H(\overline{\mu}) = H_r(\overline{\mu})$,
	\smallskip
	\item[\bf (ii)]  $H^{(j)}(\mu) = H^{(j)}_r(\mu) \,$ and $\, H^{(j)}(\overline{\mu}) = H^{(j)}_r(\overline{\mu}) \,$
	for $j = 1, \dots , 2\kappa + 1$. 
 \end{enumerate}
 \end{theorem}
 
 \medskip
 
Our proposed subspace framework at iteration $r$ first finds a minimizer of ${\mathcal F}_r(S^{\red})$, say 
$S^{\red}_r = (A^\red_r, B^\red_r, C^\red_r, D^\red_r, E^\red_r)$.  This is followed by the computation
of an $\omega_r \in {\mathbb R}, \, \omega_r \geq 0$ such that
 \[
 	{\mathcal F}(S^\red_r)
				\;	=	\;
	\sup_{\omega \geq 0} \,   \sigma_{\max}	\left(	H({\rm i} \omega)	-	H({\rm i} \omega ; S^{\red}_r)  \right)
				\;	=	\;
		\sigma_{\max}	\left(	H({\rm i} \omega_r)	-	H({\rm i} \omega_r ; S^{\red}_r) 		\right) \: .
 \]
Computing such an $\omega_r$ requires the large-scale ${\mathcal L}_\infty$-norm computation in (\ref{eq:objective2})
but by replacing $S^{\red} = (A^{\red}, B^{\red}, C^{\red},$ $D^{\red}, E^{\red})$ with 
$S^{\red}_r = (A^\red_r, B^\red_r, C^\red_r, D^\red_r, E^\red_r)$.
Then subspaces are expanded so that $H$ and its first three derivatives are interpolated at ${\rm i} \omega_r$
by those of the transfer function for the small system. A formal description of the framework is given 
in Algorithm \ref{alg:SM} below. As the subspaces ${\mathcal V}_r$ and ${\mathcal W}_r$ are
required to be of equal dimension, the description assumes that the number of inputs
and the outputs are equal, i.e., $m = p$. Even if it is omitted here for simplicity, it is straightforward 
to modify the directions $\widetilde{V}_{r+1}$, $\widetilde{W}_{r+1}$ in lines \ref{V_dir}-\ref{W_dir} to be added 
to the subspaces ${\mathcal V}_r$, ${\mathcal W}_r$ in order to deal with the systems for which $m \neq p$. 
The final refinement step in line \ref{alg:refine} aims at the satisfaction
of the interpolation condition ${\mathcal F}(S^{\red}_r) = {\mathcal F}_{r+1}(S^{\red}_r)$,
as well as the interpolation conditions on the derivatives of ${\mathcal F}(S^{\red})$ and 
${\mathcal F}_{r+1}(S^{\red})$ at $S^{\red}_r$. This step is elaborated on in the next subsection.

\begin{algorithm}[tb]
 \begin{algorithmic}[1]
 
\REQUIRE{System $S = (A, E, B, C, D)$ as in \eqref{eq:state_space}, 
	the order $\sr \in {\mathbb Z}^+$ of the reduced system sought,
	and an initial estimate $S^{\red}_0 = (A^{\red}_0, E^{\red}_0, B^{\red}_0, C^{\red}_0, D^{\red}_0)$
	of order $\sr$ for a minimizer of ${\mathcal F}$ as in \eqref{eq:objective2}.}
\ENSURE{Estimate $S^{\red}_\star = (A^\red_\star, E^{\red}_\star, B^\red_\star, C^\red_\star, D^\red_\star)$ 
for a minimizer of ${\mathcal F}$ as in \eqref{eq:objective2}.}

\vskip 1.5ex

\STATE Choose the initial subspaces ${\mathcal V}_0, {\mathcal W}_0$ and orthonormal bases $V_0, W_0$ for them. \label{init_subspaces}

\vskip 1.5ex

\STATE Form $A_0, B_0, C_0, E_0$ using (\ref{eq:small_coeffs}),$\;$
			and let $\: S_{0} = (A_{0}, E_{0}, B_{0}, C_{0}, D) \,$.	

\vskip 1.5ex

\vskip .7ex
\textcolor{mygreen}{\textbf{$\%$ main loop}}
\FOR{$r = 0,\,1,\,\dots$}
	
	\vskip 1.5ex


	 \IF{$r \geq 1$}
	
	\vskip .4ex
	
	  \STATE $S^{\red}_r$ $\; \gets \;$ 
	  			a minimizer of ${\mathcal F}_r(S^{\red})$. 	\label{exp_start}	
				
	\vskip .4ex	
		
	\ENDIF				
				
	\vskip 1.5ex			
				
	\STATE $\omega_r$ 		$\; \gets \;$	a maximizer of 		
									$\sigma( \omega; S^{\red}_r) = \sigma_{\max}	\left(	H({\rm i} \omega)	-	H({\rm i} \omega ; S^{\red}_r ) \right)$
	over $\omega \geq 0$.   \label{large_linfinity}
	
	\vskip 1.5ex
	
	\IF{$r \geq 1$}
	
	\vskip .4ex
	
		\STATE \textbf{Return}  if convergence has occurred with 
		$\: S^\red_\star \: \gets \: S^\red_r$.  \label{return_spec}
		
	\vskip .4ex	
		
	\ENDIF

	\vskip 1.5ex

	\textcolor{mygreen}{\textbf{$\%$ expand the subspaces to interpolate at ${\rm i} \omega_r$}}  \\[.2em]      	
	\STATE  $\widetilde{V}_{r+1} \gets \left[ \;\; \Re [ ({\rm i} \omega_r E - A)^{-1} B ]  \;\;\;\;  \Re [ ({\rm i} \omega_r E - A)^{-1} E ({\rm i} \omega_r E - A)^{-1} B] \right.$	\\[.2em]
	\phantom{aaaaaaaaaaaaa}			
			$ \left.	\Im [ ({\rm i} \omega_r E - A)^{-1} B ]  \;\;\;\;  \Im [ ({\rm i} \omega_r E - A)^{-1} E ({\rm i} \omega_r E - A)^{-1} B]	\;\;	\right]$. \label{V_dir}
	
	\vskip .4ex
	
  \STATE  $\widetilde{W}_{r+1} \gets \left[ \;\; \Re [ ({\rm i} \omega_r E - A)^{-\ast} C^\ast ]  \;\;\;\;  \Re [ ({\rm i} \omega_r E - A)^{-\ast} E ({\rm i} \omega_r E - A)^{-\ast} C^\ast] \right.$	\\[.2em]
	\phantom{aaaaaaaaaaaaa}			
			$ \left.	\Im [ ({\rm i} \omega_r E - A)^{-\ast} C^\ast ]  \;\;\;\;  \Im [ ({\rm i} \omega_r E - A)^{-\ast} E ({\rm i} \omega_r E - A)^{-\ast} C^\ast]	\;\;	\right]$.  \label{W_dir}
	
	\vskip .3ex
		        
	\STATE $V_{r+1} \gets \operatorname{orth}\left(\begin{bmatrix} V_{r} & \widetilde{V}_{r+1} \end{bmatrix}\right)
			 \text{ and } \: W_{r+1} \gets \operatorname{orth}\left(\begin{bmatrix} W_{r} & \widetilde{W}_{r+1} \end{bmatrix}\right)$. \label{R_orthogonalize}

	\vskip 1.5ex

	\textcolor{mygreen}{\textbf{$\%$ update the small system}}       
	\STATE 	Form $A_{r+1}, B_{r+1}, C_{r+1}, E_{r+1}$ using (\ref{eq:small_coeffs}), \\[.2em]
	\hskip 18ex
	and let $\: S_{r+1} = (A_{r+1}, E_{r+1}, B_{r+1}, C_{r+1}, D) \,$.	\label{form_rp1}

	\vskip 1.5ex

	\textcolor{mygreen}{\textbf{$\%$ refine the small system}}       
	\STATE   Refine $V_{r+1}$, $W_{r+1}$ and $S_{r+1}$ if necessary
	(using Algorithm \ref{alg:refine_step}). \label{alg:refine}

	\vskip 1.5ex

\ENDFOR
 \end{algorithmic}
\caption{Subspace framework for ${\mathcal L}_\infty$ model reduction}
\label{alg:SM}
\end{algorithm}

\subsection{Refinement Step}
First we make a few observations regarding the relation between 
${\mathcal F}(S^{\red}_r)$ and ${\mathcal F}_{r+1}(S^{\red}_r)$
at the $r$th subspace iteration in Algorithm \ref{alg:SM} right before 
the refinement step.

At the $r$th iteration of Algorithm \ref{alg:SM} right after line \ref{form_rp1}, 
by Theorem \ref{thm:interpolate}, we have
\begin{equation}\label{eq:int_props}
	H({\rm i} \omega_r)	=	H_{r+1}({\rm i} \omega_r)	
	\quad	\text{and}		\quad	H^{(j)}({\rm i} \omega_r)	=	H^{(j)}_{r+1}({\rm i} \omega_r)	
\end{equation}
for $j = 1, 2, 3$ under the assumptions that $A - {\rm i} \omega_r E$ and $A_{r+1} - {\rm i} \omega_r E_{r+1}$
are invertible. Consequently, $H({\rm i} \omega_r)  - H({\rm i} \omega_r; S^{\red}_r)$ and 
$H_{r+1}({\rm i} \omega_r)  - H({\rm i} \omega_r; S^{\red}_r)$ are equal, and share 
the same set of left and right singular vectors.
It immediately follows that setting
\begin{equation}\label{eq:red_sval}
	\sigma_{r+1}(\omega; S^{\red}) 	\;  :=	\;  \sigma_{\max} ( H_{r+1}({\rm i} \omega)  - H({\rm i} \omega; S^{\red}) ) \: ,
\end{equation}
and recalling the definition of $\sigma(\omega; S^{\red})$ in (\ref{eq:objective}), we have
\begin{equation}\label{eq:int_props_sval}
		\sigma( \omega_r; S^{\red}_r)		\;	=	\;	\sigma_{r+1}( \omega_r; S^{\red}_r).
\end{equation}

Indeed, as the singular values and vectors of $H({\rm i} \omega_r)  - H({\rm i} \omega_r; S^{\red}_r)$ and 
$H_{r+1}({\rm i} \omega_r)  - H({\rm i} \omega_r ; S^{\red}_r)$  are
the same, and the first two derivatives of $H({\rm i} \omega)  - H({\rm i} \omega ; S^{\red}_r )$ and 
$H_{r+1}({\rm i} \omega)  - H({\rm i} \omega ; S^{\red}_r)$ at $\omega = \omega_r$ are equal due to (\ref{eq:int_props}),
we also have
\begin{equation}\label{eq:interp_interm}
	\frac{{\rm d}^j \sigma}{{\rm d} \omega^j} (\omega_r; S^{\red}_r)
		=	\;	\frac{{\rm d}^j \sigma_{r+1} }{{\rm d} \omega^j}	(\omega_r; S^{\red}_r) 
\end{equation}
for $j = 1,2$. Now $\omega_r$ is a global maximizer of $\sigma( \omega; S^{\red}_r)$ over $\omega$ implying
\[
	\frac{{\rm d} \sigma}{{\rm d} \omega} (\omega_r ; S^{\red}_r)
		\;	=	\;	0		\quad	\text{and}		\quad
	\frac{{\rm d}^2 \sigma}{{\rm d} \omega^2} (\omega_r ; S^{\red}_r)
		\;	\leq	\;	0	\:	.
\]
Assuming that the last inequality on the second derivative above holds strictly, (\ref{eq:interp_interm}) implies
$\omega_r$ is also a local maximizer of $\sigma_{r+1}( \omega; S^{\red}_r)$.

Regarding ${\mathcal F}(S^{\rm red}_r)$ and ${\mathcal F}_{r+1}(S^{\rm red}_r)$,
the following relation always hold:
\begin{equation}\label{eq:F_rel}
	\begin{split}
	\quad	{\mathcal F}(S^\red_r)	&	=	\sup_{\omega \geq 0} \sigma( \omega; S^{\red}_r)		\\
					&	=	\sigma( \omega_r; S^{\red}_r)		\\[.2em]
					&	=	\;	\sigma_{r+1}( \omega_r; S^{\red}_r)	\\[.2em]
					&	\leq	\sup_{\omega \geq 0} \sigma_{r+1}( \omega; S^{\red}_r)
						=	{\mathcal F}_{r+1}(S^\red_r),
	\end{split}
\end{equation}
where the third equality is due to the interpolatory property in (\ref{eq:int_props_sval}). As argued in the
previous paragraph, the point $\omega_r$ is not only a global maximizer of $\sigma( \omega; S^{\red}_r)$, 
but also generically a local maximizer of $\sigma_{r+1}( \omega; S^{\red}_r)$.
If it happens that $\omega_r$ is also a global maximizer of $\sigma_{r+1}( \omega; S^{\red}_r)$ 
beyond being a local maximizer, then the inequality in the equation above becomes an equality, 
and we have the interpolation property
\begin{equation}\label{eq:Finterpolate}
	{\mathcal F}(S^\red_r)
				\;	=	\;
	{\mathcal F}_{r+1}(S^\red_r).
\end{equation}

In the refinement step, if it happens that $\omega_r$ is merely a local maximizer of 
$\sigma_{r+1}( \omega; S^{\red}_r)$, but not a global maximizer, then we find a global 
maximizer $\omega^{(0)}_{r}$ of $\sigma_{r+1}( \omega; S^{\red}_r)$ over 
$\omega \geq 0$ (equivalently compute the ${\mathcal L}_\infty$ norm of 
$H_{r+1}(\cdot) - H(\, \cdot \,\, ; \, S^{\red}_r$)).
Observe that finding such a global maximizer has a small computational cost, as the orders 
of $S_{r+1}$ and $S^{\red}_r$ are small. Then, by employing Theorem \ref{thm:interpolate}, 
we expand the subspaces ${\mathcal V}_{r+1}$, ${\mathcal W}_{r+1}$ further so that the interpolatory
properties are attained between $H_{r+1}({\rm i} \omega)$ after this refinement 
and $H({\rm i} \omega)$ at $\omega = \omega^{(0)}_r$, which in turn implies 
interpolatory properties between $\sigma_{r+1}( \omega; S^{\red}_r)$ and $\sigma( \omega; S^{\red}_r)$
at $\omega = \omega^{(0)}_{r}$. If $\omega_r$ 
after this refinement of $S_{r+1}$ is still only a local maximizer of $\sigma_{r+1}( \omega; S^{\red}_r)$,
but not a global maximizer, then we repeat this refinement
procedure of $S_{r+1}$ up until $\omega_r$ becomes a global maximizer of 
$\sigma_{r+1}( \omega; S^{\red}_r)$ (in practice up to prescribed tolerances).
A formal description of the refinement step is given below in Algorithm \ref{alg:refine_step}.
For simplicity, in line \ref{refine_if} of Algorithm \ref{alg:refine_step} it is assumed that $\omega_r$ is the unique 
global maximizer of $\sigma(\omega; S^\red_r)$. More generally, all of the global maximizers of 
$\sigma(\omega; S^{\red}_r)$ can be returned in line \ref{large_linfinity} of Algorithm \ref{alg:SM} 
(e.g., by employing the level-set methods to compute the ${\mathcal L}_\infty$ norm), and
whether $\omega^{(j)}_r$ is equal to any of these global maximizers can be checked in line \ref{refine_if}
of Algorithm \ref{alg:refine_step}.

Assuming
$\sigma(\omega ; S^{\rm red}_{r})$ is Lipschitz continuous, $\sigma_{r+1}(\omega ; S^{\rm red}_{r})$
is Lipschitz continuous with a uniform Lipschitz constant over the iterations of Algorithm \ref{alg:refine_step},
and the maximizer $\omega^{(j)}_r$ of $\sigma_{r+1}(\omega ; S^{\rm red}_{r})$ over $\omega \geq 0$ 
at every $j$ is required to be in a prescribed bounded interval, the gap $| \omega^{(j)}_r - \omega_r |$  
can be made less than any prescribed amount after finitely many iterations of Algorithm \ref{alg:refine_step}.
At this point, the interpolation condition (\ref{eq:Finterpolate}) is also met up to a multiple of the prescribed amount.

\begin{algorithm}[tb]
 \begin{algorithmic}[1]

\vskip  1.2ex

\FOR{$j = 0,\,1,\,\dots$}
	
	\vskip 1.8ex

	\STATE $\omega^{(j)}_r$ 		$ \gets $ a maximizer of 		
		$\sigma_{r+1}( \omega; S^{\red}_r) = 
			\sigma_{\max}	\left(	H_{r+1}({\rm i} \omega)	-	H({\rm i} \omega ; S^{\red}_r ) \right)$ \\[.2em]
	\hskip 8ex
	over $\omega \geq 0$.   \label{large_linfinity_2}

	\vskip 2ex
	
	\IF{$ \omega^{(j)}_r  \, = \, \omega_r $ 
			(up to prescribed tolerances)} \label{refine_if}
	
	\vskip 1ex
	
		\STATE \textbf{Terminate}  with $V_{r+1}$, $W_{r+1}$ and $S_{r+1}$.  \label{return_spec_2}
		
	\vskip 1ex	
		
	\ENDIF
		
	\vskip 2ex
		
	\textcolor{mygreen}{\textbf{$\%$ expand the subspaces to interpolate at ${\rm i} \omega^{(j)}_r$}}  \\[.2em]      	
	\STATE  $\widetilde{V}_{r+1} \gets \left[ \;\; \Re [ ({\rm i} \omega^{(j)}_r E - A)^{-1} B ]  \;\;\;\;  \Re [ ({\rm i} \omega^{(j)}_r E - A)^{-1} E ({\rm i} \omega_r E - A)^{-1} B] \right.$	\\[.2em]
	\phantom{aaaaaaaaaaaa}		
			$ \left.	\Im [ ({\rm i} \omega^{(j)}_r E - A)^{-1} B ]  \;\;\;\;  \Im [ ({\rm i} \omega^{(j)}_r E - A)^{-1} E ({\rm i} \omega^{(j+1)}_r E - A)^{-1} B]	\;\;	\right]$.
	
	\vskip .4ex
	
  \STATE  $\widetilde{W}_{r+1} \gets \left[ \;\; \Re [ ({\rm i} \omega^{(j)}_r E - A)^{-\ast} C^\ast ]  \;\;\;\;  \Re [ ({\rm i} \omega^{(j)}_r E - A)^{-\ast} E ({\rm i} \omega^{(j)}_r E - A)^{-\ast} C^\ast] \right.$	\\[.2em]
	\phantom{aaaaaaaaaaaai}			
			$ \left.	\Im [ ({\rm i} \omega^{(j)}_r E - A)^{-\ast} C^\ast ]  \;\;\;\;  \Im [ ({\rm i} \omega^{(j)}_r E - A)^{-\ast} E ({\rm i} \omega^{(j)}_r E - A)^{-\ast} C^\ast]	\;\;	\right]$.
	
	\vskip .3ex
		        
	\STATE $V_{r+1} \gets \operatorname{orth}\left(\begin{bmatrix} V_{r+1} & \widetilde{V}_{r+1} \end{bmatrix}\right)
			 \text{ and } \: W_{r+1} \gets \operatorname{orth}\left(\begin{bmatrix} W_{r+1} & \widetilde{W}_{r+1} \end{bmatrix}\right)$. \label{R_orthogonalize_2}

	\vskip 2ex

	\textcolor{mygreen}{\textbf{$\%$ update the small system}}       
	\STATE 	Form $A_{r+1}, B_{r+1}, C_{r+1}, E_{r+1}$ using (\ref{eq:small_coeffs}), \\[.2em]
	\hskip 18ex
	and let $\: S_{r+1} = (A_{r+1}, E_{r+1}, B_{r+1}, C_{r+1}, D) \,$.	\label{form_rp1_2}

	\vskip 1.5ex

\ENDFOR
 \end{algorithmic}
\caption{Refinement Step}
\label{alg:refine_step}
\end{algorithm}

\section{Interpolation Properties of the Subspace Framework}\label{sec:int_prop}
Suppose that $\omega_r$ is a global maximizer of $\sigma_{r+1}(\omega ; S^{\red}_r)$ 
by the termination of the refinement step, in which case the interpolation condition (\ref{eq:Finterpolate}) 
holds due to (\ref{eq:F_rel}). It can be shown that, assuming ${\mathcal F}$ and ${\mathcal F}_{r+1}$ are 
twice differentiable at $S^\red_r$, indeed all of the first two derivatives 
of ${\mathcal F}$ and ${\mathcal F}_{r+1}$ are equal at $S^\red_r$ as well. To this end, 
let $x_1, x_2$ be any two entries of the matrix variables $A^\red , B^\red , C^\red , D^\red, E^\red$ 
of ${\mathcal F}$ and ${\mathcal F}_{r+1}$. Recalling
\begin{equation*}
	\begin{split}
		{\mathcal H}( {\rm i} \omega ; S^\red ) 	
			\;	&	=		\;	 H({\rm i} \omega)  - H({\rm i} \omega ; S^\red) \, ,				\\
		{\mathcal H}_{r+1}( {\rm i} \omega ;  S^\red ) 	\;	&	=	\;	 
				H_{r+1}({\rm i} \omega)  - H({\rm i} \omega; S^\red) \, ,
	\end{split}
\end{equation*}
and by employing (\ref{eq:int_props}), it is apparent that
\begin{align}
			{\mathcal H} ( {\rm i} \omega_r ; S^\red_r ) 
					\;	&	=	\;
{\mathcal H}_{r+1} ( {\rm i} \omega_r ;  S^\red_r )	\,	,		\label{eq:Diffint_0} \\[.7em]
	\frac{\partial  {\mathcal H}}{\partial y} ( {\rm i} \omega_r ; S^\red_r ) 
		 	\;	&	=	\;
\frac{\partial  {\mathcal H}_{r+1}}{\partial y} ( {\rm i} \omega_r ; S^\red_r )	\,	,
		\label{eq:Diffint_1} \\[.2em]
\frac{\partial^2  {\mathcal H}}{\partial y \, \partial z} ( {\rm i} \omega_r ; S^\red_r )  \;	&	=	\;
\frac{\partial^2  {\mathcal H}_{r+1}}{\partial y \, \partial z} ( {\rm i} \omega_r ; S^\red_r )	\,	,
	 	\label{eq:Diffint_2}
\end{align}
for all $y, z \in \{ \omega, x_1, x_2 \}$. By exploiting
\begin{equation*}
	{\mathcal F}( S^\red ) 	
					\;	=	\;
	\sup_{\omega \geq 0} \:
	\sigma_{\max} ( {\mathcal H}( {\rm i} \omega ; S^\red ) ) \, ,  \quad
	{\mathcal F}_{r+1}( S^\red ) 	
					\; = \;
	\sup_{\omega \geq 0} \:
	\sigma_{\max} ( {\mathcal H}_{r+1}( {\rm i} \omega ; S^\red ) ) \, ,
\end{equation*}
and using implicit differentiation
\begin{equation}\label{eq:int_fder}
\begin{split}
	\frac{\partial  {\mathcal F}}{\partial y} ( S^\red_r ) 
		\;		&		=	\;
	\frac{ \partial \, (\sigma_{\max} \circ {\mathcal H}) }{ \partial y} ( {\rm i} \omega_r ; S^\red_r )  		\\
		&	=	\;
	\frac{ \partial \, (\sigma_{\max} \circ  {\mathcal H}_{r+1})} { \partial y}( {\rm i} \omega_r ; S^\red_r ) 	
		\;	=	\;
	\frac{\partial  {\mathcal F}_{r+1}}{\partial y} ( S^\red_r ) 
\end{split}
\end{equation}
for $y \in \{ x_1, x_2 \}$, where the second equality is due to (\ref{eq:Diffint_1}), as well as 
 (\ref{eq:Diffint_0}) implying the fact that
${\mathcal H}( {\rm i} \omega_r ; S^\red_r )$ and
${\mathcal H}_{r+1}( {\rm i} \omega_r ;  S^\red_r )$ have
the same left and right singular vectors. We remark that, for the first and third equalities 
above, we use the fact $\omega_r$ is a global maximizer of 
$\sigma(\omega ; S^{\red}_r) = \sigma_{\max}({\mathcal H}({\rm i} \omega ; S^\red_r))$ and
$\sigma_{r+1}(\omega ; S^{\red}_r) = \sigma_{\max}({\mathcal H}_{r+1}({\rm i} \omega ; S^\red_r))$,
respectively.

Moreover, for any $y, z \in \{ \omega, x_1, x_2 \}$, we have
\begin{equation*}
	\frac{ \partial^2 (\sigma_{\max} \circ {\mathcal H})}{ \partial y \, \partial z} (  {\rm i}\omega_r ; S^\red_r )  
					\;\;	=	\;\;
	\frac{ \partial^2 (\sigma_{\max} \circ {\mathcal H}_{r+1})}{ \partial y \, \partial z}  ( {\rm i} \omega_r ;  S^\red_r )  
\end{equation*}
due to (\ref{eq:Diffint_1}) and (\ref{eq:Diffint_2}) combined with the fact that
${\mathcal H}( {\rm i}  \omega_r ;  S^\red_r )$,
${\mathcal H}_{r+1}( {\rm i} \omega_r ; S^\red_r )$ have the same 
singular values and vectors due to  (\ref{eq:Diffint_0}). Consequently,
\begin{equation}\label{eq:int_sder}
\begin{split}
	 \frac{\partial^2  {\mathcal F}}{\partial y \: \partial z} ( S^\red_r ) 
				\;\;	& =	\;\;
	\frac{ \partial^2 ( \sigma_{\max} \circ {\mathcal H} ) }{ \partial y \: \partial z}
	( {\rm i} \omega_r ; S^\red_r ) 	\; 	+	\;
	 \frac{ \partial^2 (\sigma_{\max} \circ {\mathcal H})}{ \partial y \: \partial \omega}( {\rm i} \omega_r ; S^\red_r ) 
		\;	\times	\;  \\
	&	\;\;\;\;	\;\;\;	\left\{
	- \frac{\partial^2 (\sigma_{\max} \circ {\mathcal H})}{ \partial \omega \: \partial z }
		( {\rm i} \omega_r ; S^\red_r ) 	\bigg/
		\frac{ \partial^2 ( \sigma_{\max} \circ {\mathcal H}) }{ \partial^2 \omega }( {\rm i} \omega_r ; S^\red_r ) 		
	\right\}	\hskip 2ex			\\[1.5em]
				& =	\;
\frac{ \partial^2 (\sigma_{\max} \circ {\mathcal H}_{r+1}) }{ \partial y \: \partial z}( {\rm i} \omega_r ; S^\red_r )  \;  +  \;
	  \frac{ \partial^2 (\sigma_{\max} \circ {\mathcal H}_{r+1})}{ \partial y \: \partial \omega}( {\rm i} \omega_r ; S^\red_r )
		\;	\times	\;			\\
	&	\;\;\;\;\;\;\;	 \left\{
- \frac{\partial^2 (\sigma_{\max} \circ {\mathcal H}_{r+1})}{ \partial \omega \: \partial z }( {\rm i} \omega_r ; S^\red_r ) 	
			\bigg/
\frac{ \partial^2 ( \sigma_{\max} \circ {\mathcal H}_{r+1})}{ \partial^2 \omega }( {\rm i} \omega_r ; S^\red_r ) 		
	\right\}	\hskip -4ex	\\[1.4em]
		&	=	\;\;\;	 \frac{\: \partial^2  {\mathcal F}_{r+1}}{\partial y \: \partial z} ( S^\red_r ) 
\end{split}
\end{equation}
for any $y, z \in \{ x_1, x_2 \}$.


\section{A Quadratic Convergence Result Regarding Algorithm \ref{alg:SM}}\label{sec:quad_conv}

In this section, we establish a result that indicates a quadratic convergence regarding
the iterates of Algorithm \ref{alg:SM} under a few assumptions, especially smoothness
assumptions.


In this section and the next section, we denote with ${\mathcal D}^{\sr, m , p}$ the set of consisting 
of every descriptor system $S^\red$ of order $\sr$ and index at most one with semi-simple poles,
$m$ inputs, $p$ outputs.
Throughout this section, we make use of the vectorization ${\mathcal V}(S^{\red})$ of the system
$S^{\red} = (A^{\red}, E^{\red}, B^{\red}, C^{\red}, D^{\red})$ defined as
\begin{equation}\label{eq:vec_S}
	{\mathcal V}(S^{\red})
			:=	
	\big[	
	\;
			\text{vec}(A^{\red})^T  \;\;
			\text{vec}(E^{\red})^T  \;\;
			\text{vec}(B^{\red})^T  \;\;
			\text{vec}(C^{\red})^T  \;\;
			\text{vec}(D^{\red})^T
	\;
	\big]^T		,
\end{equation}
where $\text{vec}(M)$ denotes the vector obtained by stacking up the columns of matrix $M$.
The gradients $\nabla {\mathcal F} (S^{\red})$ and $\nabla {\mathcal F}_{r+1} (S^{\red})$
are vectors formed of the first partial derivatives of ${\mathcal F}(S^{\red})$ and ${\mathcal F}_{r+1}(S^{\red})$
based on the ordering of the variables, i.e., the entries of $A^{\red}, E^{\red}, B^{\red}, C^{\red}, D^{\red}$, 
in the vectorization in (\ref{eq:vec_S}). Similarly, 
$\nabla^2 {\mathcal F} (S^{\red})$ and $\nabla^2 {\mathcal F}_{r+1} (S^{\red})$ denote the Hessians
of ${\mathcal F}(S^{\red})$ and ${\mathcal F}_{r+1}(S^{\red})$ based on the ordering of the variables
according to (\ref{eq:vec_S}).

We assume that there are two consecutive iterates $S^{\red}_r$ and $S^{\red}_{r+1}$ of Algorithm \ref{alg:SM}
that are sufficiently close to a local maximizer $S^{\red}_\ast$ of ${\mathcal F}(S^{\red})$.
Moreover, we silently assume throughout that the interpolation properties in (\ref{eq:int_fder}) and (\ref{eq:int_sder})
hold at $S^{\red}_r$.
We also keep the assumption stated below that guarantees that ${\mathcal F}(S^{\red})$
is real analytic at $S^{\red}_\ast$. 
\begin{assumption}\label{ass:full}
The maximum of $\sigma(\omega ; S^{\red}_\ast)$ over all $\omega \geq 0$  is attained at a unique
$\omega_\ast$. Furthermore, 
$\sigma(\omega_\ast ; S^{\red}_\ast) = \sigma_{\max}( {\mathcal H}( {\rm i} \omega_\ast ; S^{\red}_\ast))$ 
is a simple singular value of ${\mathcal H}( {\rm i} \omega_\ast ; S^{\red}_\ast)$.
\end{assumption}
An assumption regarding the smoothness of ${\mathcal F}_{r+1}(S^{\red})$ that we rely on
is given next. Recalling $\| v \|_2$ for a vector $v$ denotes the 2-norm of $v$,
we make use of the distance
$\| \widetilde{S}^{\red} - \widehat{S}^{\red} \| := 
			\| {\mathcal V}(\widetilde{\mathcal S}^{\red}) - {\mathcal V}(\widehat{\mathcal S}^{\red}) \|$
for systems $\widetilde{S}^{\red}$, $\widehat{S}^{\red} \in {\mathcal D}^{\sr,m,p}$, and
the ball $B( \widehat{S}^{\red} , \delta) :=  
	\{ \widetilde{S}^\red \in {\mathcal D}^{\sr,m,p} \; | \;  \| \widetilde{S}^{\red} - \widehat{S}^{\red} \| < \delta  \}$ 
for a system $\widehat{S}^\red \in {\mathcal D}^{\sr,m,p}$ and positive real number $\delta$.
\begin{assumption}\label{ass:red}
$\;\;$ \\[-1.1em]
\begin{enumerate}
\item[\bf (i)]
For every $S^{\red} \in B( S^{\red}_{r} , \delta_r)$ 
with $\delta_r := \| S^{\red}_{r+1} - S^{\red}_r \|$
the following conditions hold:
\begin{itemize}
	\item The maximum of $\sigma_{r+1}(\omega ; S^{\red})$ over all $\omega \geq 0$  is attained at a unique
	$\overline{\omega}$.
	\item 
	The singular value
	$\sigma(\overline{\omega} ; S^{\red}) = \sigma_{\max}( {\mathcal H}_{r+1}( {\rm i} \overline{\omega} ; S^{\red}))$ 
	of ${\mathcal H}_{r+1}( {\rm i} \overline{\omega} ; S^{\red})$ is simple.
\end{itemize}
\item[\bf (ii)]
Moreover, all of the third derivatives of ${\mathcal F}_{r+1}(S^{\red})$ can be bounded by quantities independent
of $r$ at all $S^{\red} \in B( S^{\red}_{r} , \delta_r)$.
\end{enumerate}
\end{assumption}
We remark that part (i) of Assumption \ref{ass:red} guarantees that ${\mathcal F}_{r+1}(S^{\red})$ 
is real-analytic in the ball  $B( S^{\red}_{r} , \delta_r)$, and so three times differentiable in this ball,
which we depend on in part (ii) of Assumption \ref{ass:red}. 

We state and prove the main result that relates $\| S^{\red}_{r} - S^{\red}_\ast \|$
and $\| S^{\red}_{r+1} - S^{\red}_\ast \|$ below.
\begin{theorem}
Suppose that two consecutive iterates $S^{\red}_r$ and $S^{\red}_{r+1}$ of Algorithm \ref{alg:SM}
are sufficiently close to a local maximizer $S^{\red}_\ast$ of ${\mathcal F}(S^{\red})$. Moreover,
suppose Assumptions \ref{ass:full}, \ref{ass:red} hold, and $\nabla^2 {\mathcal F}(S^{\red}_\ast)$
is invertible. Then there is a constant $C$ independent of $r$ such that
\[
	\| S^{\red}_{r+1} - S^{\red}_\ast \|	\;\;	\leq		\;\;	C \cdot  \| S^{\red}_{r} - S^{\red}_\ast \|^2	\,	.
\]
\end{theorem}
\begin{proof}
By continuity $\sigma(\omega ; S^{\red}) = \sigma_{\max}({\mathcal H}({\rm i} \omega ; S^{\red}))$
remains simple at all $\omega$ and $S^\red \in {\mathcal D}^{\sr, m, p}$ such that $\omega$ is sufficiently close to $\omega_\ast$
and $S^\red$ is sufficiently close to $S^\red_\ast$, 
where $\omega_\ast$ is as in Assumption \ref{ass:full}. 
Thus, by the analytic implicit function theorem, there is $\widetilde{\delta} > 0$ such that
${\mathcal F}(S^{\red})$ is real analytic at all $S^{\red} \in B(S^{\red}_\ast, \widetilde{\delta})$
(see, e.g., \cite[Lemma 16]{Men2018} for the details in the analogous context of the distance 
instability). 
By the assumption that  $\nabla^2 {\mathcal F}(S^{\red}_\ast)$ is invertible, and continuity of the second 
partial derivatives of ${\mathcal F}(S^{\red})$ in $B(S^{\red}_\ast, \widetilde{\delta})$,
the Hessian $\nabla^2 {\mathcal F}(S^{\red})$ remains invertible in a ball $B(S^{\red}_\ast, \delta)$
for some $\delta < \widetilde{\delta}$.
Furthermore, without loss of generality, 
we assume $S^{\red}_{r}, S^{\red}_{r+1}$ are close enough to $S^{\red}_\ast$ so that
$S^{\red}_{r}, S^{\red}_{r+1} \in B(S^{\red}_\ast, \delta)$, and the ball $B(S^{\red}_r, \delta_r)$ 
in Assumption \ref{ass:red} is contained in $B(S^{\red}_\ast, \delta)$.
We let $\beta := \min_{S^{\red} \in B(S^{\red}_\ast, \delta)} \, \sigma_{\min}(\nabla^2 {\mathcal F}(S^{\red})) > 0$,
and note that $\nabla^2 {\mathcal F}(S^{\red})$ is Lipshitz continuous in $B(S^{\red}_\ast, \delta)$, say with the
Lipschitz constant $\gamma$.

By an application of Taylor's theorem with integral remainder, we have
\begin{equation}\label{eq:Taylor_int}
	\begin{split}
	0	 \;	&	=	\;	 
	\nabla {\mathcal F}(S^{\red}_\ast)	\\[.3em]
			&	 = 	\;
	\nabla {\mathcal F}(S^{\red}_r)
		  \; +  \;
	\int_0^1 \nabla^2 {\mathcal F}(S^{\red}_r + t (S^{\red}_\ast - S^{\red}_r)) \, 
			({\mathcal V}(S^{\red}_\ast) - {\mathcal V}(S^{\red}_r)) \, {\rm d} t		\\[.3em]
		&	 = 	\;
	\nabla {\mathcal F}(S^{\red}_r)
		 + 
	\nabla^2 {\mathcal F}(S^{\red}_r) ({\mathcal V}(S^{\red}_\ast) - {\mathcal V}(S^{\red}_r))
		 + 
	{\mathcal O}( \|  S^{\red}_\ast -  S^{\red}_r \|^2 )	\,	,
	\end{split}
\end{equation}
where, for the third equality, we have used the Lipschitz continuity of  
$\nabla^2 {\mathcal F}(S^{\red})$ in $B(S^{\red}_\ast, \delta)$. Additionally,
by Taylor's theorem with second order Lagrange remainder, 
\begin{equation}\label{eq:rconv_inter}
\begin{split}
	0	&	\:	=	\:	\nabla {\mathcal F}_{r+1}(S^{\red}_{r+1})	\\[.3em]
		&	\;	=	\;	
		\nabla {\mathcal F}_{r+1}(S^{\red}_r)
		\: + \:
		\nabla^2 {\mathcal F}_{r+1}(S^{\red}_r) ({\mathcal V} (S^{\red}_{r+1})  -  {\mathcal V} (S^{\red}_r))
		\: + \:
		{\mathcal O}( \|  S^{\red}_{r+1} -  S^{\red}_r \|^2 )	\\[.3em]
		&	\;	=	\;	
		\nabla {\mathcal F}(S^{\red}_r)
		\: + \:
		\nabla^2 {\mathcal F}(S^{\red}_r) ({\mathcal V}(S^{\red}_{r+1}) - {\mathcal V}(S^{\red}_r))
		\: + \:
		{\mathcal O}( \|   S^{\red}_{r+1}   -   S^{\red}_r \|^2 )	\,	,
\end{split}
\end{equation}
where the third equality is due to $\nabla {\mathcal F}_{r+1}(S^{\red}_r) = \nabla {\mathcal F}(S^{\red}_r)$
and $\nabla^2 {\mathcal F}_{r+1}(S^{\red}_r) = \nabla^2 {\mathcal F}(S^{\red}_r)$,
that are consequences of (\ref{eq:int_fder}) and (\ref{eq:int_sder}).

By employing (\ref{eq:rconv_inter}) in (\ref{eq:Taylor_int}), we deduce
\[
	\nabla^2 {\mathcal F}(S^{\red}_r) ({\mathcal V}(S^{\red}_\ast) - {\mathcal V}(S^{\red}_{r+1}))
		\;	=	\;  	{\mathcal O}( \| S^{\red}_{r+1} -  S^{\red}_r \|^2 )
					+	
					{\mathcal O}( \| S^{\red}_\ast -  S^{\red}_r \|^2 ).
\]
Finally, noting 
$\| \nabla^2 {\mathcal F}(S^{\red}_{r}) ({\mathcal V}(S^{\red}_\ast) -  {\mathcal V}(S^{\red}_{r+1})) \|
		\geq		\beta		\| S^{\red}_\ast -  S^{\red}_{r+1} \|$, from the last equality we obtain 
\[
	\| S^{\red}_{r+1} - S^{\red}_\ast \|	\;	\leq		\;
								{\mathcal O}( \| S^{\red}_r -  S^{\red}_\ast \|^2 )
\]
as desired.
\end{proof}

\section{Dealing with Asymptotic Stability Constraints}\label{sec:imp_stab}
In many applications, the reduced order system sought $S^\red = (A^\red, E^\red, B^\red, C^\red, D^\red)$
of order $\sr$ not only is close with respect to the ${\mathcal L}_\infty$ norm, but may also be required
to be asymptotically stable. As we search through reduced order systems of index at most one,
the asymptotic stability requirement is equivalent to the condition $\alpha (A^\red, E^\red) < 0$,
where $\alpha(A^\red, E^\red)$ is the spectral abscissa of the pencil $L(s) = A^\red - s E^\red$
defined by
\[
		\alpha (A^\red, E^\red)
			\;	:=	\;
		\max
		\left\{
				\text{Re}(z) \; | 	\;	z \in {\mathbb C} \;\, \text{ s.t. } \det(A - zE) = 0		
		\right\}.
\]
In this setting, with ${\mathcal F}(S^\red)$ defined as in (\ref{eq:objective2}), 
rather than the unconstrained minimization of ${\mathcal F}(S^\red)$ over all 
descriptor systems $S^\red \in {\mathcal D}^{\sr , m , p}$, it may be desirable to solve the
constrained minimization problem
\begin{equation}\label{eq:constrained_problem}
	\min
	\left\{
		{\mathcal F}(S^\red) 	\;	:	\;
		 S^\red \in {\mathcal D}^{\sr , m , p} \; \text{ s.t. }  \;  \alpha (A^\red, E^\red) \leq \beta	
	\right\} 
\end{equation}
for a prescribed negative real number $\beta$,
where ${\mathcal D}^{\sr , m , p}$ denotes the set of all descriptor systems of order $\sr$ and index at most one
with semi-simple poles, $m$ inputs, $p$ outputs.

Extension of the proposed subspace framework, that is Algorithm \ref{alg:SM}, to deal
with the constrained minimization problem in (\ref{eq:constrained_problem}) rather than
the unconstrained minimization of ${\mathcal F}(S^\red)$ is straightforward.
The only difference in Algorithm \ref{alg:SM} is in line \ref{exp_start}, where
$S^\red_r$ is no longer a minimizer of ${\mathcal F}_r(S^\red)$, but instead
a minimizer of the constrained problem
\begin{equation}\label{eq:constrained_problem_reduced}
	\min
	\left\{
		{\mathcal F}_r(S^\red) 	\;	:	\;
		 S^\red \in {\mathcal D}^{\sr , m , p} \; \text{ s.t. }  \;  \alpha (A^\red, E^\red) \leq \beta	
	\right\} 
\end{equation}
for the reduced function ${\mathcal F}_r(S^\red)$ as in (\ref{eq:reduced_objective}).
The problem in (\ref{eq:constrained_problem_reduced}) involves only the small systems
$S_r$ as well as $S^\red$, and is solvable by means of Newton-method based 
approaches. Such a Newton-method based approach makes use of the gradient of the
constraint function ${\mathcal C}(S^\red) :=  \alpha (A^\red, E^\red) - \beta$,
in addition to the gradient of the objective ${\mathcal F}_r(S^\red)$.
Let $\lambda$ be the rightmost eigenvalue of
the pencil $L(s) = A^\red - s E^\red$ with $u$ and $v$ denoting a pair of corresponding
left and right eigenvectors normalized so that $u^\ast E^\red v = 1$, and assume 
$\lambda$ is a simple eigenvalue and the unique rightmost eigenvalue of $L(s)$,
which ensures the differentiability of ${\mathcal C}(S^\red)$. Then, by differentiating
the equation $A^\red v = \lambda E^{\red} v$ with respect to the entries of $A^\red$
and $E^\red$ and multiplying with $u^\ast$ from left,
the partial derivatives of ${\mathcal C}(S^\red)$ are given by
\begin{equation*}
	\frac{\partial {\mathcal C}}{A^\red_{ij}}(S^\red)	\;	=	\Re( \overline{u}_i v_j ) \; ,
	\quad\quad
	\frac{\partial {\mathcal C}}{E^\red_{jj}}(S^\red)	\;	=	-\Re( \lambda \overline{u}_j v_j ) \; ,
\end{equation*}
where $A^\red_{ij}$ is the subdiagonal, superdiagonal or diagonal entry of the matrix variable $A^\red$
at position $(i,j)$, and $E^\red_{jj}$ is the diagonal entry of $E^\red$ at position $(j, j)$.


We remark that, assuming $\omega_r$ is again a global minimizer of $\sigma_{r+1}(\omega; S^\red_r)$
after the refinement step, the interpolation properties 
\[
	{\mathcal F}(S^\red_r) = {\mathcal F}_{r+1}(S^\red_r)	\;	,	\;\;
	\nabla {\mathcal F}(S^\red_r) = \nabla {\mathcal F}_{r+1}(S^\red_r)	\;	,	\;\;	\text{and}   \;\; 
	\nabla^2 {\mathcal F}(S^\red_r) = \nabla^2 {\mathcal F}_{r+1}(S^\red_r)
\]
still hold. Moreover, if a logarithmic barrier approach is adopted for the solution
of the constrained problems, then, in essence, 
constrained problems are turned into unconstrained problems by lifting the
constraints to the objective via the logarithmic barrier functions
\begin{equation*}
\begin{split}
	L^\mu(S^\red) 		\:	 & = 		\:	{\mathcal F}(S^\red)
									\, - \,	\mu \cdot  \log 	(\beta - \alpha (A^\red, E^\red)) \; , 	\\
	L^\mu_r(S^\red) 		\:	& = 		\:	{\mathcal F}_r(S^\red)
									\, - \,	\mu \cdot  	\log (\beta - \alpha (A^\red, E^\red))
\end{split}
\end{equation*}
associated with problems (\ref{eq:constrained_problem}), (\ref{eq:constrained_problem_reduced}),
respectively, where $\log(\cdot)$ denotes the natural logarithm of its parameter,
$\mu$ is a positive real number and represent the barrier parameter.
In this case, the interpolation properties extend to the logarithmic barrier functions as well.
In particular, we have
\[
	L^\mu(S^\red_r) = L^\mu_{r+1}(S^\red_r)	\,	,	\;\,
	\nabla L^\mu(S^\red_r) = \nabla L^\mu_{r+1}(S^\red_r)	\,	,	\;\,	\text{and}   \;\,
	\nabla^2 L^\mu(S^\red_r) = \nabla^2 L^\mu_{r+1}(S^\red_r)
\]
for every positive real number $\mu$.

\section{Practical Issues}\label{sec:prac_issues}
Here we spell out a few practical issues regarding Algorithm \ref{alg:SM}
such as how we form the initial systems $S_0$, $S_0^\red$, when we
terminate, the details of bases for projection subspaces,
solutions of reduced ${\mathcal L}_\infty$-norm minimization problems,
and ${\mathcal L}_\infty$-norm computations. 

\subsection{Initialization}
The initial subspaces ${\mathcal V}_ 0, {\mathcal W}_0$ (in line \ref{init_subspaces} of Algorithm \ref{alg:SM})
are chosen so that $H_{0}$, the transfer function of $S_0$, interpolates $H$ at the imaginary parts of a 
prescribed number of dominant poles of $H$.
Formally, for a prescribed $\ell$, let $s_1, \dots, s_\ell \in {\mathbb C}$ be the most dominant $\ell$ 
poles of $H$ with nonnegative imaginary parts (i.e., only the dominant poles with nonnegative
imaginary parts are taken into consideration, as the poles of $H$ appear in complex conjugate pairs such that
any two complex conjugate poles have the same dominance metric), we set
\begin{equation*}
	\begin{split}
	&	{\mathcal V}_0		\;\;	=	\;\;		\bigoplus_{k = 1}^\ell 	
											\bigoplus_{j=0}^1  
												\bigg\{		\Re \left[ \{ (A - {\rm i} \Im s_k E)^{-1} E \}^j (A - {\rm i}\Im s_k E)^{-1} B	 \right] 		\\[-.7em]
	&		\hskip 27ex				\bigoplus	\;	\Im \left[ \{ (A - {\rm i} \Im s_k E)^{-1} E \}^j (A - {\rm i} \Im s_k E)^{-1} B	 \right] 		
																\bigg\},			\\
	&	{\mathcal W}_0		\;\;	=	\;\;		\bigoplus_{k = 1}^\ell 	
											\bigoplus_{j=0}^1  
											\bigg\{		\Re  \left[ C  (A - {\rm i} \Im s_k E)^{-1}  \{ E (A - {\rm i} \Im s_k E)^{-1}  \}^j  \right]^\ast 		\\[-.7em]
	&		\hskip 27ex				\bigoplus	\;	\Im  \left[ C  (A - {\rm i} \Im s_k E)^{-1}  \{ E (A - {\rm i} \Im s_k E)^{-1}  \}^j  \right]^\ast 	
																	\bigg\}.								
	\end{split}
\end{equation*}
Theorem \ref{thm:interpolate} ensures that
\begin{equation*}
	\begin{split}
		&	H({\rm i} \Im s_k)	=	H_0({\rm i} \Im s_k),	
				\quad		H^{(j)}({\rm i} \Im s_k)	=	H^{(j)}_0({\rm i} \Im s_k)			\\
		&	H(-{\rm i} \Im s_k)	=	H_0(-{\rm i} \Im s_k),	
				\quad	\text{and}		\quad	H^{(j)}(-{\rm i} \Im s_k)	=	H^{(j)}_0(-{\rm i} \Im s_k)
	\end{split}
\end{equation*}
for $j = 1, 2, 3$ and $k = 1,\dots, \ell$.

Additionally, at every subspace iteration with $r > 0$,
an initial point is needed for the solution of the minimization problem 
in line \ref{exp_start} of Algorithm \ref{alg:SM} regardless of how it is solved, e.g., via gradient 
descent or BFGS. This initialization carries significance, as it affects which local minimizer 
of ${\mathcal F}_r$ is to be converged. At a subspace iteration with $r > 0$, 
the minimizer is initialized with the optimal reduced system from the previous iteration, 
that is with $S^\red_{r-1}$. Initial $S^{\red}_0$ of order $\sr$ must be supplied to Algorithm \ref{alg:SM}.
We set $S^{\red}_0$ as either
\vskip .5ex
	\begin{itemize}
		\item the model of order $\sr$ obtained from an application of the balanced truncation approach, or 
		\smallskip
		\item the model of order $\sr$ whose transfer function interpolates $H$ at the imaginary parts of a prescribed number 
		of dominant poles of $H$.
	\end{itemize}
\vskip .5ex
For the latter choice, we remark that the number of dominant poles used to form $S_0$ is strictly greater than the
number of dominant poles used to form this initial model $S^{\red}_0$ for the minimizer. For either choice, we make sure
$\, \text{dim} \: {\mathcal V}_0 \, = \, \text{dim} \: {\mathcal W}_0 \, > \, \sr \,$ by using sufficiently many dominant
poles of $H$ when forming $S_0 = (A_0, E_0, B_0, C_0, D)$. 

An issue that requires attention is that 
$S^{\red}_0 = (A^{\red}_0, E^{\red}_0, B^{\red}_0, C^{\red}_0, D^{\red}_0)$ must be such that 
$A^{\red}_0$ is tridiagonal and $E^{\red}_0$ is diagonal, whereas the balanced truncation or
the interpolatory approach yields the system $(\widehat{A}, \widehat{E} , \widehat{B}, \widehat{C}, \widehat{D})$
of order $\sr$ such that $\widehat{A}$ and $\widehat{E}$ do not necessarily have these structures. 
Let us suppose $\widehat{E}$ is invertible.
Then we can first compute the eigenvalues of the $\sr \times \sr$ pencil $\widehat{L}(s) = \widehat{A} - s \widehat{E}$, and form
a block diagonal real matrix $T_2$ with $2\times 2$ and $1\times 1$ blocks along its diagonal that have 
the same eigenvalues as $\widehat{L}$. The $2\times 2$ and $1\times 1$ blocks of $T_2$ on its diagonal
correspond to a conjugate pair of complex eigenvalues and real eigenvalues of $\widehat{L}$, respectively.
Here we remark that $T_2$ is an $\sr \times \sr$ matrix.
Hence, we can compute its eigenvalue decomposition
\[
		T_2	=	V_2	\Lambda V_2^{-1}
\]
for a diagonal matrix $\Lambda$ and invertible $V_2$ at ease. We 
also have the eigenvalue decomposition
\[
			\widehat{E}^{-1} \widehat{A}	=	V	\Lambda V^{-1}
\]
at hand. Note that the middle factors in eigenvalue decompositions above are the same, as
$T_2$ has the same eigenvalues as the pencil $\widehat{L}$, which in turn has the
same eigenvalues as $\widehat{E}^{-1} \widehat{A}$. But then 
\begin{equation*}
	\begin{split}
	\widehat{H}(s)	:=	\widehat{C} (s \widehat{E}  -   \widehat{A} )^{-1} \widehat{B} + \widehat{D}
				&		=	\widehat{C} (s I  -   \widehat{E}^{-1} \widehat{A} )^{-1} \widehat{E}^{-1} \widehat{B} 
									+ \widehat{D}	 	\\
				&		=	\widehat{C} (s I  -   V	\Lambda V^{-1} )^{-1} \widehat{E}^{-1} \widehat{B} + \widehat{D}		\\
				&		=	( \widehat{C} V ) (s I  -   	\Lambda  )^{-1}  ( V^{-1} \widehat{E}^{-1} \widehat{B} ) 
									+ \widehat{D}		\\
				&		=	( \widehat{C} V ) (s I  -   	V_2^{-1} T_2 V_2  )^{-1}  ( V^{-1} \widehat{E}^{-1} \widehat{B} ) + \widehat{D}		\\
				&		=	( \widehat{C} V V_2^{-1} ) (s I  -   T_2  )^{-1}  ( V_2 V^{-1} \widehat{E}^{-1} \widehat{B} ) + \widehat{D}.
	\end{split}
\end{equation*}
Hence, we can use 
\[
	A^{\red}_0 := T_2, \;\;	E^{\red}_0 := I,		\;\;	
					B^\red_0 := V_2 V^{-1} \widehat{E}^{-1} \widehat{B},	\;\;	C^\red_0 :=  \widehat{C} V V_2^{-1},  \;\;
	D^{\red}_0 := \widehat{D}_r
\] 
as the matrices of the initial system $S^\red_0$.

\subsection{Termination}\label{sec:termination}
The termination in line \ref{return_spec} of Algorithm \ref{alg:SM} is determined based on the values of 
$\| H(\cdot) - H(\, \cdot \,\, ; \, S_r^\red ) \|_{{\mathcal L}_\infty}$ at two consecutive subspace iterations. The error
$\| H(\cdot) - H(\, \cdot \,\, ; \, S_r^\red ) \|_{{\mathcal L}_\infty}$ is readily available at the $r$th subspace iteration
after line \ref{large_linfinity}, as 
\[
	\| H(\cdot) - H(\, \cdot \,\, ; \, S_r^\red ) \|_{{\mathcal L}_\infty}  	
			\;	=	\; 	 
	\sigma_{\max}	\left(	H({\rm i} \omega_r)	-	H({\rm i} \omega_r ; S^\red_r) \right).
\]
To be precise, we terminate at the $r$th subspace iteration in line \ref{return_spec} if $r \geq 1$ and
\begin{equation}\label{eq:term_cond}
		\left| 		\| H(\cdot) - H( \cdot  ;  S_r^\red ) \|_{{\mathcal L}_\infty}		
				-	
		\| H(\cdot) - H( \cdot  ;  S_{r-1}^\red ) \|_{{\mathcal L}_\infty}		\right|
					\leq		
		\texttt{tol} 			\| H(\cdot) - H( \cdot  ;  S_r^\red ) \|_{{\mathcal L}_\infty}	
\end{equation}
for a prescribed tolerance $\texttt{tol}$.

The termination condition for the minimizer to solve the minimization problem in line \ref{exp_start} of Algorithm \ref{alg:SM}
also requires some care. Recall that the objective ${\mathcal F}_r$ here is nonsmooth, and, as a result, the norms of the 
gradients of ${\mathcal F}_r$ at the iterates generated by the minimizer do not have to converge to zero. Instead, 
the minimizer is terminated if the line-search fails (to return a point that causes sufficient decrease), or the
decrease in ${\mathcal F}_r$ at two consecutive iterates is less than $\varepsilon \cdot \texttt{tol}$ in a relative
sense, where \texttt{tol} is as in (\ref{eq:term_cond}) and $\varepsilon$ is a real number in $(0,0.5)$.

As for the termination condition of the refinement step (i.e., the condition in line \ref{refine_if} of 
Algorithm \ref{alg:refine_step}) employed in practice, we rely on
\[
	\big|
		\omega^{(j)}_r	-	\omega_r
	\big|
		\;\;	\leq		\;\;	\texttt{tol} \left| 	   \omega_r 	   \right|	\,
\]
where $\texttt{tol}$ is again the tolerance in (\ref{eq:term_cond}).

\subsection{Orthonormalization of the Bases for the Subspaces}
Keeping the bases for the subspaces ${\mathcal V}_r$, ${\mathcal W}_r$ (i.e., the columns of $V_r$, $W_r$) 
orthonormal improves the robustness of the algorithm against the rounding errors. For instance, then
the system matrices $A_r, B_r, C_r, E_r$ can be formed more accurately in the presence of rounding
errors. 

This orthonormality property of the bases is attained in line \ref{R_orthogonalize} of Algorithm \ref{alg:SM},
as well as line \ref{R_orthogonalize_2} of Algorithm \ref{alg:refine_step}.
In line \ref{R_orthogonalize} of Algorithm \ref{alg:SM}, $V_r$ and $W_r$ are already 
orthonormal bases for ${\mathcal V}_r$ and ${\mathcal W}_r$.
The expansion directions $\widetilde{V}_{r+1}$, $\widetilde{W}_{r+1}$ to be included in the subspaces to 
obtain the expanded subspaces ${\mathcal V}_{r+1}$, ${\mathcal W}_{r+1}$ has to be orthonormalized
with respect to the existing orthonormal bases $V_{r}$, $W_{r}$. This is achieved
in practice by executing
\begin{equation}\label{eq:orthonormalize}
		\widetilde{V}_{r+1}	\;	\gets		\;	\widetilde{V}_{r+1}	-	V_r  (V_r^T \widetilde{V}_{r+1}	)	
				\quad\quad		\text{and}			\quad\quad
		\widetilde{W}_{r+1}	\;	\gets		\;	\widetilde{W}_{r+1}	-	W_r  (W_r^T \widetilde{W}_{r+1} ).		
\end{equation} 
Near convergence the interpolation points ${\rm i} \omega_r$ start not changing by much in consecutive
iterations. This results in the new expansion directions $\widetilde{V}_{r+1}$, $\widetilde{W}_{r+1}$ 
that are nearly contained in the existing subspaces ${\mathcal V}_r$, ${\mathcal W}_r$. In this
case, the orthonormalization in (\ref{eq:orthonormalize}) of $\widetilde{V}_{r+1}$, $\widetilde{W}_{r+1}$
with respect to existing $V_{r}$, $W_{r}$ suffers from cancellation type rounding errors.
Applying the orthonormalization in (\ref{eq:orthonormalize}) several times improves the accuracy,
and result in directions $\widetilde{V}_{r+1}$, $\widetilde{W}_{r+1}$ that are better orthonormalized
against $V_{r}$, $W_{r}$. In practice, we apply (\ref{eq:orthonormalize}) a few times, e.g.,
3-4 times, then orthonormalize the resulting $\widetilde{V}_{r+1}, \widetilde{W}_{r+1}$ via the Gram-Schmidt 
procedure, and take 
$V_{r+1}
		=
	\big[	
			V_r		\;\;		\widetilde{V}_{r+1}
	\big]
$,
$\, W_{r+1}
		=
	\big[	
			W_r		\;\;		\widetilde{W}_{r+1}
	\big]
$ 	
as the matrices whose columns form orthonormal bases for ${\mathcal V}_{r+1}$, ${\mathcal W}_{r+1}$.
In line \ref{R_orthogonalize_2} of Algorithm \ref{alg:refine_step}, the columns of $V_{r+1}$ and $W_{r+1}$ 
are similarly orthonormalized.
We ultimately use $V_{r+1}, W_{r+1}$ when forming the system $S_{r+1}$ 
in line \ref{form_rp1} of Algorithm \ref{alg:SM}.

\subsection{Solution of the Reduced ${\mathcal L}_\infty$-Norm Minimization Problem}
We use BFGS to minimize the reduced ${\mathcal L}_\infty$ objective 
${\mathcal F}_r(S^{\red})$ in line \ref{exp_start} of Algorithm \ref{alg:SM}.
To be precise, we have explored two alternatives here; a small variation of
a Matlab implementation of a line-search BFGS due to Michael L. Overton
making use of weak Wolfe conditions, and GRANSO \cite{CurMO2017}.
The former is only meant for unconstrained problems when we do not 
impose the asymptotic stability constraints described in Section \ref{sec:prac_issues},
whereas the asymptotic stability constraints in Section \ref{sec:prac_issues}
are also incorporated into this optimization when we use GRANSO. 

\subsection{Computation of the ${\mathcal L}_\infty$ Norm}
Algorithm \ref{alg:SM} in line \ref{large_linfinity} 
requires the computation of the ${\mathcal L}_\infty$ norm of a system whose order
is the sum of the order of the original system $S$ and $\sr$. If the original system does
not have large order, we use the built-in \texttt{norm} command in Matlab for these
${\mathcal L}_\infty$-norm computations. Otherwise, if the original system has large
order, we use the subspace framework introduced in \cite{Aliyev2017} for the
large-scale ${\mathcal L}_\infty$-norm computations. Additionally, the minimization
of the reduced ${\mathcal L}_\infty$ objective in line \ref{exp_start} via BFGS
requires small-scale ${\mathcal L}_\infty$-norm computations, which
we carry out using the \texttt{norm} command in Matlab.

\section{Numerical Results}\label{sec:numerical_exp}
In this section, we report the results of numerical experiments performed
with a Matlab implementation of Algorithm \ref{alg:SM} taking also the practical
issues indicated in the previous section into account. The first two subsections
\S\ref{sec:num_exp1} and \S\ref{sec:num_exp2} below concern experiments on 
rather smaller order systems, \S\ref{sec:num_exp3}
concerns experiments on a system of medium order, while the results of
experiments on several large-order systems are reported in \S\ref{sec:num_exp4}. 
All of the experiments are conducted in Matlab 2020b on an an 
iMac with Mac OS~12.1 operating system, Intel\textsuperscript{\textregistered} 
Core\textsuperscript{\texttrademark} i5-9600K CPU and 32GB RAM.

The numerical experiments are performed using the variation of the Matlab 
implementation of BFGS due to Michael L. Overton for the solution of the reduced 
${\mathcal L}_\infty$-norm minimization problems. Hence, the asymptotic stability
constraints are not imposed. The original systems in all of the experiments
in \S\ref{sec:num_exp1}-\ref{sec:num_exp3} concerning small- to medium-order
systems are asymptotically stable, and the computed optimal reduced 
systems in these examples also turn out to be asymptotically stable. The tolerance \texttt{tol} for termination 
(discussed in \S\ref{sec:termination}) is set equal to $10^{-8}$ in \S\ref{sec:num_exp1}-\ref{sec:num_exp3}, 
and $10^{-6}$ in \S\ref{sec:num_exp4}. 

For comparison or initialization purposes, some of the numerical experiments
involve the application of the balanced truncation for which we use the Matlab
toolbox MORLAB \cite{morBenW19b}, in particular the routine 
\texttt{ml}$\_$\texttt{ct}$\_$\texttt{dss}$\_$\texttt{bt} or
\texttt{ml}$\_$\texttt{ct}$\_$\texttt{ss}$\_$\texttt{bt} depending on
whether the system at hand is a descriptor system or more specifically
a linear time-invariant system. Moreover, the Hankel singular values
computed for smaller systems for comparison purposes are retrieved
by calling the built-in routine \texttt{hankelsv} in Matlab. As the first
three subsections concern the model reduction of relatively smaller systems,
the built-in routine \texttt{norm} is employed in line \ref{large_linfinity}
of Algorithm \ref{alg:SM} for ${\mathcal L}_\infty$-norm computations,
while the subspace framework in \cite{Aliyev2017} is employed for this purpose in
\S\ref{sec:num_exp4} that concerns the model reduction of descriptor systems with large order.

\subsection{ISS Example}\label{sec:num_exp1}
We start with the \texttt{iss} example of order $n = 270$
that is also considered when optimizing the objective 
${\mathcal F}$ directly in Section \ref{sec:direct_optimize}.
As before, we seek the nearest reduced descriptor system of order $\sr = 12$
with respect to the ${\mathcal L}_\infty$ norm.
An application of Algorithm \ref{alg:SM} with the initial estimate $S^{\red}_0$ 
produced by the balanced truncation
terminates when $r = 6$. The error ${\mathcal F}(S^\red_\star) = 0.0022516$
of the estimate $S^\red_\star$ returned is nearly half of the error ${\mathcal F}(S^\red_0) = 0.0044701$
of the initial estimate $S^\red_0$. The optimal reduced system $S^{\red}_\star$ is indeed slightly better than
the estimate returned by the direct optimization at which the objective ${\mathcal F}$ takes
the value $0.0024154$. Yet the total runtime is about 66 seconds, much shorter than 500
seconds, roughly the time required by the direct optimization. The local optimality of 
$S^\red_\star = (A^\red_\star , B^\red_\star , C^\red_\star , D^\red_\star , E^\red_\star)$
is apparent from Figure \ref{fig:sf_gd_converge}, which indicates an increase in the objective
${\mathcal F}$ if one of the entries of one of $A^\red_\star , B^\red_\star , C^\red_\star , D^\red_\star , E^\red_\star$
is modified. Moreover, the Hankel singular value $\sigma_{\sr+1}$ for this example,
a lower bound for the minimal error possible for any system of order $\sr$, is  
$0.0022353$ smaller than ${\mathcal F}(S^\red_\star) = 0.0022516$ only by a slim margin, 
so $S^\red_\star$ must be nearly optimal globally as well.

\begin{figure} 
 \centering
		\begin{tabular}{cc}
			 \subfigure[$(1,1)$, $A^{\red}$ \label{fig:1a_sf}]{\includegraphics[width = .47\textwidth]{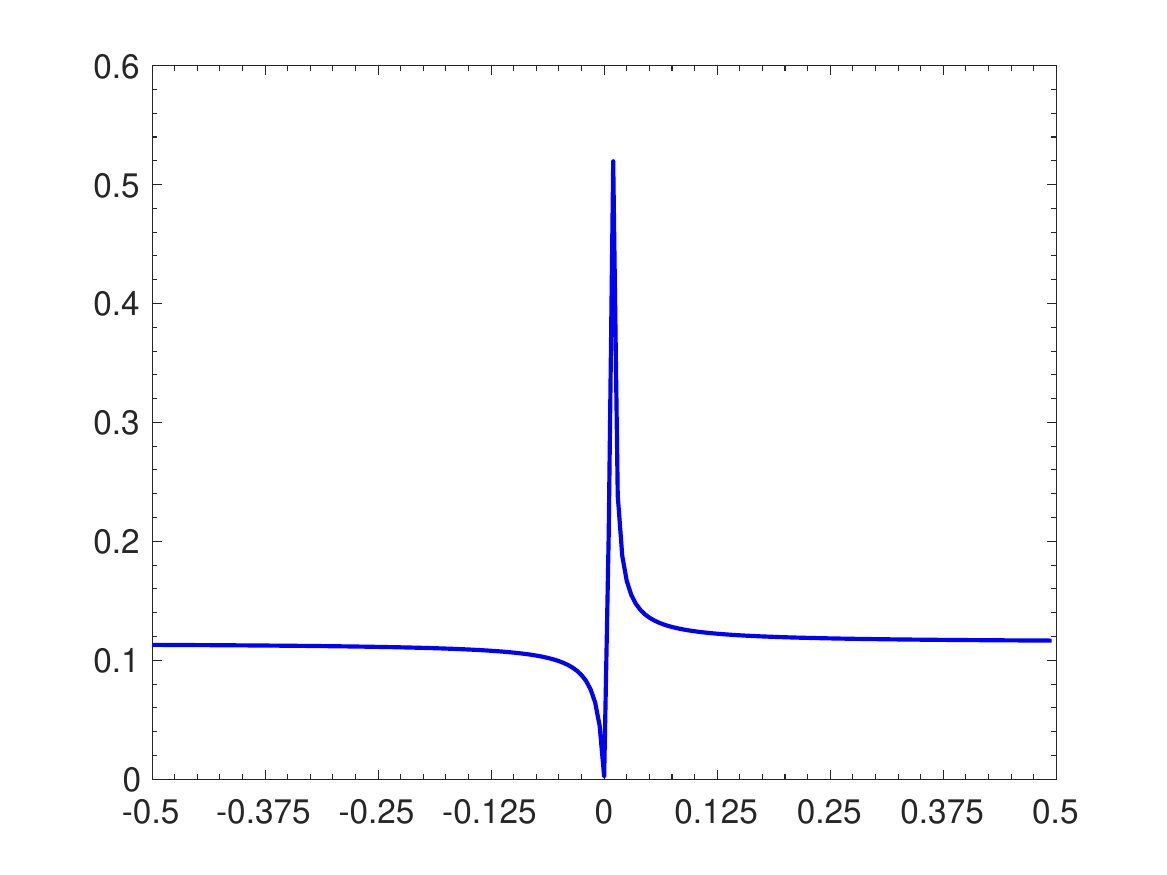}} & 		 
			 \subfigure[$(11,10)$, $A^{\red}$ \label{fig:1b_sf}]{\includegraphics[width = .47\textwidth]{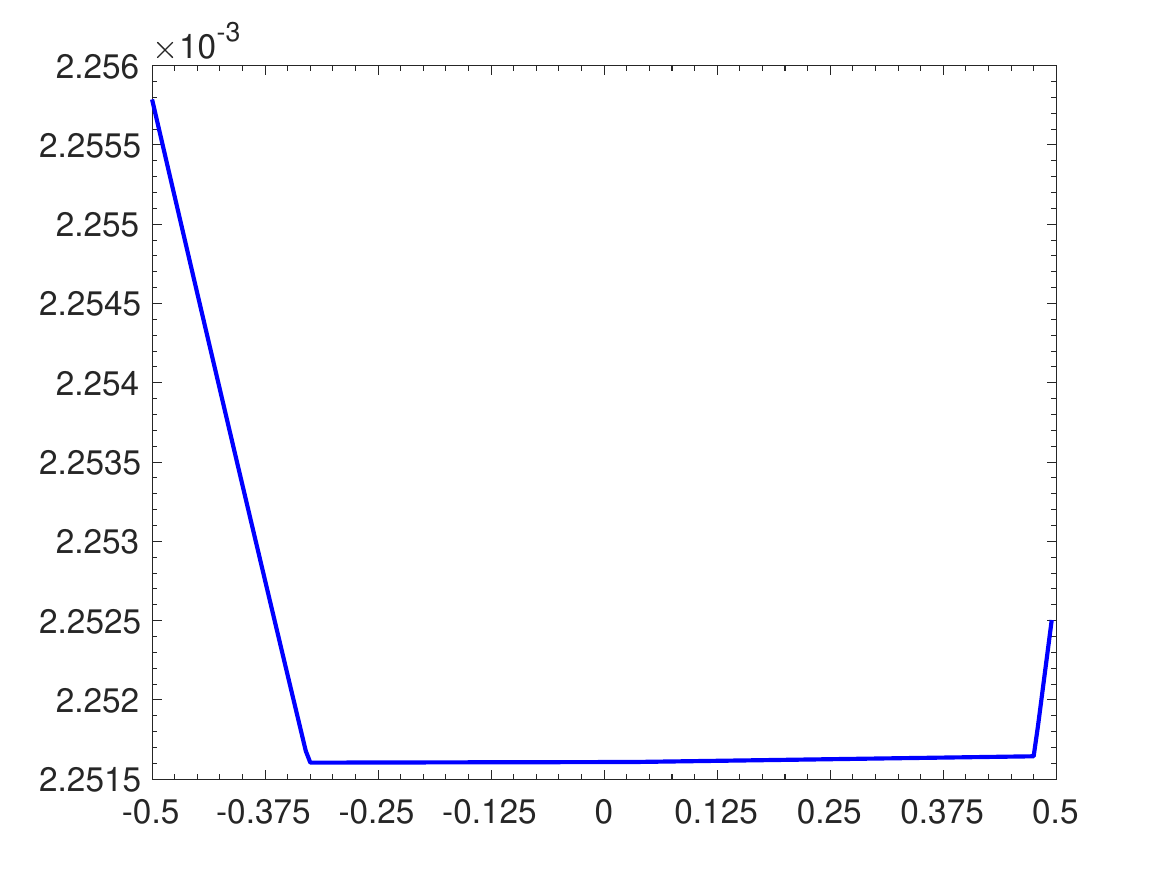}}
			 \\  
			 \subfigure[$(6,7)$, $A^{\red}$ \label{fig:1c_sf}]{\includegraphics[width = .47\textwidth]{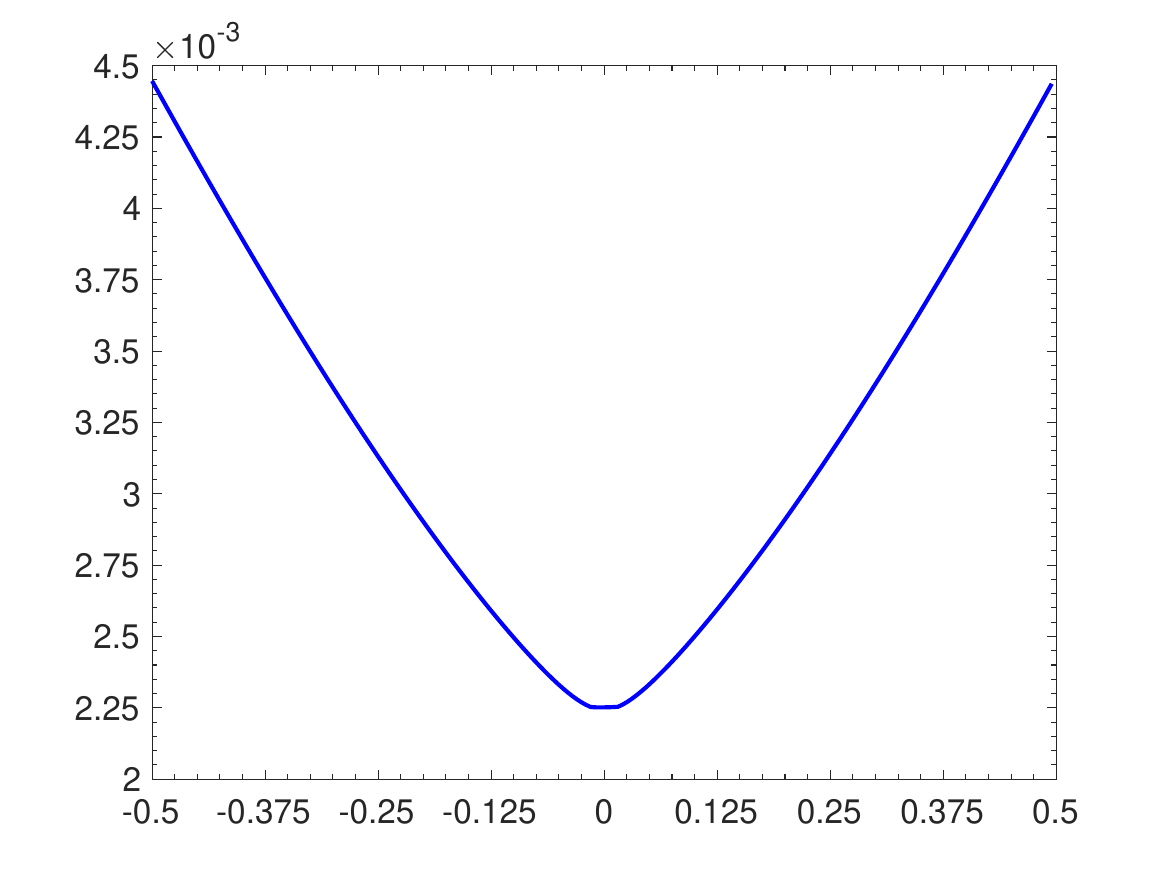}} &			 
			 \subfigure[$(1,1)$, $E^{\red}$ \label{fig:1d_sf}]{\includegraphics[width = .47\textwidth]{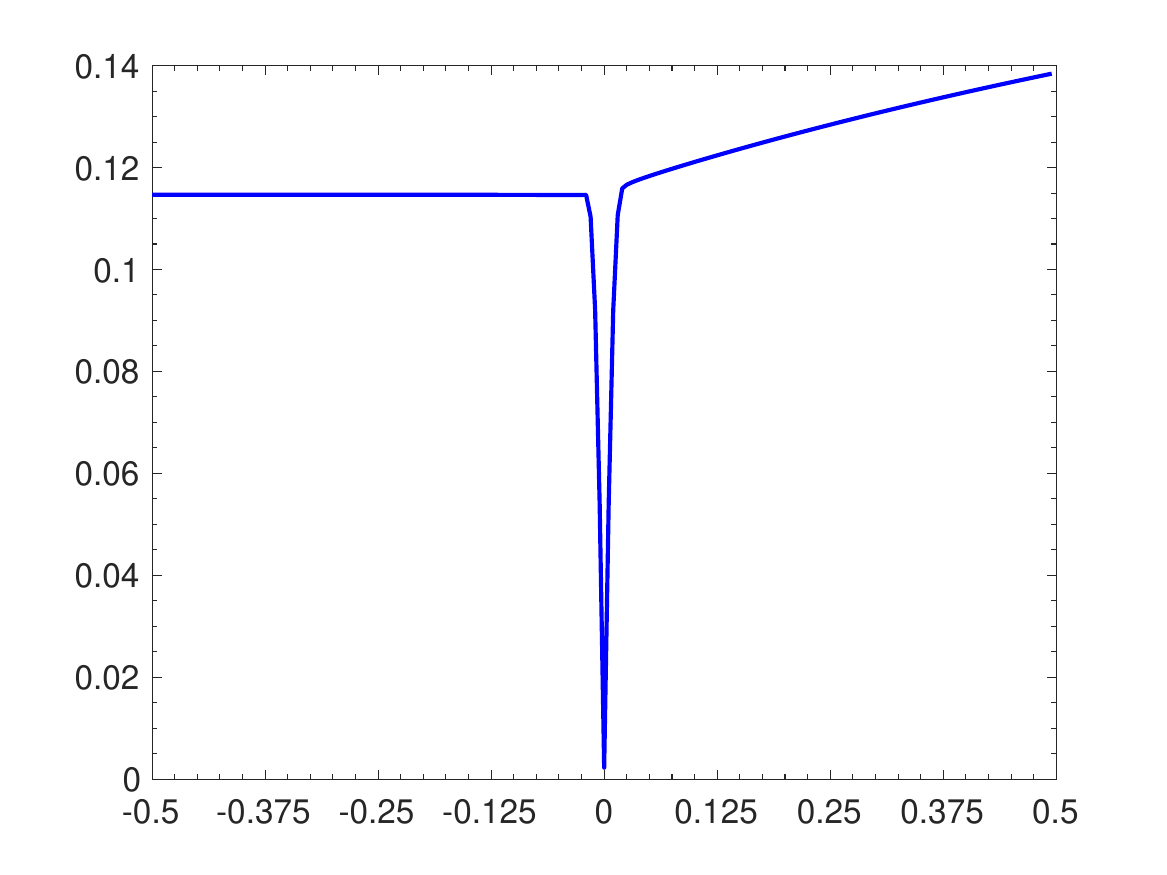}} 
			 \\
			 \subfigure[$(5,5)$, $E^{\red}$ \label{fig:1e_sf}]{\includegraphics[width = .47\textwidth]{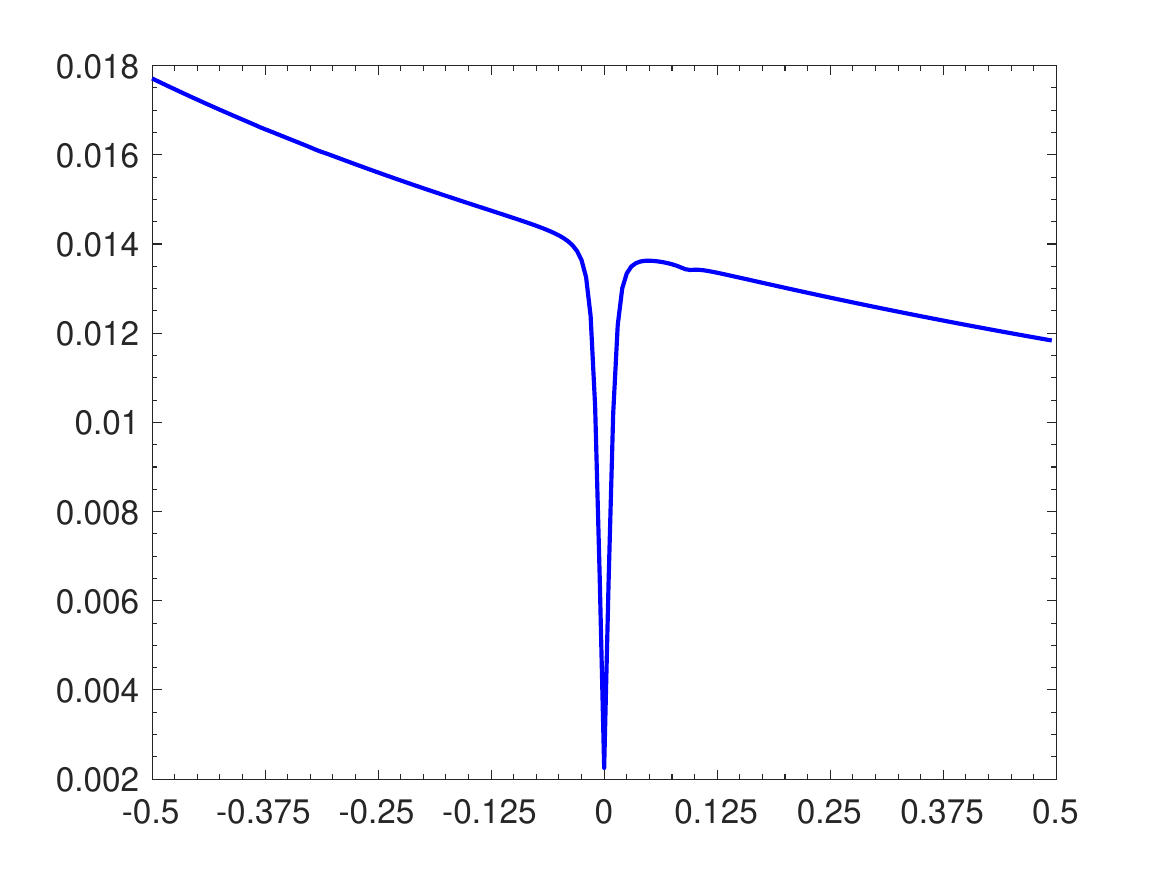}} &			 
			 \subfigure[$(8,3)$, $B^{\red}$ \label{fig:1f_sf}]{\includegraphics[width = .47\textwidth]{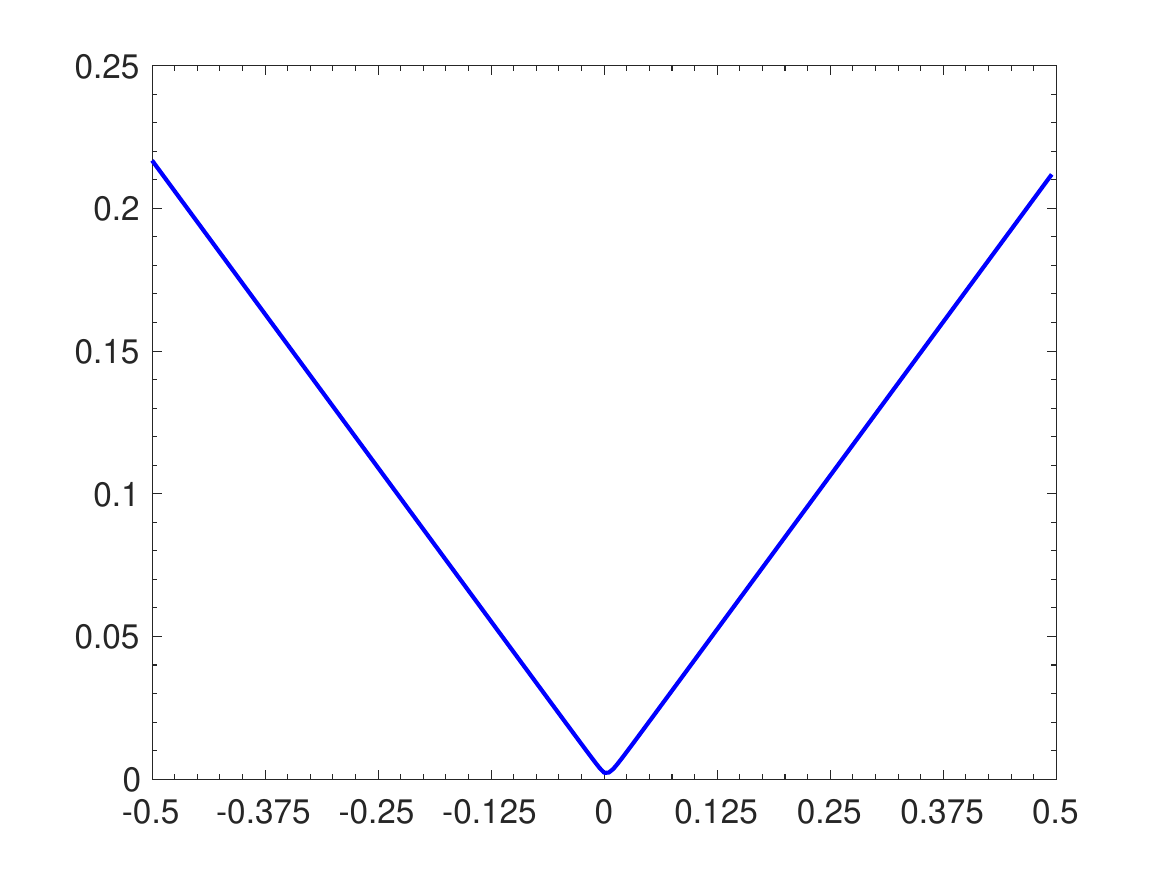}} 
			 \\
			 \subfigure[$(1,5)$, $C^{\red}$ \label{fig:1g_sf}]{\includegraphics[width = .47\textwidth]{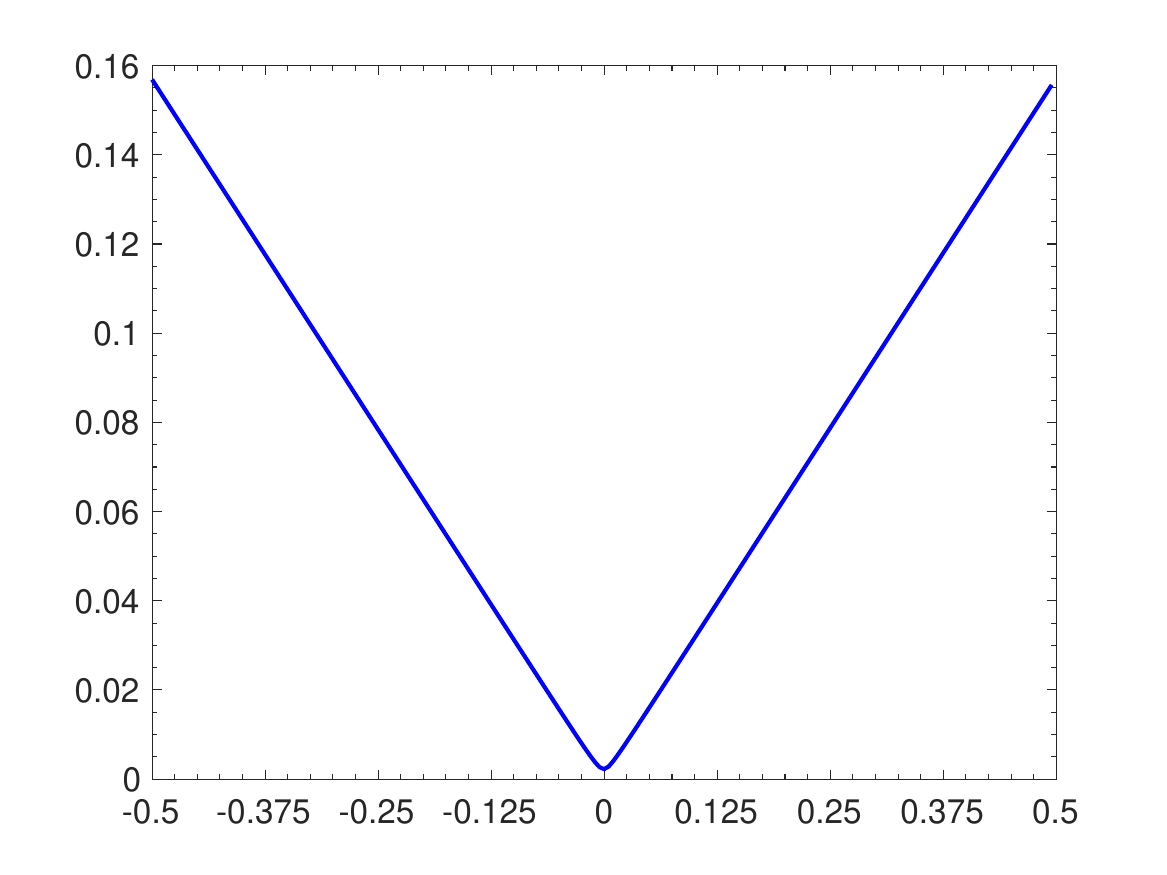}} &
			 \subfigure[$(3,3)$, $D^{\red}$ \label{fig:1h_sf}]{\includegraphics[width = .47\textwidth]{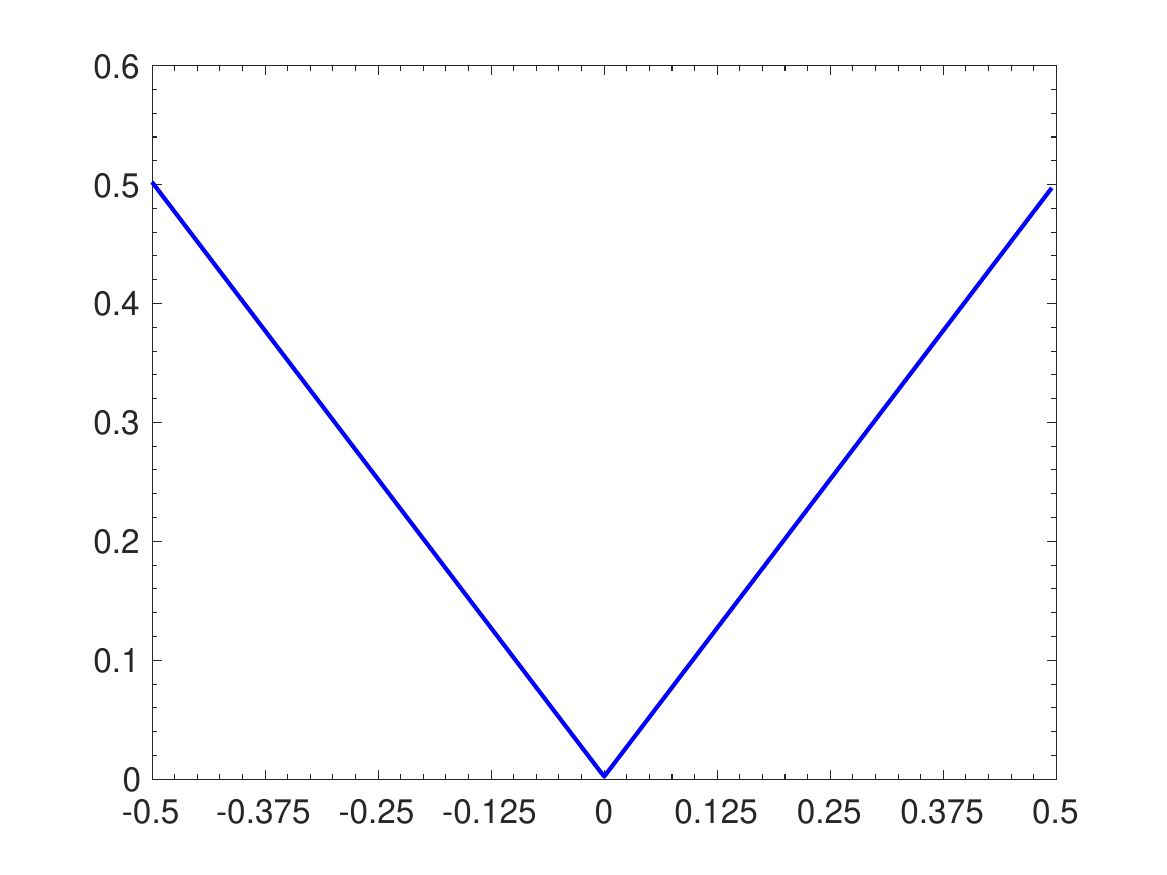}}		
		\end{tabular}   
		\caption{  
		The figure is similar to Figure \ref{fig:direct_gd_converge} and concerns the \texttt{iss} example with $\sr = 12$. 
		Only now the minimization of ${\mathcal F}$ 
		is performed using the subspace framework outlined in Algorithm \ref{alg:SM}. Specifically, each plot
		depicts ${\mathcal F}$ as a function of the variation of one of the entries of one of
		$A^{\red}$, $B^{\red}$, $C^{\red}$, $D^{\red}$, $E^{\red}$. Zero variation corresponds to the
		optimal reduced system by Algorithm \ref{alg:SM}.}
		\label{fig:sf_gd_converge}
\end{figure}

The largest singular values of the errors $\sigma_{\max}( H({\rm i} \omega) - H({\rm i} \omega ; S^{\red}_0))$ 
and $\sigma_{\max}( H({\rm i} \omega) - H({\rm i} \omega ; S^{\red}_\star))$ of the initial
estimate $S^\red_0$ and the optimal estimate $S^\red_\star$ are plotted
as functions of $\omega$ in Figure \ref{fig:iss_error}. The singular value error function
$\sigma_{\max}( H({\rm i} \omega) - H({\rm i} \omega ; S^{\red}_\star))$ for the optimal
$S^\red_\star$ is extremely flat, as indeed
$\sigma_{\max}( H({\rm i} \omega) - H({\rm i} \omega ; S^{\red}_\star)) \in
    [2.02 \cdot 10^{-3}  \, ,  \, 2.26 \cdot 10^{-3}]$ at all $\omega$. Furthermore,
the error $\sigma_{\max}( H({\rm i} \omega) - H({\rm i} \omega ; S^{\red}_\star))$
is maximized at four distinct points marked by the circles on the right-hand plot. 
This indicates that the objective ${\mathcal F}$
is not differentiable at the computed optimizer $S^\red_\star$.

\begin{figure}
\begin{tabular}{cc}
	\hskip -3ex
	\includegraphics[width = .53\textwidth]{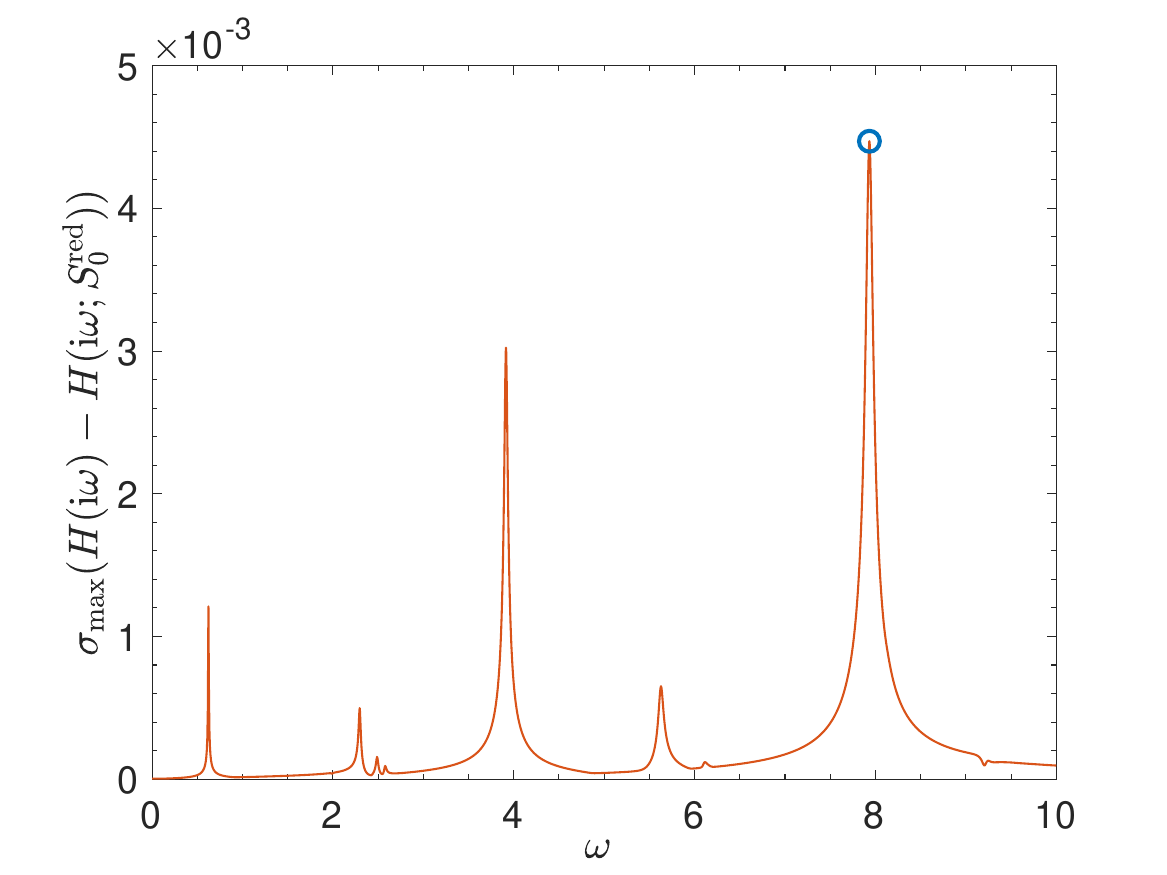}
				&
	\hskip -2ex
	\includegraphics[width = .53\textwidth]{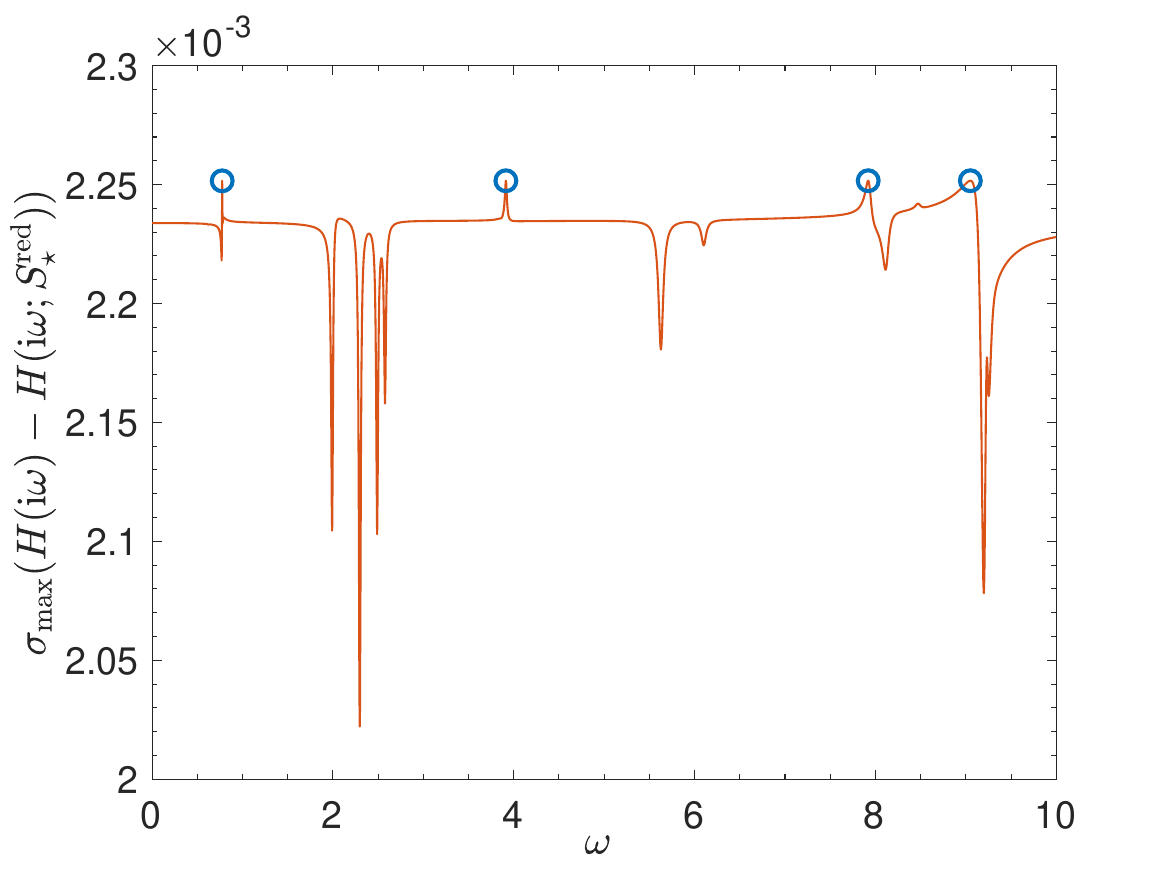}
\end{tabular}
\caption{ The plots of $\sigma_{\max}( H({\rm i} \omega) - H({\rm i} \omega ; S^{\red}_0))$ (left)
and $\sigma_{\max}( H({\rm i} \omega) - H({\rm i} \omega ; S^{\red}_\star))$ (right)
as functions of $\omega$ for the \texttt{iss} example with $\sr = 12$,
where $S^{\red}_0$ is the initial estimate, and $S^{\red}_\star$ is the optimal estimate computed
by Algorithm \ref{alg:SM}. In each plot, the circles mark the points where the largest singular value
function attains the largest value.}
\label{fig:iss_error}
\end{figure}

Information about the progress of Algorithm \ref{alg:SM} is given in Table \ref{table:iss_iterates}.
We start with the reduced system $S_0$ of order 36 that interpolates the original system $S$
of order 270 at three points on the imaginary axis, namely the imaginary parts of the
most dominant three poles of $S$. At every iteration, if no refinement step is performed,
the order of the reduced system $S_r$ increases by $4m = 12$. Additionally, each refinement
step results in an increase of $4m = 12$ in the order of $S_r$. We observe in the
first column that the error ${\mathcal F}(S^\red_r)$ at the minimizer $S^\red_r$ of the
reduced objective ${\mathcal F}_r$ decays rapidly with respect to $r$.  Total number
of objective function evaluations is 492 (i.e., the sum of the function evaluations in the
fifth columns), however the ${\mathcal L}_\infty$ objective to be minimized involves the reduced
system $S_r$ rather than the full system $S$. For instance, the number of
${\mathcal L}_\infty$-norm computations performed are 105, 155, 123 at
iterations $r = 1, 2, 3$. Yet, these ${\mathcal L}_\infty$-norm computations
involve the reduced system $S_r$ of order 72, 84, 120 for $r = 1, 2, 3$.
Observe that the number of bfgs iterations eventually decrease at the later iterations,
as the computed optimal $S^\red_r$ used as the initial estimate
when minimizing ${\mathcal F}_{r+1}$ becomes stationary, i.e., as the computed
minimizer $S^\red_r$ of ${\mathcal F}_r$ is also close to a minimizer of ${\mathcal F}_{r+1}$. 
Refinement steps are needed only at the initial iteration when $r = 0$ and when $r = 2$.
No refinement step turns out to be necessary at the later iterations. This is a generic pattern
which we observe in vast majority of examples we have experimented on.

\begin{table}
\begin{center}
\begin{tabular}{c|ccccc}
	$r$	&	${\mathcal F}(S^\red_r)$	&	$\;$ red order $\;$	&	$\; \#$ bfgs iter $\;$	&	$\; \#$ fun evals $\;$ 	& $\; \#$ refine $\;$ 	\\[.2em]
\hline
0		&	0.004470060020		&	36		&		 ---		&	 ---		&		2	\\
1		&	0.003517059977		&	72		&		38		&	105		&		0	\\
2		&	0.002259657400		&	84		&		45		&	155		&		2	\\
3		&	0.002252138011		&	120		&		35		&	123		&		0	\\
4		&	0.002251613679		&	132		&		11		&	48		&		0	\\
5		&	0.002251609387		&	144		&		2		&	29		&		0	\\
6		&	0.002251607779		&	156		&		2		&	32		&		---	\\[.2em]	
\end{tabular}
\end{center}
\caption{ The iterates and information about the progress of Algorithm \ref{alg:SM} on
the \texttt{iss} example with $\sr = 12$. The columns of red order, $\#$ bfgs iter, $\#$ fun evals, and $\#$
refine list the order of the system $S_r$, number of bfgs iterations, number of objective
function evaluations performed by bfgs, and number of refinement steps performed at
the $r$th iteration.}
\label{table:iss_iterates}
\end{table}

\subsection{CD Player Model}\label{sec:num_exp2}
Our next example is the CD player model which is available in the SLICOT library.
The model is a linear-time invariant system of order $n = 120$ and with $m = 2$
inputs and $p = 2$ outputs. The details of the model can be found in \cite{ChaVD2005}, 
and the references therein. Our primary purpose here is to compare on this example
Algorithm \ref{alg:SM} with the approach in \cite{FlaBG2013} for ${\mathcal H}_\infty$ model reduction
based on rank-one modifications of the system matrices. As the approach in \cite{FlaBG2013}
is for SISO systems, the results are reported over there for this example but with
only the second input and the first output. We follow the same practice here when
applying our approach. The initial estimate $S^{\red}_0$ for a minimizer for Algorithm \ref{alg:SM} 
is constructed  using the balanced truncation. Moreover, the initial reduced system $S_0$
is of order 12, and is constructed so that it interpolates the full system $S$ at the imaginary
parts of its most dominant three poles.

The reduced systems $S^{\red}_\star$ of order $\sr = 2, 4, 6, 8, 10$ are computed using 
Algorithm \ref{alg:SM}. Table \ref{table:cdplayer_rel_errors} lists the relative errors
$\| H - H( \cdot \,\, ; \, S^{\red}_\star) \|_{{\mathcal L}_\infty} / \| H \|_{{\mathcal L}_\infty}$
for the reduced system $S^\red_\star$ computed by various approaches. In particular,
the columns of IHA, MBT, HNA are the reported results in \cite[Table 4]{FlaBG2013}
by using the approach over there initialized with the model returned by IRKA, initialized
with the model returned by the balanced truncation, and the best Hankel norm approximation.
Moreover, the columns of BT and Lower Bnd correspond to the relative error of
the reduced model by the balanced truncation, and the theoretical lower bound
$\sigma_{\sr + 1} /  \|  H \|_{{\mathcal L}_\infty}$ for any reduced system of order $\sr$
for the relative error, where $\sigma_{\sr + 1}$ is the $\sr+1$th largest Hankel singular
value of the system. As can be seen in Table \ref{table:cdplayer_rel_errors}, our approach
produces reduced systems with smaller errors compared to those produced by other
approaches in all cases. The reduced systems produced by Algorithm \ref{alg:SM}
does not seem far away from global optimality either, as their errors are slightly
greater than the theoretical lower bounds in terms of the Hankel singular values in the
last column.

We give some details of Algorithm \ref{alg:SM} applied to find a reduced system of order
$\sr = 8$ in Figures \ref{fig:cdplayer_sf_gd_converge} and \ref{fig:cdplayer_error}, as well as
in Table \ref{table:cdplayer_iterates}. In particular, Figure \ref{fig:cdplayer_sf_gd_converge}
confirms that the reduced system $S^\red_\star$ by Algorithm \ref{alg:SM} is locally optimal,
i.e., small variations in the entries of the system matrices cause increase in the ${\mathcal L}_\infty$
error objective. Figure \ref{fig:cdplayer_error} displays the error 
$\sigma_{\max}( H({\rm i} \omega) - H({\rm i} \omega ; S^{\red}_0))$ of the initial model,
and the error $\sigma_{\max}( H({\rm i} \omega) - H({\rm i} \omega ; S^{\red}_\star))$
of the model by Algorithm \ref{alg:SM} as functions of $\omega$. Once again the error function
for the optimal model $S^\red_\star$ is flatter, even if it is not as pronounced as for the
\texttt{iss} example, compared to that for the initial model $S^\red_0$. The error function
$\sigma_{\max}( H({\rm i} \omega) - H({\rm i} \omega ; S^{\red}_\star))$ for the optimal
model attains its maximum at five different $\omega$ values, which implies that the objective
${\mathcal F}$ is not smooth at $S^\red_\star$. As displayed in Table  \ref{table:cdplayer_iterates},
the convergence occurs again quite rapidly; indeed four iterations are sufficient to reach
prescribed accuracy and terminate. At each iteration, the order of the reduced system 
increases by $4m = 4$. Additionally, the refinement step performed in the initial iteration
causes also an increase of $4m = 4$ in the order of the reduced system. Larger number 
of bfgs iterations are needed at iterations with $r = 1, 2$, when
the objective involves reduced systems of order $20$, $24$, respectively. The total
runtime is around 15 seconds, and the relative error at termination is 
$( {\mathcal F}(S^\red_\star) := \| H - H( \cdot \,\, ; \, S^{\red}_\star) \|_{{\mathcal L}_\infty}) /  \| H \|_{{\mathcal L}_\infty}
  =  ( 2.87\times 10^{-1} ) / ( 6.87 \times 10^1 ) = 4.18 \times 10^{-3}$.


\begin{table}
\begin{center}
\begin{tabular}{c|cccccc}
$\sr$	&	Alg. \ref{alg:SM}	&	IHA	&	MBT 	&	HNA	 	&  BT	&	Lower Bnd	\\[.2em]
\hline
2		&	$3.12\times 10^{-1}$		&	$3.66 \times 10^{-1}$		&		$3.68\times 10^{-1}$		&	 $3.35\times 10^{-1}$	&	$3.69\times 10^{-1}$		&  $1.95\times 10^{-1}$	\\
4		&	$1.82\times 10^{-2}$		&	$2.14 \times 10^{-2}$ 		&		$2.25\times 10^{-2}$		&	$2.00\times 10^{-2}$		&	$2.25\times 10^{-2}$		&  $1.13\times 10^{-2}$	\\
6		&	$9.44\times 10^{-3}$		&	$1.04\times 10^{-2}$		&		$1.19\times 10^{-2}$		&	$1.23\times 10^{-2}$		&	$1.23\times 10^{-2}$		&  $6.79\times 10^{-3}$	\\
8		&	$4.18\times 10^{-3}$		&	$4.85\times 10^{-3}$		&		$6.40\times 10^{-3}$		&	$5.99\times 10^{-3}$		&	$6.41\times 10^{-3}$			&  $3.20\times 10^{-3}$	\\
10		&	$7.45\times 10^{-4}$		&	$8.99\times 10^{-4}$		&		$1.24\times 10^{-3}$		&	$1.08\times 10^{-3}$		&	$1.32\times 10^{-3}$		&  $5.86\times 10^{-4}$ \\[.2em]
\end{tabular}
\end{center}
\caption{ This table concerns the ``cd player model''.
Relative errors $\| H - H( \cdot \,\, ; \, S^{\red}_\star) \|_{{\mathcal L}_\infty} / \| H \|_{{\mathcal L}_\infty}$ are listed for the optimal
estimate $S^{\red}_\star$ computed by various methods for finding reduced systems of order $\sr = 2, 4, 6, 8, 10$, as well
as the lower bound $\sigma_{\sr +1} / \|  H \|_{{\mathcal L}_\infty}$.
}
\label{table:cdplayer_rel_errors}
\end{table}

\begin{figure} 
 \centering
		\begin{tabular}{cc}
			 \subfigure[$(2,2)$, $A^{\red}$ \label{fig:cd_1a_sf}]{\includegraphics[width = .47\textwidth]{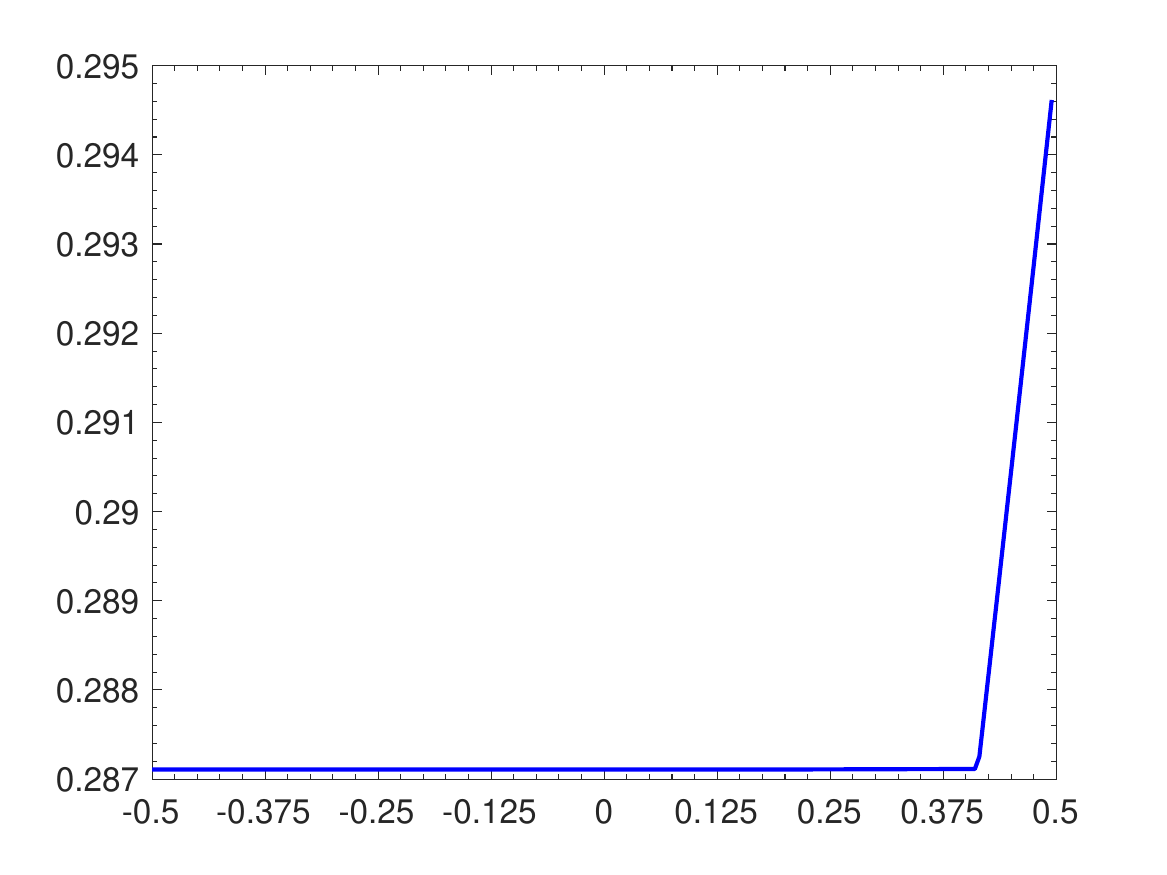}} & 		 
			 \subfigure[$(4,3)$, $A^{\red}$ \label{fig:cd_1b_sf}]{\includegraphics[width = .47\textwidth]{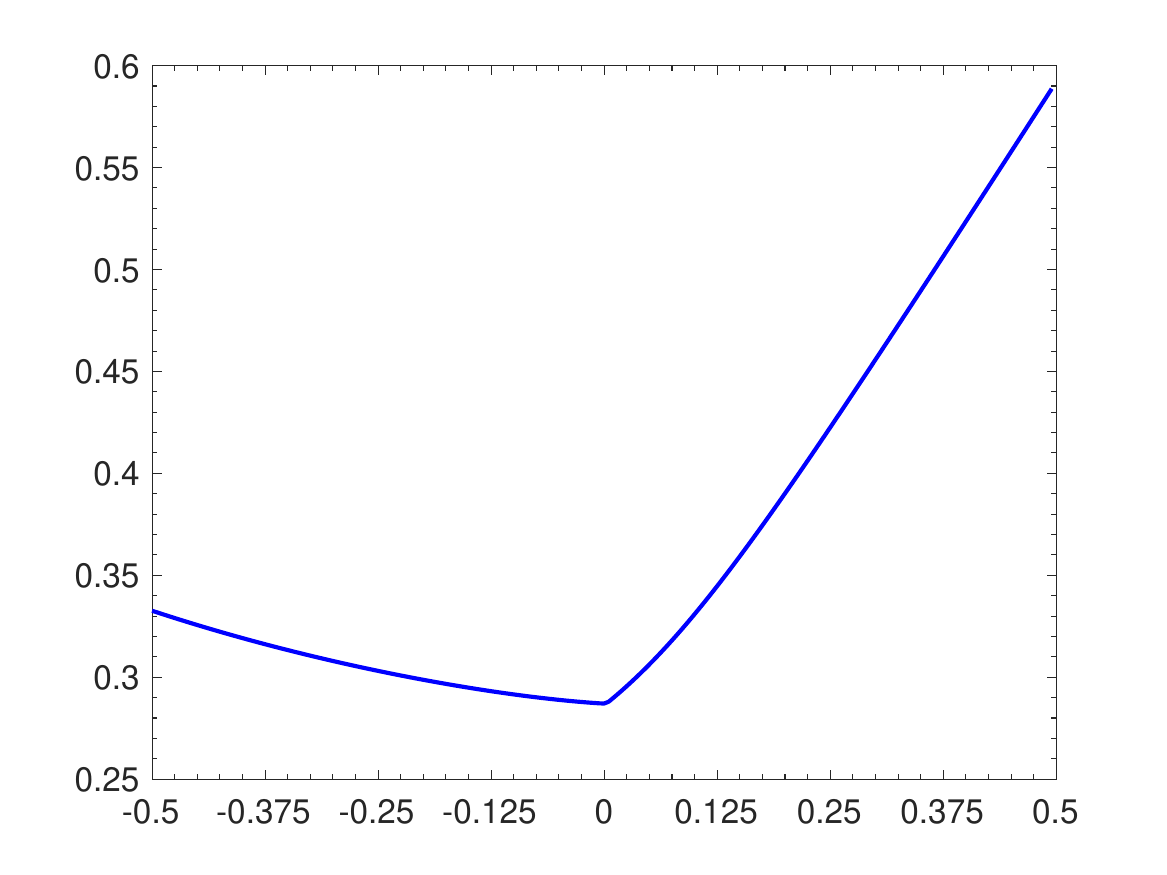}}
			 \\  
			 \subfigure[$(2,3)$, $A^{\red}$ \label{fig:cd_1c_sf}]{\includegraphics[width = .47\textwidth]{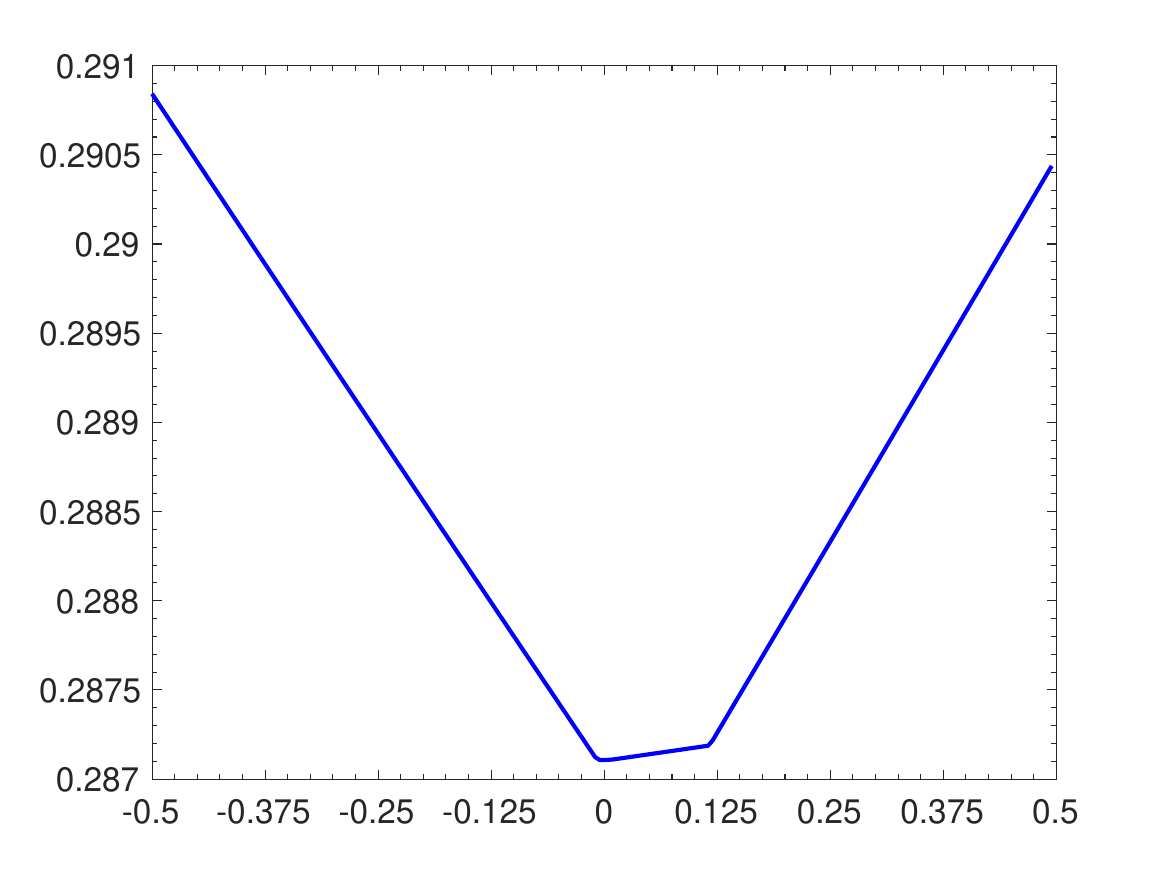}} &			 
			 \subfigure[$(1,1)$, $E^{\red}$ \label{fig:cd_1d_sf}]{\includegraphics[width = .47\textwidth]{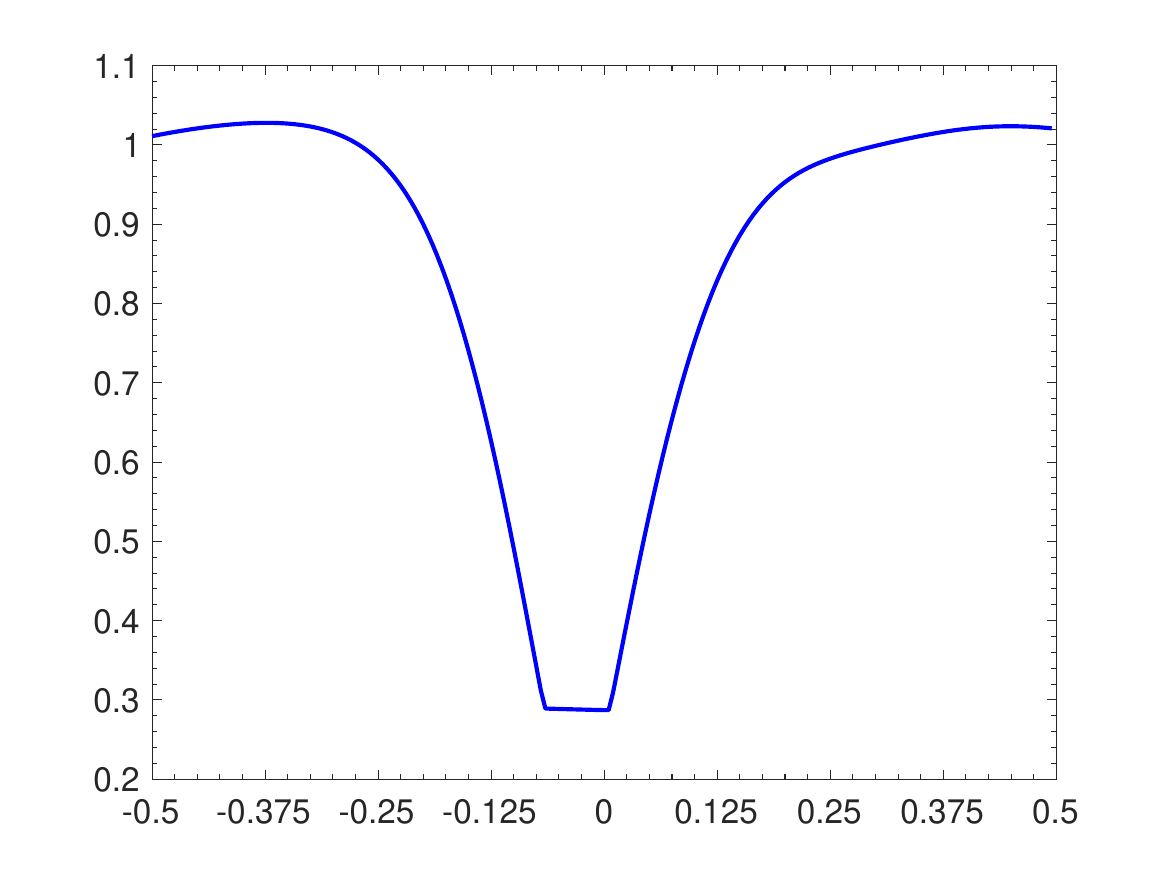}} 
			 \\
			 \subfigure[$(8,8)$, $E^{\red}$ \label{fig:cd_1e_sf}]{\includegraphics[width = .47\textwidth]{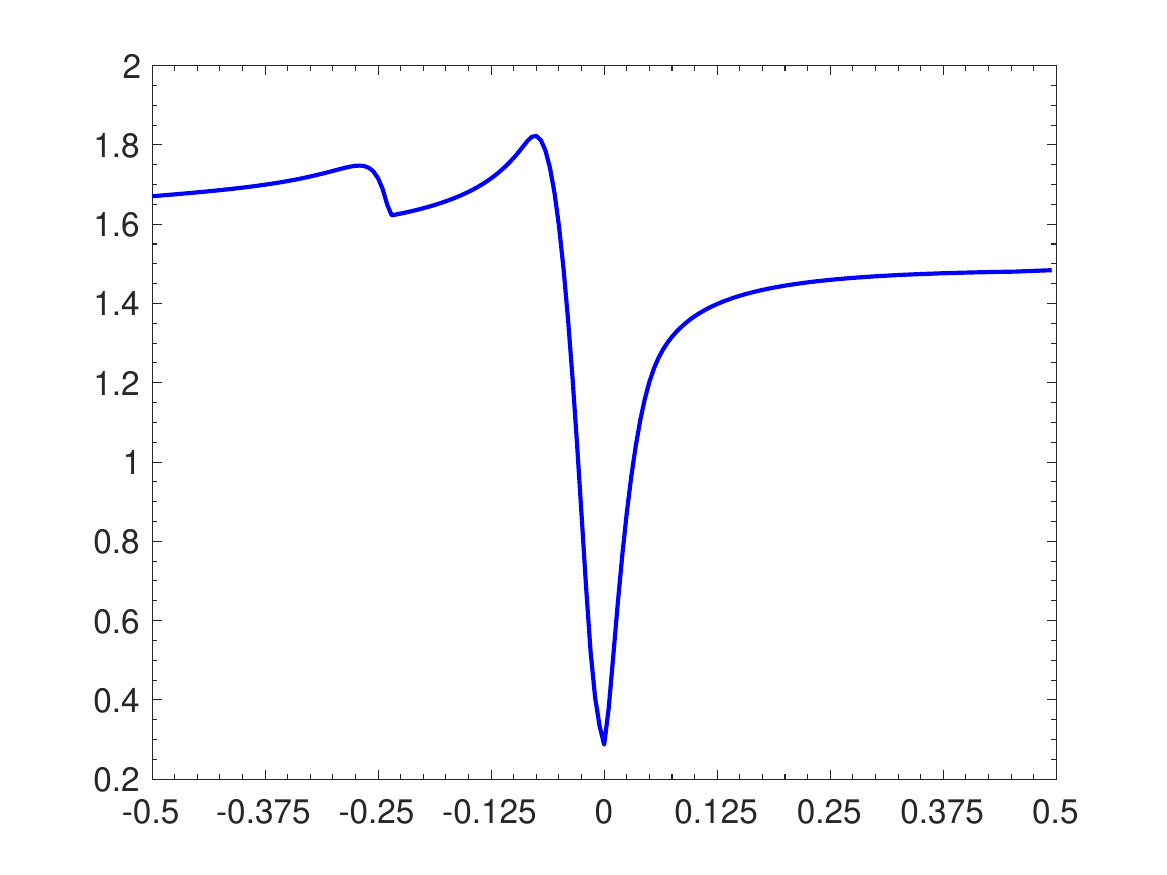}} &			 
			 \subfigure[$(8,1)$, $B^{\red}$ \label{fig:cd_1f_sf}]{\includegraphics[width = .47\textwidth]{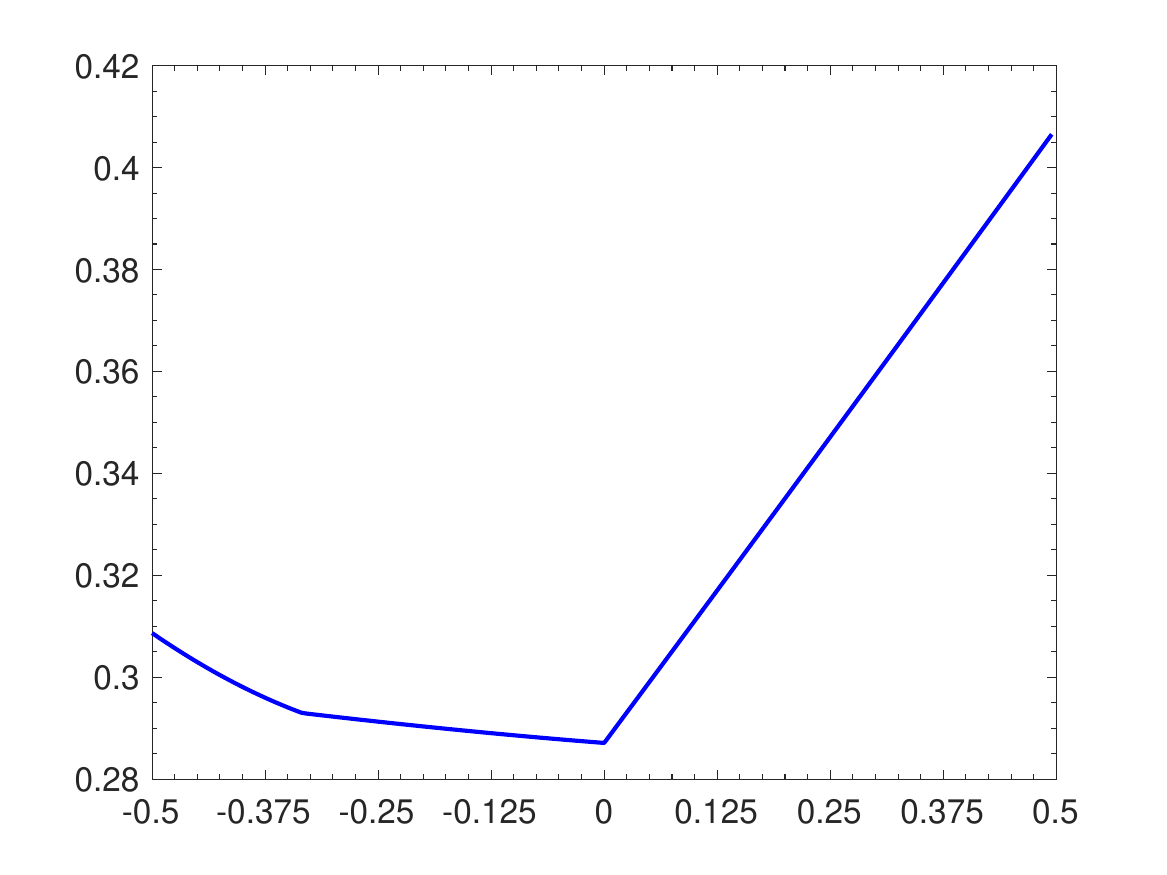}} 
			 \\
			 \subfigure[$(1,3)$, $C^{\red}$ \label{fig:cd_1g_sf}]{\includegraphics[width = .47\textwidth]{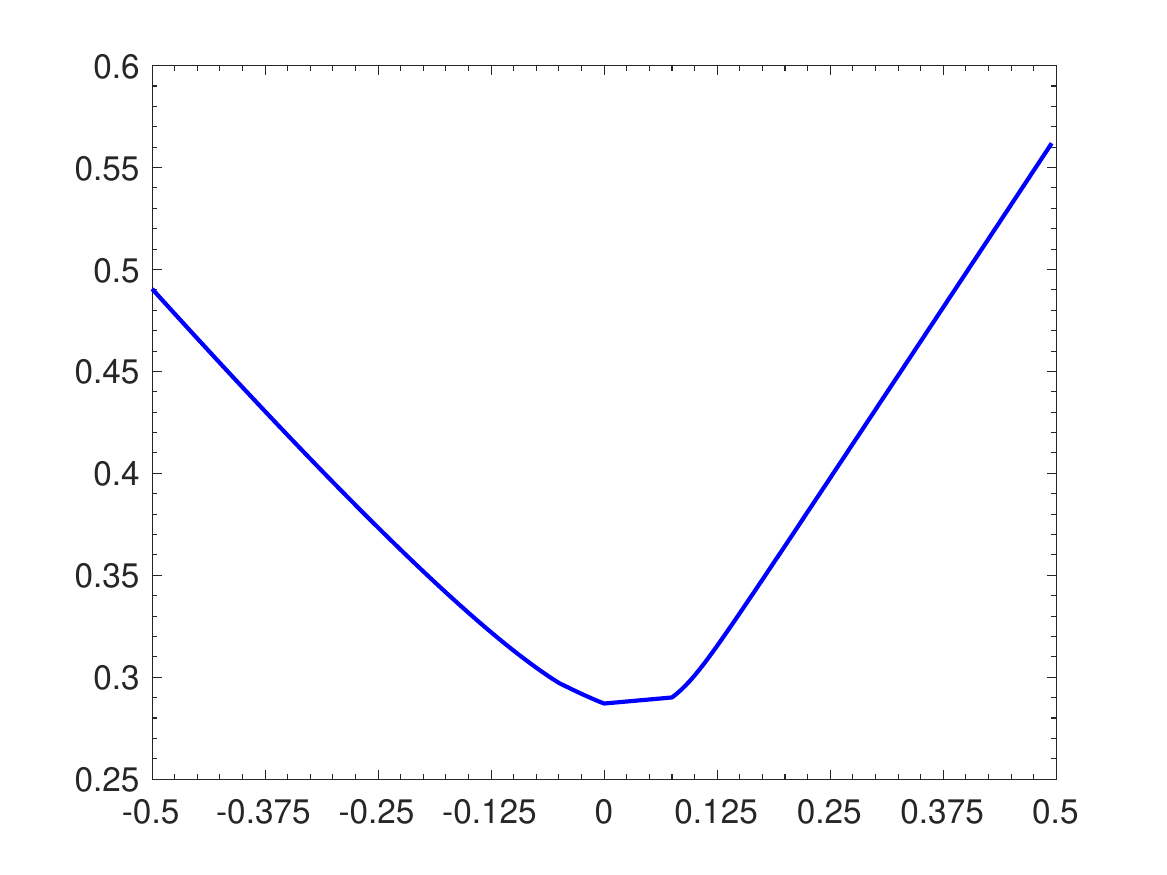}} &
			 \subfigure[$(1,1)$, $D^{\red}$ \label{fig:cd_1h_sf}]{\includegraphics[width = .47\textwidth]{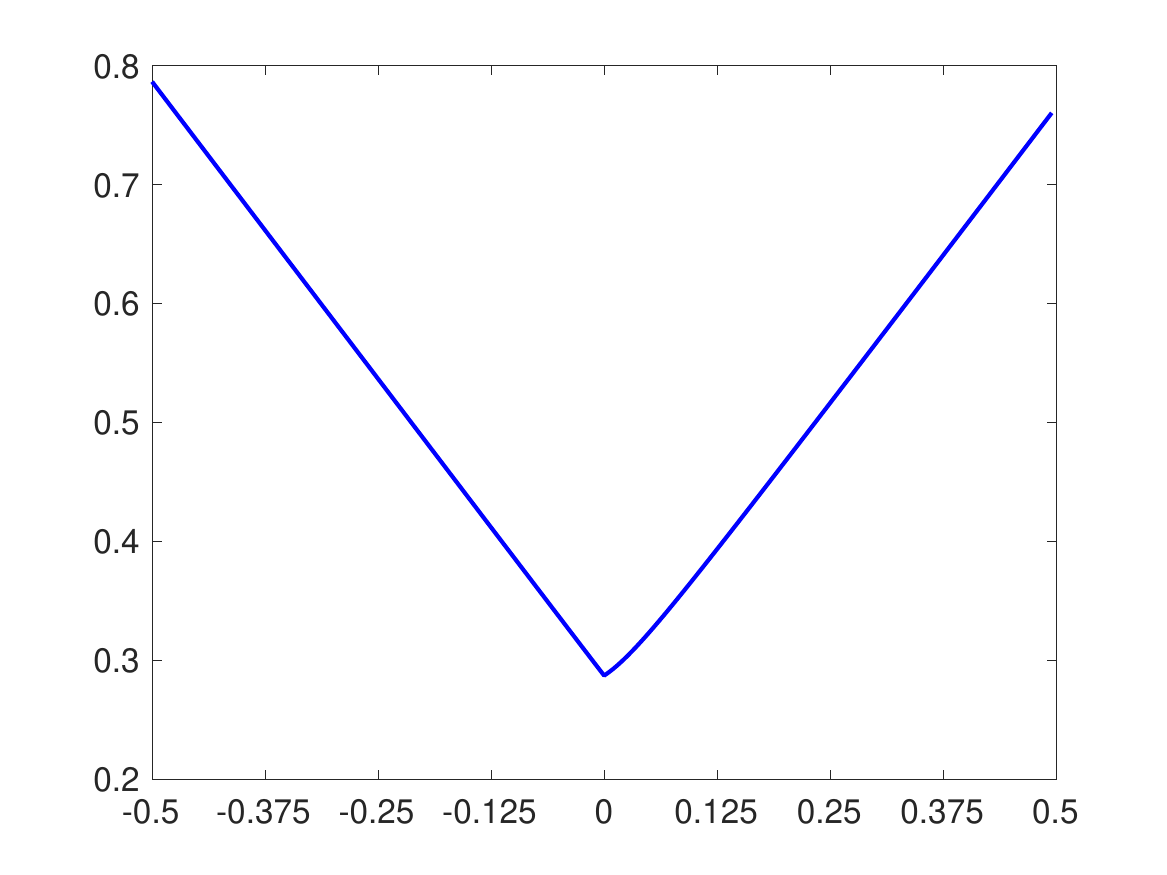}}		
		\end{tabular}   
		\caption{  
		The figure is analogous to Figure \ref{fig:direct_gd_converge}, but concerns the ``cd player model''
		with $\sr = 8$. 
		Each plot
		depicts ${\mathcal F}$ as a function of the variation of one of the entries of one of
		$A^{\red}$, $B^{\red}$, $C^{\red}$, $D^{\red}$, $E^{\red}$. Zero variation corresponds to the
		optimal reduced system by Algorithm \ref{alg:SM}.}
		\label{fig:cdplayer_sf_gd_converge}
\end{figure}

\begin{figure}
\begin{tabular}{cc}
	\hskip -3ex
	\includegraphics[width = .53\textwidth]{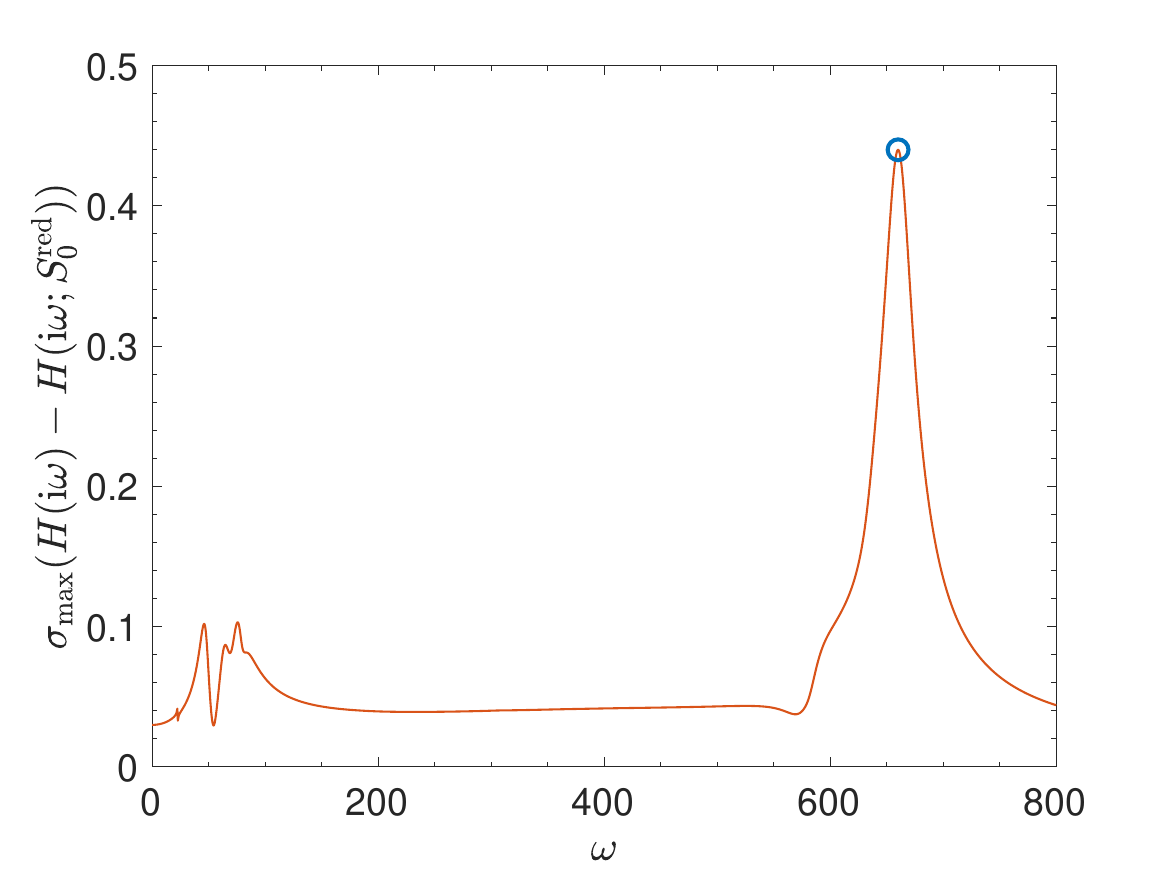}
				&
	\hskip -2ex
	\includegraphics[width = .53\textwidth]{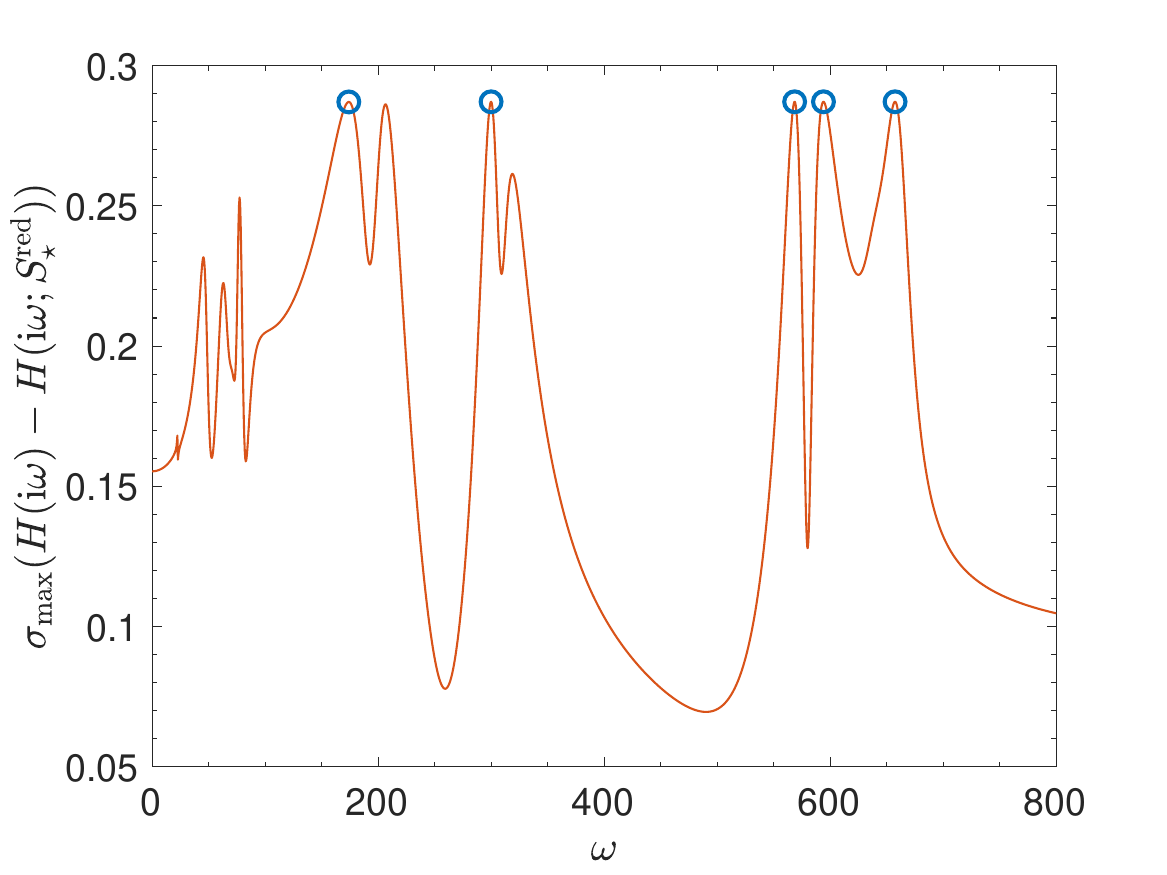}
\end{tabular}
\caption{The plots illustrate the errors of the initial, optimal models by Algorithm \ref{alg:SM}
for the ``cd player model'' with $\sr = 8$, and are analogous to those in Figure \ref{fig:iss_error}.}
\label{fig:cdplayer_error}
\end{figure}

\begin{table}
\begin{center}
\begin{tabular}{c|ccccc}
	$r$	&	${\mathcal F}(S^\red_r)$	&	$\;$ red order $\;$	&	$\; \#$ bfgs iter $\;$	&	$\; \#$ fun evals $\;$ 	& $\; \#$ refine $\;$ 	\\[.2em]
\hline
0		&	0.439972058849		&	12		&		 ---		&	 ---		&		1	\\
1		&	0.291281639337		&	20		&		565		&	1479		&		0	\\
2		&	0.287107598817		&	24		&		134		&	387		&		0	\\
3		&	0.287107598817		&	28		&		1		&	32		&		---	\\[.2em]	
\end{tabular}
\end{center}
\caption{ The iterates and information about the progress of Algorithm \ref{alg:SM} on
the ``cd player model'' for finding a reduced system or order $\sr = 8$. 
The columns represent quantities as in Table \ref{table:iss_iterates}.}
\label{table:cdplayer_iterates}
\end{table}

\subsection{FOM Model}\label{sec:num_exp3}
We next report numerical results on the FOM example available in
the SLICOT library. The FOM example is a linear time-invariant system
of order $n = 1006$, and with $m = p = 1$. The details are given in \cite[Example 3]{Pen2006}.
Here, we are mainly interested in investigating the quality of the estimates
for optimal reduced systems produced by Algorithm \ref{alg:SM}. To this
end, we compare the errors of the reduced systems by Algorithm \ref{alg:SM}
with those of the balanced truncation, as well as the theoretical lower bounds
for the errors in terms of Hankel singular values for varying choices of prescribed
order $\sr$ of the reduced system sought. As in \S\ref{sec:num_exp1} and \S\ref{sec:num_exp2}, we set
the initial estimate $S^{\red}_0$ for a minimizer as the system produced
by the balanced truncation, and the initial reduced system $S^0$ is always 
of order 12 and interpolates the full system $S$ at the imaginary parts of its
most dominant three poles.

In Figure \ref{fig:fom_bt_vs_opt}, the ${\mathcal L}_\infty$ error
 $\| H - H( \cdot ; S^{\red}_\star) \|_{{\mathcal L}_\infty}$ of the optimal 
reduced system $S^{\red}_\star$ by Algorithm \ref{alg:SM} and the balanced
truncation are plotted as functions of the prescribed order $\sr$ of the reduced
system sought. Included in the figure is also the plot of the Hankel singular 
value $\sigma_{\sr + 1}$, a theoretical lower bound for the ${\mathcal L}_\infty$
error  $\| H - H( \cdot ; S^{\red}) \|_{{\mathcal L}_\infty}$ of any system $S^{\red}$
of order $\sr$. Especially when $\sr \in [2,6]$, the errors of the reduced systems 
by Algorithm \ref{alg:SM} are quite close to the theoretical lower bound. 
Indeed, the errors of the reduced systems by Algorithm \ref{alg:SM} usually
differ by the theoretical lower bound by a factor of two at most. Moreover,
in most of cases the errors of reduced systems by Algorithm \ref{alg:SM}
is significantly less than the error of the reduced system by the balanced truncation.

%
%

\begin{figure}
\hskip 12ex
	\includegraphics[width = .8\textwidth]{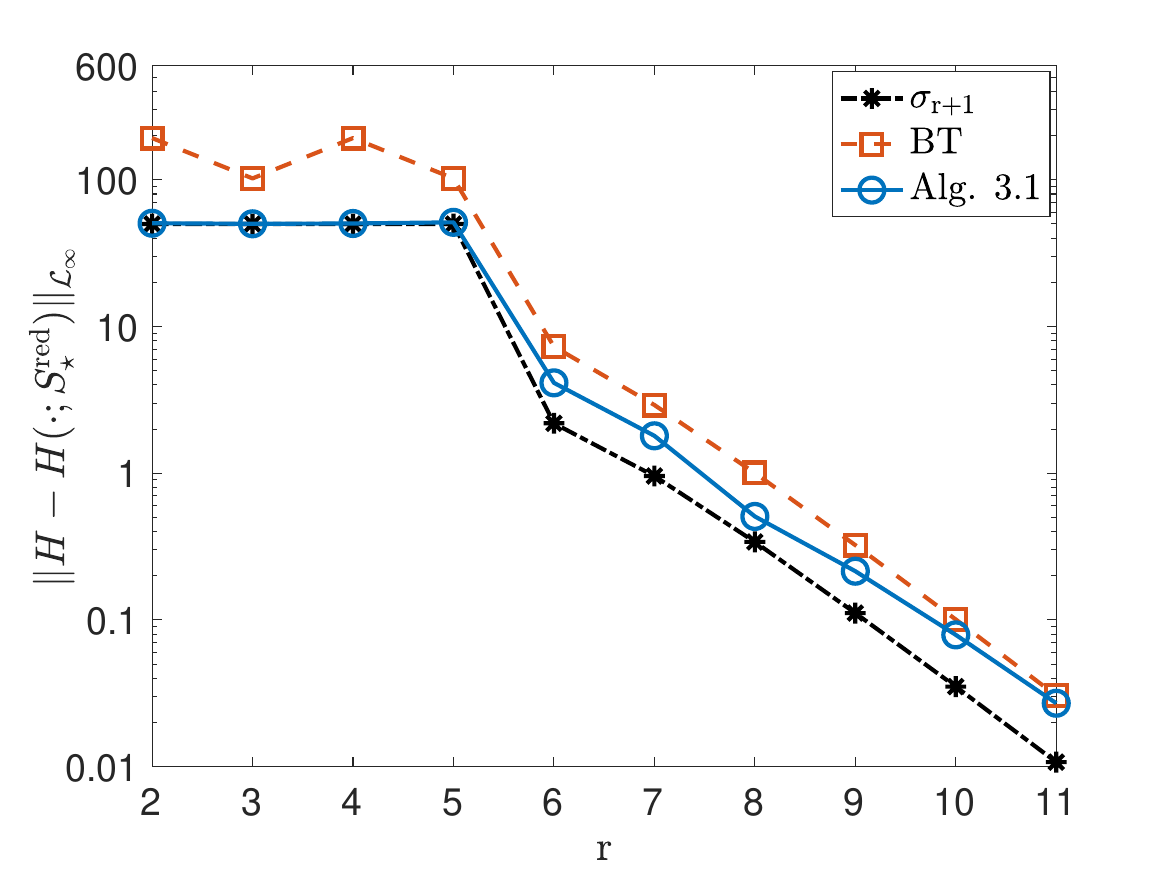}
\caption{ Errors of the reduced systems of order $\sr \in [2, 11]$
produced by Algorithm \ref{alg:SM} and the balanced truncation (BT), as well
as the $(\sr+1)$th largest Hankel singular value $\sigma_{\sr + 1}$ for the FOM example.}
\label{fig:fom_bt_vs_opt}
\end{figure}

\subsection{Systems with Large Order}\label{sec:num_exp4}
Finally, we report results on systems with large order arising from modeling
of power plants due to Rommes and his colleagues. All of these
large-scale examples are available on the website of 
Rommes\footnote{\url{http://sites.google.com/site/rommes/software}}.

Due to the large order of the systems, 
the publicly available implementations of the balanced truncation are 
usually not applicable, and even when they are applicable, they require
substantial amount of computation time. Hence, unlike the previous 
three subsections, we form the initial estimate for the minimizer $S^\red_0$
using the dominant poles of the system quite efficiently. For each
system, we first compute the ten most dominant poles of the system
using the approach in \cite{Men2022}, in particular its 
implementation publicly available at \url{https://zenodo.org/record/5103430}.  
Then $S^\red_0$ of order $\sr$  is constructed so as to interpolate the full system $S$
at the imaginary parts of its $\sr/(4m)$ most dominant poles.
Similarly, the initial reduced system $S_0$ is constructed such that it interpolates $S$
at the imaginary parts of its $\ell$ most dominant poles, where $\ell = 7$
if the system is single-input-single-output (with $m = 1$), and $\ell = 3$ if the system is
multiple-input-multiple-output (with $m > 1$). The order
of the resulting reduced system $S_0$ is $4m \ell$. In all of the examples,
the prescribed order $\sr$ is such that $\sr < 4m\ell$, that is the order of $S_0$
is greater than the prescribed order $\sr$.

Even Algorithm \ref{alg:SM} requires the computation of the ${\mathcal L}_\infty$
norm of systems of order $n + \sr$ a few times (usually not more than 5-6
times in our experiments) in line \ref{large_linfinity}, where $n$ is the large order
of the system. The classical level-set approaches \cite{Boyd1990, Bruinsma1990}
for ${\mathcal L}_\infty$-norm computation and their implementations in Matlab 
are usually no more applicable, or when they are applicable, they take excessive 
amount of time. Instead, we employ the interpolatory subspace framework
in \cite{Aliyev2017} for these large-scale ${\mathcal L}_\infty$-norm computations,
that is for maximizing $\sigma_{\max}( H({\rm i} \omega) - H({\rm i} \omega ; S^{\red}_r)$ over $\omega$ 
at the $r$th iteration. As the approach in \cite{Aliyev2017} is locally convergent, 
whether the initial interpolation points are sufficiently close to global maximizers of
$\sigma_{\max}( H({\rm i} \omega) - H({\rm i} \omega ; S^{\red}_r)$
plays a large role in converging to
a global maximizer. We choose the initial interpolation points as the union of the imaginary parts 
of the ten most dominant poles,
and 15 equally-spaced points on the interval $[-0.1 \, , \, 2 {\mathcal M}]$
with ${\mathcal M}$ denoting the largest of the absolute values of the imaginary part
of the ten most dominant poles.

The absolute and relative errors of the computed reduced systems of order $\sr$
along with the total runtimes are reported in Table \ref{table:large_systems}. For 
systems \texttt{S20PI}$\_$\texttt{n}, \texttt{S40PI}$\_$\texttt{n}, \texttt{M40PI}$\_$\texttt{n}
of order $n = 1182$ or $n = 2182$, we have also computed reduced systems of order $\sr$ 
by means of the balanced truncation. In these examples, the errors of the reduced
systems by Algorithm \ref{alg:SM} are significantly smaller than those of the
reduced systems by the balanced truncation. Moreover, Algorithm \ref{alg:SM}
on these examples require less computation time compared to the balanced truncation.
For systems with larger order, the implementation of the balanced truncation 
that we rely on does not seem suitable; as this implementation is based on
dense linear algebra routines, it cannot cope with such systems. On the other hand,
as evident from Table \ref{table:large_systems}, Algorithm \ref{alg:SM} is also able to deal with
such systems of order ten thousands in a couple of minutes in the worst case.
Most of the runtime of Algorithm \ref{alg:SM} is usually taken by BFGS for 
solving reduced ${\mathcal L}_\infty$-norm minimization problems 
in line \ref{exp_start} involving small systems. In the end, rather than performing
quite a few large-scale ${\mathcal L}_\infty$-norm computations, we
end up performing quite a few small-scale ${\mathcal L}_\infty$-norm computations,
and only a few large-scale ${\mathcal L}_\infty$-norm computations. 
This results is an approach that is not only computationally feasible
but also more reliable, as small-scale ${\mathcal L}_\infty$ norm
computations can be fulfilled accurately, efficiently and reliably 
without worrying about local convergence thanks to the level-set methods \cite{Boyd1990, Bruinsma1990}.

\begin{table}
\begin{center}
\begin{tabular}{|c||c|c|c|llc|}
\hline
	Example 				&	$n, \, m = p$	&   $\sr$  &  approach  & $\;\;\;\,$ error   &  $\:$ rel error  & time    \\	
\hline 
\hline
\texttt{S20PI}$\_$\texttt{n}	&	1182, 1	&   12   & Alg. \ref{alg:SM} &   $7.67\times 10^{-1}$  &  $2.23\times 10^{-1}$  & 19.8 \\ 
\texttt{S20PI}$\_$\texttt{n}	&	1182, 1	&   16   &  Alg. \ref{alg:SM} & $7.66\times 10^{-1}$  &  $2.22\times 10^{-1}$  & 36.2 \\ 
\texttt{S20PI}$\_$\texttt{n}	&	1182, 1	&   12   &  BT &   $1.76\times 10^{0}$  &  $5.11\times 10^{-1}$  & 45.1 \\ 
\texttt{S20PI}$\_$\texttt{n}	&	1182, 1	&   16   &  BT & $1.32\times 10^{0}$  &  $3.84\times 10^{-1}$  & 44.2 \\ 
\hline 
\texttt{S40PI}$\_$\texttt{n}	&	2182, 1	&   12   & Alg. \ref{alg:SM} &  $9.30\times 10^{-1}$  &  $2.78\times 10^{-1}$  & 48.1 \\ 
\texttt{S40PI}$\_$\texttt{n}	&	2182, 1	&   16   & Alg. \ref{alg:SM} &  $6.71\times 10^{-1}$  &  $2.00\times 10^{-1}$  & 38.1 \\ 
\texttt{S40PI}$\_$\texttt{n}	&	2182, 1	&   32   & BT &  $1.75\times 10^{0}$  &  $5.23\times 10^{-1}$  & 410.5 \\ %
\hline 
\texttt{M40PI}$\_$\texttt{n}	&	2182, 3	&   12   & Alg. \ref{alg:SM} &  $1.99\times 10^{0}$  &  $5.22\times 10^{-1}$  & 52.2 \\ 
\texttt{M40PI}$\_$\texttt{n}	&	2182, 3	&   24   & Alg. \ref{alg:SM} &  $1.70 \times 10^{0}$  &  $4.45\times 10^{-1}$  & 117.1 \\ 
\texttt{M40PI}$\_$\texttt{n}	&	2182, 3	&   36   & BT &  $3.07\times 10^0$  &  $8.03\times 10^{-1}$  & 401.9 \\ 
\hline 
\texttt{ww}$\_$\texttt{vref}$\_$\texttt{6405}		&	13251, 1	&   12   & Alg. \ref{alg:SM} &  $5.80\times 10^{-4}$  &  $2.04\times 10^{-1}$  & 9.2 \\ 
 \texttt{ww}$\_$\texttt{vref}$\_$\texttt{6405}		&	13251, 1	&   16   &  Alg. \ref{alg:SM} & $4.19\times 10^{-4}$  &  $1.48\times 10^{-1}$  & 15.1 \\ 
\hline  
\texttt{xingo}$\_$\texttt{afonso}	&	13250, 1	&   12   & Alg. \ref{alg:SM} &  $3.55\times 10^{-2}$  &  $8.74\times 10^{-3}$  & 14.4 \\ 
 \texttt{xingo}$\_$\texttt{afonso}		&	13250, 1	&   16   &  Alg. \ref{alg:SM} & $3.56\times 10^{-2}$  &  $8.77\times 10^{-3}$  & 14.0 \\ 
  \texttt{xingo}$\_$\texttt{afonso}		&	13250, 1	&   20   &  Alg. \ref{alg:SM} & $1.13\times 10^{-2}$  &  $2.79\times 10^{-3}$  & 26.2 \\
\hline 
\texttt{bips07}$\_$\texttt{1998}	&	15066, 4	&   16   & Alg. \ref{alg:SM} &  $1.24\times 10^{1}$  &  $6.30\times 10^{-2}$  & 127.6 \\ 
\texttt{bips07}$\_$\texttt{1998}	&	15066, 4	&   32   & Alg. \ref{alg:SM} &  $9.67\times 10^{0}$  &  $4.91\times 10^{-2}$  & 219.8 \\  
\hline 
\texttt{bips07}$\_$\texttt{3078}	&	21228, 4	&   16   & Alg. \ref{alg:SM} &  $1.27\times 10^{1}$  &  $6.06\times 10^{-2}$  & 200.8 \\ 
\texttt{bips07}$\_$\texttt{3078}	&	21228, 4	&   32   & Alg. \ref{alg:SM} &  $1.00\times 10^{1}$  &  $4.78\times 10^{-2}$  & 274.1 \\ 
\hline 
\end{tabular}
\end{center}
\caption{ The absolute errors $\| H - H(\cdot \, ; \, S^{\red}_\star) \|_{{\mathcal L}_{\infty}}$
and relative errors 
$\| H - H(\cdot \, ; \, S^{\red}_\star) \|_{{\mathcal L}_{\infty}} / 
             \| H \|_{{\mathcal L}_{\infty}}$ for systems of large order, where
$S^{\red}_\star$ is the optimal reduced system by either Algorithm \ref{alg:SM}
or the balanced truncation (BT). Total runtimes in seconds are also listed in the last column. }
\label{table:large_systems}
\end{table}

\section{Software}
A Matlab implementation of Algorithm \ref{alg:SM} is publicly available
at \url{https://zenodo.org/record/8344591}. The numerical results reported
in the previous section are obtained with this implementation.
Scripts are included to reproduce
the results for the CD player model in \S\ref{sec:num_exp2}, 
and the \texttt{xingo}$\_$\texttt{afonso}, \texttt{bips07}$\_$\texttt{1998} examples
in \S\ref{sec:num_exp4}. The results for other
benchmark examples can be obtained similarly.

\section{Conclusion}
We have proposed an approach to find a locally optimal solution of the
${\mathcal L}_\infty$-norm model reduction problem. To our knowledge,
this is the first work on the subject. Our approach is based on the usage
of smooth optimization techniques such as the gradient descent method
and BFGS. A direct application of such smooth optimization techniques
for the ${\mathcal L}_\infty$-norm model reduction problem does not
seem suitable even for systems with modest order, as smooth optimization
techniques converge very slowly and require the evaluation of the
costly ${\mathcal L}_\infty$-norm objective too many times.
Hence, our approach replaces the original system of modest or large
order with a system of small order, and solves the resulting reduced
${\mathcal L}_\infty$-norm minimization problem by means of the
smooth optimization. Then it refines and increases slightly the order
of the reduced system based on the minimizer of this reduced optimization problem.
This refinement is performed with an eye on interpolation between
the full and reduced ${\mathcal L}_\infty$ objectives.
Under smoothness assumptions, admittedly strong in this context, we have 
given formal arguments for the quick convergence of the approach.
We have also described how asymptotic stability constraints on the
small system of prescribed order sought can be incorporated into
the approach. The numerical experiments on a variety of real benchmark
examples indicate that our approach retrieves indeed a locally optimal
solution of the ${\mathcal L}_\infty$-norm model reduction problem in
practice. Moreover, on some small benchmark examples, we have
obtained reduced systems not far away for from being optimal globally according to the 
theoretical lower bounds in terms of Hankel singular values. Experiments
on large benchmark examples illustrate that the approach is usually suitable for  
systems of order a few ten thousands.

The quality of the converged locally optimal solution depends on the initial
guess for the optimal reduced system. To generate the initial guess, we have employed 
two different strategies based on the balanced truncation and dominant poles. The first
of these strategies may not be computationally feasible if the original system
has large order, whereas the second strategy seems suitable even for large
systems. However, a strategy generating a good initial guess is certainly
worth further research. The proposed approach typically requires a few
large-scale ${\mathcal L}_\infty$-norm computations. Performing these ${\mathcal L}_\infty$-norm
computations accurately, especially without getting stagnated at a local maximizer 
that is not optimal globally, is crucial for the reliability of the proposed approach.
We have employed the interpolatory subspace framework in \cite{Aliyev2017} with the
initial interpolation points chosen based on the dominant poles for these large-scale 
${\mathcal L}_\infty$-norm computations. This approach usually seems to work well in practice
for large-scale ${\mathcal L}_\infty$-norm computations. Still, we hope to
explore further a good initial interpolation selection strategy for \cite{Aliyev2017} 
so that it converges globally, leading to the correct ${\mathcal L}_\infty$ norm
with very high probability. Other efficient and accurate candidates for large-scale
${\mathcal L}_\infty$-norm computation are worth studying. In \cite{FlaBG2013},
the original system is replaced by a smaller order system obtained from the Loewner
framework \cite{MayA2007} to reduce the burden of large-scale ${\mathcal L}_\infty$-norm 
computations. We have not attempted here to incorporate the Loewner framework
into our approach. As a future work, our approach can possibly benefit from the Loewner framework;
for instance, the initial reduced system replacing the full system can perhaps be
obtained using the Loewner framework. Our quick convergence result for the
proposed approach is under strong smoothness assumptions. Investigating the
order of convergence of the approach in the likely nonsmooth setting (i.e.,
when the ${\mathcal L}_\infty$ objective at the converged minimizer is nonsmooth)  
is a possible direction for future research. Last but not the least, the convergence
of smooth optimization techniques such as BFGS is more of an empirical phenomenon 
in the current-state-of-art with some intuition as to why. Analyzing the convergence
of smooth optimization techniques in the presence of nonsmoothness at the optimizers
is an important open problem.

%


\bibliography{HinfMOR}

\end{document}